\documentclass{amsart}

\usepackage[T1]{fontenc}
\usepackage{enumerate, amsmath, amsfonts, amssymb, amsthm, mathrsfs, wasysym, graphics, graphicx, xcolor, url, hyperref, hypcap,  shuffle, xargs, multicol, overpic, pdflscape, multirow, hvfloat, minibox, accents, array, multido, xifthen, a4wide, ae, aecompl, blkarray, pifont, mathtools, etoolbox, relsize}
\usepackage{marginnote}
\hypersetup{colorlinks=true, citecolor=darkblue, linkcolor=darkblue}
\usepackage[all]{xy}
\usepackage[bottom]{footmisc}
\usepackage{tikz}
\usetikzlibrary{trees, decorations, decorations.markings, shapes, arrows, matrix, calc, fit, intersections, patterns, angles}
\usepackage[noabbrev,capitalise]{cleveref}
\graphicspath{{figures/}{figures/nodes/}}
\makeatletter\def\input@path{{figures/}}\makeatother
\usepackage{caption}
\usepackage[export]{adjustbox}
\usepackage{figures/picins/picins}

\newtheorem{theorem}{Theorem}
\newtheorem{corollary}[theorem]{Corollary}
\newtheorem{proposition}[theorem]{Proposition}
\newtheorem{lemma}[theorem]{Lemma}

\newtheorem*{theorem*}{Theorem}

\theoremstyle{definition}

\newtheorem{example}[theorem]{Example}
\newtheorem{remark}[theorem]{Remark}

\newcommand{\R}{\mathbb{R}} 
\newcommand{\fS}{\mathfrak{S}} 
\newcommand{\cI}{\mathcal{I}} 
\renewcommand{\b}[1]{{\boldsymbol{#1}}} 

\newcommand{\set}[2]{\left\{ #1 \;\middle|\; #2 \right\}} 
\newcommand{\bigset}[2]{\big\{ #1 \;\big|\; #2 \big\}} 
\newcommand{\Bigset}[2]{\Big\{ #1 \;\Big|\; #2 \Big\}} 
\newcommand{\ssm}{\smallsetminus} 
\newcommand{\dotprod}[2]{\left\langle \, #1 \; \middle| \; #2 \, \right\rangle} 
\newcommand{\one}{\b{1}} 
\newcommand{\eqdef}{\mbox{\,\raisebox{0.2ex}{\scriptsize\ensuremath{\mathrm:}}\ensuremath{=}\,}} 
\newcommand{\defeq}{\mbox{~\ensuremath{=}\raisebox{0.2ex}{\scriptsize\ensuremath{\mathrm:}} }} 

\DeclareMathOperator{\inv}{inv} 

\newcommand{\fref}[1]{Figure~\ref{#1}} 
\newcommand{\ie}{\textit{i.e.}~} 
\newcommand{\eg}{\textit{e.g.}~} 
\newcommand{\viceversa}{\textit{vice versa}} 
\definecolor{darkblue}{rgb}{0,0,0.7} 
\definecolor{green}{RGB}{57,181,74} 
\definecolor{violet}{RGB}{147,39,143} 
\newcommand{\darkblue}{\color{darkblue}} 
\newcommand{\defn}[1]{\textsl{\darkblue #1}} 

\usepackage{todonotes}

\newcommand{\meet}{\wedge} 
\newcommand{\join}{\vee} 

\newcommand{\shard}{\Sigma}
\newcommand{\shards}{\boldsymbol{\Sigma}}

\newcommand{\decoration}{\delta} 
\newcommand{\includeSymbol}[1]{\ensuremath{%
	\mathchoice
		{\raisebox{-.7mm}{\includegraphics[height=2.2ex]{#1}}}	
		{\raisebox{-.7mm}{\includegraphics[height=2.2ex]{#1}}}
		{\raisebox{-.6mm}{\includegraphics[height=1.6ex]{#1}}}
		{\raisebox{-.5mm}{\includegraphics[height=1ex]{#1}}}
}}
\robustify{\includeSymbol}
\newcommand{\noneCirc}{\includeSymbol{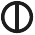}}
\newcommand{\upCirc}{\includeSymbol{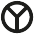}}
\newcommand{\downCirc}{\includeSymbol{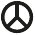}}
\newcommand{\upDownCirc}{\includeSymbol{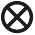}}
\newcommand{\Decorations}{\{\noneCirc{}, \downCirc{}, \upCirc{}, \upDownCirc{}\}} 
\newcommand{\edgecut}[2]{\left( #1 \;\middle\|\; #2 \right)} 
\newcommand{\cle}{\ensuremath{\preccurlyeq}} 

\newcommandx{\ray}[1][1=r]{\boldsymbol{#1}} 
\newcommandx{\rays}[1][1=R]{\boldsymbol{#1}} 
\newcommand{\hyp}{\mathbb{H}} 
\newcommand{\Hyp}{\mathbf{H}} 
\newcommand{\HA}{\mathcal{H}} 

\newcommandx{\Perm}[1][1=n]{\mathsf{Perm}_{#1}} 
\newcommandx{\Asso}[1][1=n]{\mathsf{Asso}_{#1}} 
\newcommandx{\Zono}[1][1=n]{\mathsf{Zono}_{#1}} 
\newcommandx{\Fan}[1][1=F]{\mathcal{#1}} 
\newcommandx{\Permutreehedron}[1][1=\decoration]{\mathsf{PT}_{#1}} 
\newcommandx{\Defo}[1][1=h]{\mathsf{Defo}_{#1}} 
\newcommandx{\Remo}[1][1=\cI]{\mathsf{Remo}_{#1}} 

\newcommand{\typeCone}{\mathbb{TC}} 
\newcommand{\ctypeCone}{\smash{\overline{\mathbb{TC}}}} 
\newcommandx{\coefficient}[3][1={\ray[t]}, 2=\ray, 3={\ray[s]}]{\alpha_{#2,#3}(#1)} 

\makeatletter
\def\l@section{\@tocline{1}{1pt}{0pc}{}{}}
\makeatother
\let\oldtocsection=\tocsection
\renewcommand{\tocsection}[2]{\bf\oldtocsection{#1}{#2}}
\let\oldtocsubsubsection=\tocsubsubsection
\renewcommand{\tocsubsubsection}[2]{\quad\oldtocsubsubsection{#1}{#2}}

\title{\mbox{Removahedral congruences versus permutree congruences}}

\thanks{VP \& JR were supported by the French ANR (grants CAPPS~17\,CE40\,0018 and CHARMS~19\,CE40\,0017)}

\author{Doriann Albertin}
\address[DA]{LIGM, Université Gustave Eiffel, Champs-sur-Marne}
\email{doriann.albertin@u-pem.fr}

\author{Vincent Pilaud}
\address[VP]{CNRS \& LIX, \'Ecole Polytechnique, Palaiseau}
\email{vincent.pilaud@lix.polytechnique.fr}
\urladdr{\url{http://www.lix.polytechnique.fr/~pilaud/}}

\author{Julian Ritter}
\address[JR]{LIX, \'Ecole Polytechnique, Palaiseau}
\email{julian.ritter@polytechnique.edu}
\urladdr{\url{www.nailuj.de}}

\begin{document}

\begin{abstract}
The associahedron is classically constructed as a removahedron, \ie by deleting inequalities in the facet description of the permutahedron. This removahedral construction extends to all permutreehedra (which interpolate between the permutahedron, the associahedron and the cube). Here, we investigate removahedra constructions for all quotientopes (which realize the lattice quotients of the weak order). On the one hand, we observe that the permutree fans are the only quotient fans realized by a removahedron. On the other hand, we show that any permutree fan can be realized by a removahedron constructed from any realization of the braid fan. Our results finally lead to a complete description of the type cones of the permutree fans.

\medskip
\noindent
\textsc{msc classes.} 52B11, 52B12, 03G10, 06B10
\end{abstract}

\vspace*{-.6cm}
\maketitle



\section{Introduction}

This paper deals with particular polytopal realizations of quotient fans of lattice congruences of the weak order.
The prototypes of such polytopes are the classical permutahedron~$\Perm$ realizing the weak order on permutations and the classical associahedron~$\Asso$ realizing the Tamari lattice on binary trees.
These two polytopes belong to the family of permutreehedra realizing the rotation lattice on permutrees, which play a fundamental role in this paper.
Permutrees were introduced in~\cite{PilaudPons-permutrees} to generalize and interpolate between permutations and binary trees, and explain the combinatorial, geometric and algebraic similarities between them.
They were inspired by Cambrian trees~\cite{ChatelPilaud,LangePilaud} which provide a combinatorial model to the type~$A$ Cambrian lattices of~\cite{Reading-CambrianLattices}.
As the classical construction of the associahedron due to~\cite{ShniderSternberg,Loday} and its generalization to all Cambrian associahedra by~\cite{HohlwegLange}, all permutreehedra are obtained by deleting inequalities in the facet description of the permutahedron~$\Perm$.
Such polytopes are called removahedra, and were studied in the context of graph associahedra in~\cite{Pilaud-removahedra}.

In general, any lattice congruence~$\equiv$ of the weak order on~$\fS_n$ defines a quotient fan~$\Fan_\equiv$ obtained by gluing together the chambers of the braid fan corresponding to permutations in the same congruence class~\cite{Reading-HopfAlgebras}.
This quotient fan~$\Fan_\equiv$ was recently proven to be the normal fan of a polytope~$P_\equiv$ called quotientope~\cite{PilaudSantos-quotientopes, PadrolPilaudRitter}.
As their normal fans all refine the braid fan, quotientopes belong to the class of deformed permutahedra studied in~\cite{Postnikov, PostnikovReinerWilliams} (we prefer the name ``deformed permutahedra'' rather than ``generalized permutahedra'' as there are many generalizations of permutahedra).
All deformed permutahedra are obtained from the permutahedron~$\Perm$ by moving facets without ``passing a vertex'' (in the sense of~\cite{Postnikov}).
Observe that not all deformed permutahedra are removahedra, since it is sometimes inevitable to move facets, not only to remove~them.

The construction of~\cite{PilaudPons-permutrees} for permutreehedra and the construction of~\cite{PilaudSantos-quotientopes} for quotientopes seem quite different.
The polytopes resulting from the former construction lie outside the permutahedron~$\Perm$ (outsidahedra) while those resulting from the latter construction lie inside the permutahedron~$\Perm$ (insidahedra).
In fact, it was already observed in~\cite[Rem.~12]{PilaudSantos-quotientopes} that the quotient fans of some lattice congruences cannot be realized by a removahedron.
The first contribution of this paper is to show the following strong dichotomy between the lattices congruences of the weak order regarding realizability by removahedra.

\begin{theorem}
\label{thm:main}
Let~$\equiv$ be a lattice congruence of the weak order on~$\fS_n$. Then
\begin{enumerate}[(i)]
\item if~$\equiv$ is not a permutree congruence, then the quotient fan~$\Fan_\equiv$ is not the normal fan of a removahedron,
\item if~$\equiv$ is a permutree congruence, then the quotient fan~$\Fan_\equiv$ is the normal fan of a polytope obtained by deleting inequalities in the facet description of any polytope realizing the braid fan (not only the classical permutahedron~$\Perm$).
\end{enumerate}
\end{theorem}

This statement is based on the understanding of the inequalities governing the facet heights that ensure to obtain a polytopal realization of the quotient fan.
These inequalities, given by pairs of adjacent cones of the fan and known as wall-crossing inequalities, define the space of all realizations of the fan.
This space of realizations is an open polyhedral cone called type cone and studied by~\cite{McMullen-typeCone}, and its closure is also known as the deformation cone~\cite{Postnikov, PostnikovReinerWilliams}.
For instance, the deformation cone of the permutahedron~$\Perm$ is the space of submodular functions, and corresponds to all deformed permutahedra~\cite{Postnikov}.
The main contribution of this paper is the facet description of the type cone of any permutree fan, thus providing a complete description of all polytopal realizations of the permutree fans.
This description requires three steps: for the $\decoration$-permutree fan~$\Fan_\decoration$ corresponding to a decoration~$\decoration$, we provide purely combinatorial descriptions (only in terms of~$\decoration$) for
\begin{itemize}
\item the subsets of~$[n]$ which correspond to rays of~$\Fan_\decoration$ (\cref{prop:raysPermutreeFan}),
\item the pairs of subsets of~$[n]$ which correspond to exchangeables rays of~$\Fan_\decoration$ (\cref{prop:exchangeablePairsPermutreeFan}),
\item the pairs of subsets of~$[n]$ which correspond to facets of the type cone of~$\Fan_\decoration$ (\cref{prop:extremalExchangeablePairsPermutreeFan}).
\end{itemize}
From this facet description, we derive summation formulas for the number of facets of the type cones of permutree fans, leading to a characterization of the permutree fans whose type cone is simplicial.
As advocated in~\cite{PadrolPaluPilaudPlamondon}, this property is interesting because it leads on the one hand to a simple description of all polytopal realizations of the fan in the kinematic space~\cite{ArkaniHamedBaiHeYan}, and on the other hand to canonical Minkowski sum decompositions of these realizations.

This paper opens the door to a description of the type cone of the quotient fan for any lattice congruence of the weak order on~$\fS_n$, not only for permutree congruences.
Preliminary computations however indicate that the combinatorics of the facet description of the type cone of an arbitrary quotient fan is much more intricated than that of permutree fans.

The paper is organized as follows.
\cref{sec:preliminaries} provides a recollection of all material needed in the paper, including polyhedral geometry and type cones (\cref{subsec:polyhedralGeometry}), lattice quotients of the weak order (\cref{subsec:latticeCongruences}), deformed permutahedra and removahedra (\cref{subsec:deformedPermutahedraRemovahedra}) and permutrees (\cref{subsec:permutrees}).
\cref{sec:removahedralCongruences} is devoted to the proof of~\cref{thm:main}.
Finally, the type cones of all permutree fans are described in \cref{sec:typeConesPermutreeFans}.


\section{Preliminaries}
\label{sec:preliminaries}

\enlargethispage{.3cm}
We start with preliminaries on polyhedral geometry, type cones, braid arrangements, quotient fans, shards, deformed permutahedra, removahedra and permutrees.
The presentation is largely inspired by the papers~\cite{PadrolPaluPilaudPlamondon, PilaudSantos-quotientopes, PilaudPons-permutrees} and we reproduce here some of their pictures.


\subsection{Polyhedral geometry and type cones}
\label{subsec:polyhedralGeometry}

We start with basic notions of polyhedral geometry (see G.~Ziegler's classic textbook~\cite{Ziegler-polytopes}) and a short introduction to type cones (see the original work of P.~McMullen~\cite{McMullen-typeCone} or their recent application to $\b{g}$-vector fans in~\cite{PadrolPaluPilaudPlamondon}).

\subsubsection{Polyhedral geometry}

A hyperplane~$H \subset \R^n$ is a \defn{supporting hyperplane} of a set~$X \subset \R^n$ if~$H \cap X \ne \varnothing$ and~$X$ is contained in one of the two closed half-spaces of~$\R^n$ defined by~$H$.

We denote by~$\R_{\ge0}\rays \eqdef \set{\sum_{\ray \in \rays} \lambda_{\ray} \, \ray}{\lambda_{\ray} \in \R_{\ge0}}$ the \defn{positive span} of a set~$\rays$ of vectors of~$\R^n$.
A (polyhedral) \defn{cone} is a subset of~$\R^n$ defined equivalently as the positive span of finitely many vectors or as the intersection of finitely many closed linear halfspaces.
Its \defn{faces} are its intersections with its supporting linear hyperplanes, and its \defn{rays} (resp.~\defn{facets}) are its dimension~$1$ (resp. codimension~$1$) faces.
A cone is \defn{simplicial} if it is generated by a set of linearly independent vectors.

A (polyhedral) \defn{fan}~$\Fan$ is a collection of cones which are closed under faces (if~$C \in \Fan$ and~$F$ is a face of~$C$, then~$F \in \Fan$) and intersect properly (if~$C, D \in \Fan$, then~$C \cap D$ is a face of both~$C$ and~$D$).
The \defn{chambers} (resp.~\defn{walls}, resp.~\defn{rays}) of~$\Fan$ are its codimension~$0$ (resp.~codimension~$1$, resp.~dimension~$1$) cones.
The fan~$\Fan$ is \defn{simplicial} if all its cones are, \defn{complete} if the union of its cones covers the ambient space~$\R^n$, and \defn{essential} if it contains the cone~$\{\b{0}\}$. Note that every complete fan is the product of an essential fan with its lineality space (the largest linear subspace contained in all cones of~$\Fan$).
Given two fans~$\Fan, \Fan[G]$ in~$\R^n$, we say that~$\Fan$ \defn{refines}~$\Fan[G]$ (and that~$\Fan[G]$ \defn{coarsens}~$\Fan$) if every cone of~$\Fan[G]$ is a union of cones of~$\Fan$.
In a simplicial fan, we say that two maximal cones are \defn{adjacent} if they share a facet, and that two rays are \defn{exchangeable} if they belong to two adjacent cones but not to their common facet.

A \defn{polytope} is a subset of~$\R^n$ defined equivalently as the convex hull of finitely many points or as a bounded intersection of finitely many closed affine halfspaces.
Its \defn{faces} are its intersections with its supporting affine hyperplanes, and its \defn{vertices} (resp.~\defn{edges}, resp.~\defn{facets}) are its dimension~$0$ (resp. dimension~$1$, codimension~$1$) faces.
The \defn{normal cone} of a face~$F$ of a polytope~$P$ is the cone generated by the outer normal vectors of the facets of~$P$ containing~$F$.
The \defn{normal fan} of~$P$ is the fan formed by the normal cones of all faces of~$P$.

\subsubsection{Type cones}
\label{subsec:typeCones}

Fix an essential complete simplicial fan~$\Fan$ in~$\R^n$ with~$N$ rays.
Let~$\b{G}$ be the $N \times n$-matrix whose rows are (representative vectors for) the rays of~$\Fan$.
For any vector~$\b{h} \in \R^N$, we define the polytope~$P_\b{h} \eqdef \set{\b{x} \in \R^n}{\b{G}\b{x} \le \b{h}}$.
In other words, $P_\b{h}$ has an inequality ${\dotprod{\ray}{\b{x}} \le \b{h}_{\ray}}$ for each ray~$\ray$ of~$\Fan$, where~$\b{h}_{\ray}$ denotes the coordinate of~$\b{h}$ corresponding to~$\ray$.
Note that~$\Fan$ is not necessarily the normal fan of~$P_\b{h}$.
The vectors~$\b{h}$ for which this holds are characterized by the following classical statement.
It already appeared in the study of Minkowski summands of polytopes~\cite{Meyer, McMullen-typeCone}, and in the theory of secondary polytopes~\cite{GelfandKapranovZelevinsky}, see also~\cite{DeLoeraRambauSantos}.
We present here a convenient formulation from~\cite[Lem.~2.1]{ChapotonFominZelevinsky}.

\begin{proposition}
\label{prop:characterizationPolytopalFan}
Let~$\Fan$ be an essential complete simplicial fan in~$\R^n$ and~$\b{G}$ be the $N \times n$-matrix whose rows are the rays of~$\Fan$.
Then the following are equivalent for any vector~$\b{h} \in \R^N$:
\begin{enumerate}
\item The fan~$\Fan$ is the normal fan of the polytope~$P_\b{h} \eqdef \set{\b{x} \in \R^n}{\b{G}\b{x} \le \b{h}}$.
\item For any two adjacent chambers~$\R_{\ge0}\rays$ and~$\R_{\ge0}\rays[S]$ of~$\Fan$ with~$\rays \ssm \{\ray\} = \rays[S] \ssm \{\ray[s]\}$,
\[
\alpha \, \b{h}_{\ray} + \beta \, \b{h}_{\ray[s]} + \sum_{\ray[t] \in \rays \cap \rays[S]} \gamma_{\ray[t]} \, \b{h}_{\ray[t]} > 0,
\]
where
\[
\alpha \, \ray + \beta \, \ray[s] + \sum_{\ray[t] \in \rays \cap \rays[S]} \gamma_{\ray[t]} \, \ray[t] = 0
\]
is the unique (up to rescaling) linear dependence with~$\alpha, \beta > 0$ between the rays of~$\rays \cup \rays[S]$.
\end{enumerate}
\end{proposition}

The inequalities in this statement are called \defn{wall-crossing inequalities}.
For convenience, let us denote by~$\coefficient[{\ray[t]}][\rays][{\rays[S]}]$ the coefficient of~$\ray[t]$ in the unique linear dependence between the rays of~$\rays \cup \rays[S]$ such that~${\coefficient[\ray][\rays][{\rays[S]}] + \coefficient[{\ray[s]}][\rays][{\rays[S]}] = 2}$, so that the inequality above rewrites as~${\sum_{\ray[t] \in \rays \cup \rays[S]} \coefficient[{\ray[t]}][\rays][{\rays[S]}] \, \b{h}_{\ray[t]} > 0}$.

When considering the question of the realizability of a complete simplicial fan~$\Fan$ by a polytope, it is natural to consider all possible realizations of this fan, as was done by P.~McMullen in~\cite{McMullen-typeCone}.
The \defn{type cone} of~$\Fan$ is the cone
\begin{align*}
\typeCone(\Fan) & \eqdef \set{\b{h} \in \R^N}{\Fan \text{ is the normal fan of } P_\b{h}} \\
& = \Bigset{\b{h} \in \R^N}{\sum_{\ray[t] \in \rays \cup \rays[S]} \coefficient[{\ray[t]}][\rays][{\rays[S]}] \, \b{h}_{\ray[t]} > 0 \; \begin{array}{l} \text{for any adjacent chambers} \\ \text{$\R_{\ge0}\rays$ and~$\R_{\ge0}\rays[S]$ of~$\Fan$} \end{array}}.
\end{align*}

Note that the type cone~$\typeCone(\Fan)$ is an open cone.
We denote by $\ctypeCone(\Fan)$ the closure of $\typeCone(\Fan)$, and still call it the type cone of~$\Fan$. 
It is the closed polyhedral cone defined by the inequalities $\sum_{\ray[t] \in \rays \cup \rays[S]} \coefficient[{\ray[t]}][\rays][{\rays[S]}] \, \b{h}_{\ray[t]} \ge 0$ for any adjacent chambers~$\R_{\ge0}\rays$ and~$\R_{\ge0}\rays[S]$. 
It describes all polytopes~$P_\b{h}$ whose normal fans coarsen the fan~$\Fan$.
If $\Fan$ is the normal fan of the polytope~$P$, then $\ctypeCone(\Fan)$ is also known as the \defn{deformation cone} of $P$, see~\cite{Postnikov, PostnikovReinerWilliams}.
We use $\ctypeCone(\Fan)$ rather than~$\typeCone(\Fan)$ when we want to speak about the facets of $\ctypeCone(\Fan)$.

Also observe that the lineality space of the type cone~$\ctypeCone(\Fan)$ has dimension~$n$ (it is invariant by translation in~$\b{G} \R^n$).
In particular, the type cone~$\ctypeCone(\Fan)$ is simplicial when it has~$N-n-1$ facets.
While very particular, the fans for which the type cone is simplicial are very interesting as all their polytopal realizations can be described as follows.

\begin{proposition}[{\cite[Coro.~1.11]{PadrolPaluPilaudPlamondon}}]
\label{prop:simplicialTypeCone}
Let~$\Fan$ be an essential complete simplicial fan in~$\R^n$ with~$N$ rays, such that the type cone~$\ctypeCone(\Fan)$ is simplicial.
Let~$\b{K}$ be the $(N-n) \times N$-matrix whose rows are the inner normal vectors of the facets of~$\ctypeCone(\Fan)$.
Then the polytope
\[
Q(\b{u}) \eqdef \set{\b{z} \in \R_{\ge 0}^N}{\b{K}\b{z} = \b{u}}
\]
is a realization of the fan~$\Fan$ for any positive vector~$\b{u} \in \R_{>0}^{N-n}$.
Moreover, the polytopes~$Q(\b{u})$ for~$\b{u} \in \R_{>0}^{N-n}$ describe all polytopal realizations of~$\Fan$.
\end{proposition}

In this paper, we shall also use a non-simplicial version of \cref{prop:characterizationPolytopalFan}, see \eg~\cite{Meyer}.

\begin{proposition}
\label{prop:characterizationPolytopalFanNonSimplicial}
Let~$\Fan$ be an essential complete (non necessarily simplicial) fan in~$\R^n$ and~$\b{G}$ be the $N \times n$-matrix whose rows are the rays of~$\Fan$.
Then the following are equivalent for any~$\b{h} \in \R^N$:
\begin{enumerate}
\item The fan~$\Fan$ is the normal fan of the polytope~$P_\b{h} \eqdef \set{\b{x} \in \R^n}{\b{G}\b{x} \le \b{h}}$.
\item The coordinates of~$\b{h}$ satisfy the following equalities and inequalities:
	\begin{enumerate}[$\bullet$]
	\item for any (non-simplicial) chamber~$\R_{\ge0}\rays$ of~$\Fan$ and any linear dependence~${\sum_{\ray \in \rays} \gamma_{\ray} \ray = 0}$ among the rays of~$\rays$, we have~$\sum_{\ray \in \rays} \gamma_{\ray} \b{h}_{\ray} = 0$,
	\item for any two adjacent chambers~$\R_{\ge0}\rays$ and~$\R_{\ge0}\rays[S]$ of~$\Fan$, any rays~$\ray \in \rays \ssm \rays[S]$ and~${\ray[s] \in \rays[S] \ssm \rays}$, and any linear dependence $\alpha \, \ray + \beta \, \ray[s] + \sum_{\ray[t] \in \rays \cap \rays[S]} \gamma_{\ray[t]} \, \ray[t] = 0$ among the rays~$\{\ray, \ray[s]\} \cup (\rays \cap \rays[S])$ with $\alpha, \beta > 0$, we have~$\alpha \, \b{h}_{\ray} + \beta \, \b{h}_{\ray[s]} + \sum_{\ray[t] \in \rays \cap \rays[S]} \gamma_{\ray[t]} \, \b{h}_{\ray[t]} > 0$.
	\end{enumerate}
\end{enumerate}
\end{proposition}


\subsection{Geometry of lattice quotients of the weak order}
\label{subsec:latticeCongruences}

We now recall the combinatorial and geometric toolbox to deal with lattice quotients of the weak order on~$\fS_n$.
The presentation and pictures are borrowed from~\cite{PilaudSantos-quotientopes}.

\subsubsection{Weak order, braid fan, and permutahedron}

\begin{figure}[b]
	\capstart
	\centerline{\includegraphics[scale=.6]{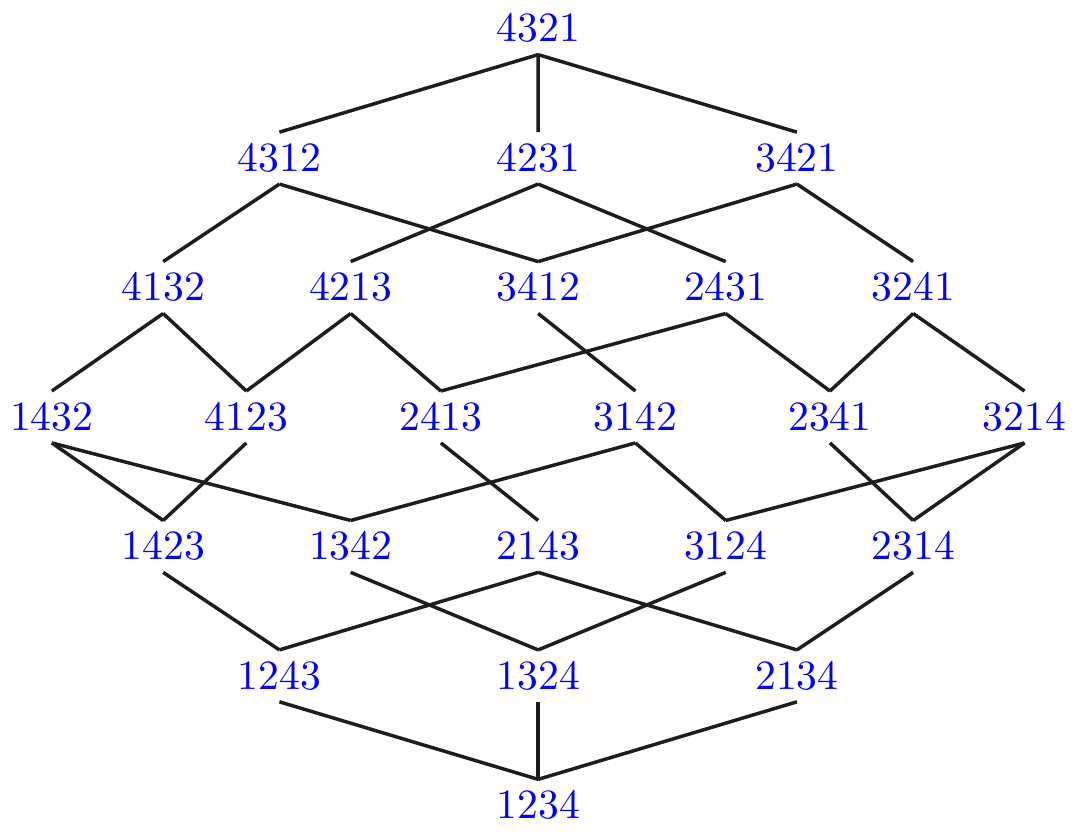} \; \includegraphics[scale=.6]{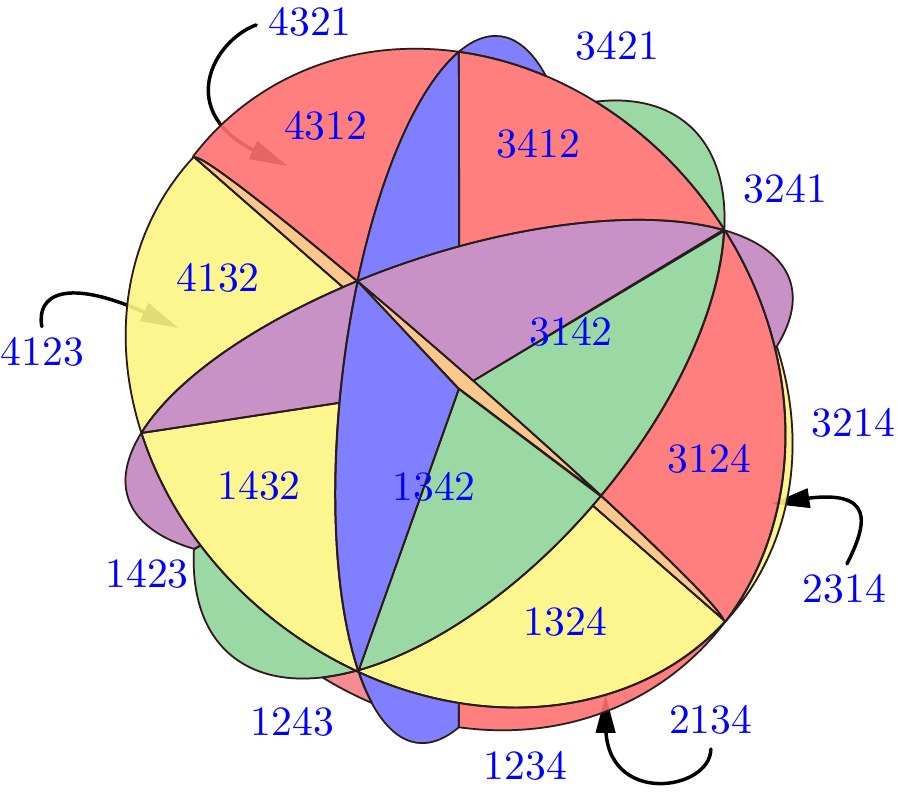} \; \includegraphics[scale=.6]{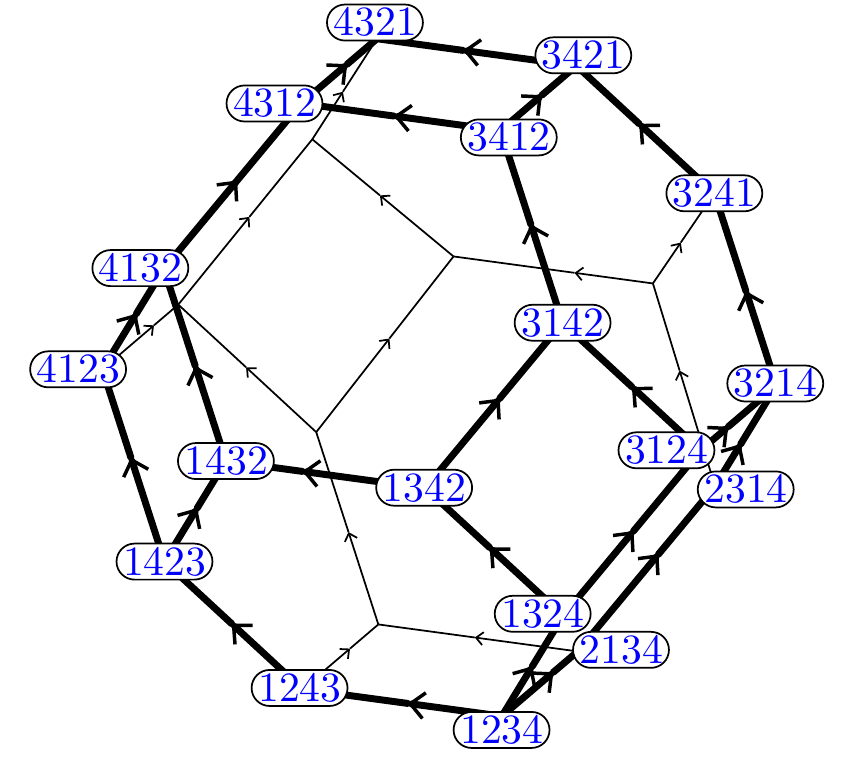}}
	\caption{The Hasse diagram of the weak order on~$\fS_4$ (left) can be seen as the dual graph of the braid fan~$\Fan_4$ (middle) or as the graph of the permutahedron~$\Perm[4]$ (right). \cite[Fig.~1]{PilaudSantos-quotientopes}}
	\label{fig:weakOrder4}
\end{figure}

Let~$\fS_n$ denote the set of permutations of the set~$[n] \eqdef \{1, \dots, n\}$.
We consider the \defn{weak order} on~$\fS_n$ defined by ${\sigma \le \tau \iff \inv(\sigma) \subseteq \inv(\tau)}$ where~$\inv(\sigma) \eqdef \set{(\sigma_i, \sigma_j)}{1 \le i < j \le n \text{ and } \sigma_i > \sigma_j}$ is the \defn{inversion set} of~$\sigma$.
See \cref{fig:weakOrder4}\,(left).
It can be interpreted geometrically on the braid fan~$\Fan_n$ or the permutahedron~$\Perm$ defined~below.

The \defn{braid arrangement} is the set~$\HA_n$ of hyperplanes~$\set{\b{x} \in \R^n}{\b{x}_i = \b{x}_j}$ for ${1 \le i < j \le n}$.
As all hyperplanes of~$\HA_n$ contain the line~$\R \one \eqdef \R(1,1,\dots,1)$, we restrict to the hyperplane ${\hyp \eqdef \bigset{\b{x} \in \R^n}{\sum_{i \in [n]} \b{x}_i = 0}}$.
The hyperplanes of~$\HA_n$ divide~$\hyp$ into chambers, which are the maximal cones of a complete simplicial fan~$\Fan_n$, called \defn{braid fan}.
It has 
\begin{itemize}
\item a chamber~$C(\sigma) \eqdef \set{\b{x} \in \hyp}{\b{x}_{\sigma_1} \le \b{x}_{\sigma_2} \le \dots \le \b{x}_{\sigma_n} }$ for each permutation~$\sigma$ of~$\fS_n$, 
\item a ray~$C(I) \eqdef \set{\b{x} \in \hyp}{\b{x}_{i_1} = \dots = \b{x}_{i_p} \le \b{x}_{j_1} = \dots = \b{x}_{j_{n-p}}}$ for each subset~$\varnothing \ne I \subsetneq [n]$, where~$I = \{i_1, \dots, i_p\}$ and~$[n] \ssm I = \{j_1, \dots, j_{n-p}\}$. When needed, we use the representative vector~$\ray(I) \eqdef |I| \one - n \one_I$ in~$C(I)$, where~$\one \eqdef \sum_{i \in [n]} \b{e}_i$ and~$\one_I \eqdef \sum_{i \in I} \b{e}_i$. (Note that the advantage of~$\ray(I)$ over~$\one_I$ is that it belongs to the hyperplane~$\hyp$ and leads to a simple formula for the support function of the classical permutahedron.)
\end{itemize}
The chamber~$C(\sigma)$ has rays~$C(\sigma([k]))$ for~$k \in [n]$.
See Figures~\ref{fig:weakOrder4}\,(middle), \ref{fig:shards3}\,(left) and \ref{fig:shards4}\,(left) where the chambers are labeled in blue and the rays are labeled in red.

\enlargethispage{.7cm}
The \defn{permutahedron} is the polytope~$\Perm$ defined equivalently as:
\begin{itemize}
\item the convex hull of the points~$\sum_{i \in [n]} i \, \b{e}_{\sigma_i}$ for all permutations~$\sigma \in \fS_n$,
\item the intersection of the hyperplane~$\Hyp \eqdef \bigset{\b{x} \in \R^n}{\sum_{i \in [n]} \b{x}_i = \binom{n+1}{2}}$ with the halfspaces $\bigset{\b{x} \in \R^n}{\sum_{i \in I} \b{x}_i \ge \binom{|I|+1}{2}}$ for all proper subsets~${\varnothing \ne I \subsetneq [n]}$.\end{itemize}
See \cref{fig:weakOrder4}\,(right).
This standard facet description is equivalent to~${\dotprod{\ray(I)}{\b{x}} \le n|I|(n-|I|)/2}$ for all proper subsets~${\varnothing \ne I \subsetneq [n]}$, which matches the conventions of \cref{subsec:typeCones}.

The normal fan of the permutahedron~$\Perm$ is the braid fan~$\Fan_n$.
The Hasse diagram of the weak order on~$\fS_n$ can be seen geometrically as the dual graph of the braid fan~$\Fan_n$, or the graph of the permutahedron~$\Perm$, oriented in the linear direction~$\b{\alpha} \eqdef \sum_{i \in [n]} (2i-n-1) \, \b{e}_i = (-n+1, -n+3, \dots, n-3, n-1)$.
See \cref{fig:weakOrder4}.

\subsubsection{Lattice congruences and quotient fans}

A \defn{lattice congruence} of a lattice~$(L,\le,\meet,\join)$ is an equivalence relation on~$L$ that respects the meet and the join operations, \ie such that $x \equiv x'$ and~$y \equiv y'$ implies $x \meet y \, \equiv \, x' \meet y'$ and~$x \join y \, \equiv \, x' \join y'$.
It defines a \defn{lattice quotient}~$L/{\equiv}$ on the congruence classes of~$\equiv$ where~$X \le Y$ if and only if there exist~$x \in X$ and~$y \in Y$ such that~$x \le y$, and~$X \meet Y$ (resp.~$X \join Y$) is the congruence class of~$x \meet y$ (resp.~$x \join y$) for any~$x \in X$~and~$y \in Y$.

\begin{example}
\label{exm:sylvesterCongruence}
The prototype lattice congruence of the weak order is the \defn{sylvester congruence}~$\equiv_\textrm{sylv}$, see~\cite{LodayRonco, HivertNovelliThibon-algebraBinarySearchTrees}.
Its congruence classes are the fibers of the binary search tree insertion algorithm, or equivalently the sets of linear extensions of binary trees (labeled in inorder and considered as posets oriented from bottom to top).
It can also be seen as the transitive closure of the rewriting rule~$U i k V j W \equiv_\textrm{sylv} U k i V j W$ for some letters~$i < j < k$ and words~$U,V,W$ on~$[n]$.
The quotient of the weak order by the sylvester congruence is (isomorphic to) the classical \defn{Tamari lattice}~\cite{Tamari}, whose elements are the binary trees on~$n$ nodes and whose cover relations are rotations in binary trees.
The sylvester congruence and the Tamari lattice are illustrated in \cref{fig:Tamari4} for~$n = 4$.
We will use the sylvester congruence and the Tamari lattice as a familiar example throughout~the~paper.

\begin{figure}[b]
	\capstart
	\centerline{\includegraphics[scale=.6]{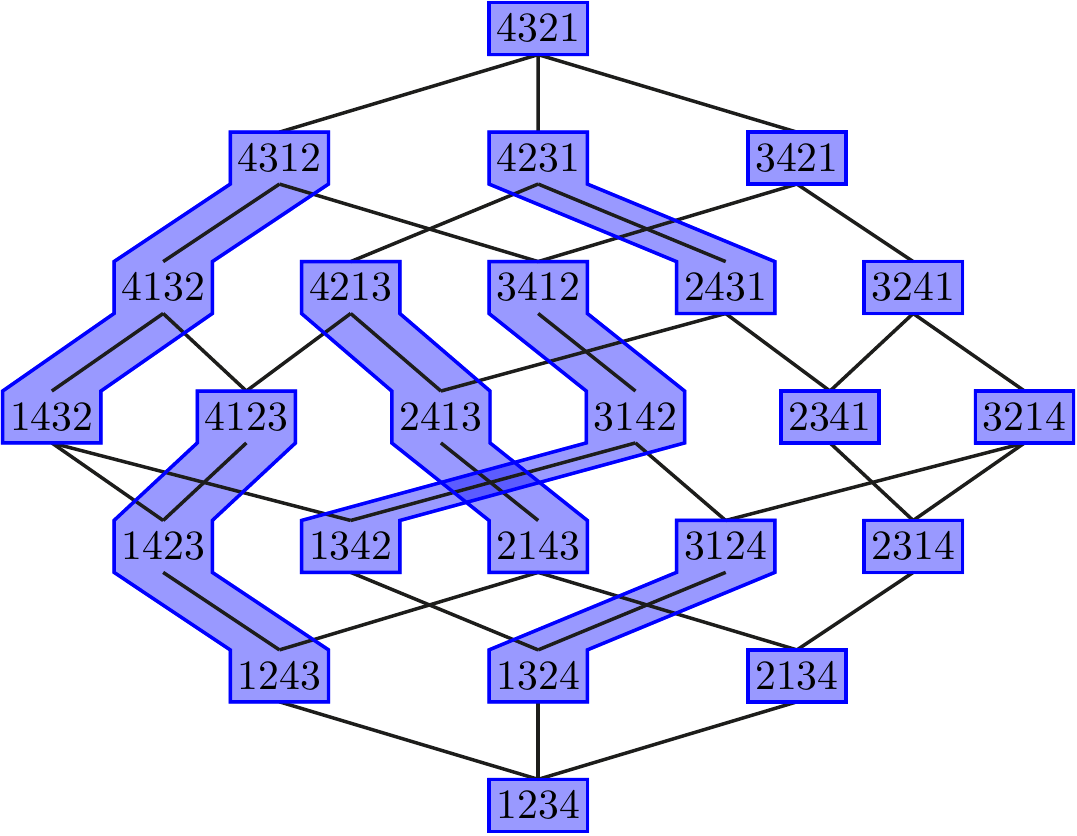} \; \includegraphics[scale=.48]{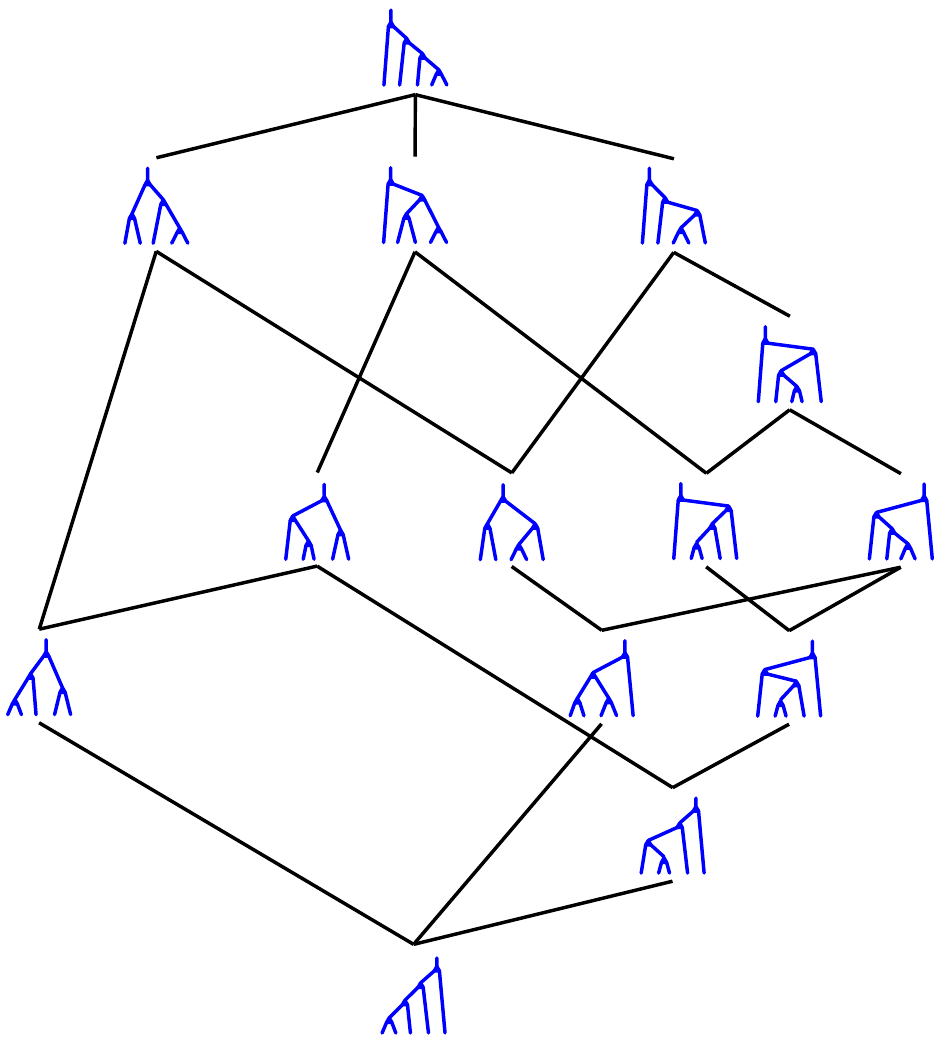} \hspace{-.5cm} \includegraphics[scale=.6]{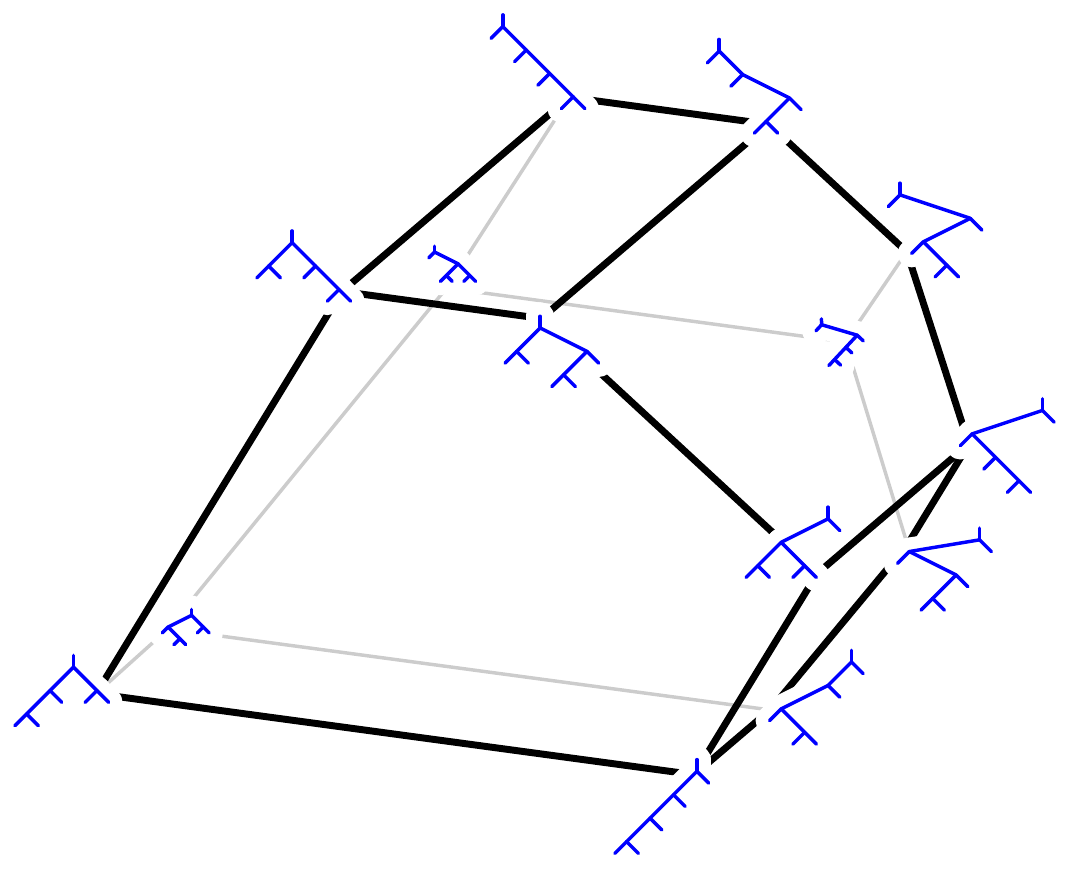}}
	\caption{The quotient of the weak order by the sylvester congruence~$\equiv_\textrm{sylv}$~(left) is the Tamari lattice (middle) and is realized by J.-L.~Loday's associahedron (right). \cite[Fig.~2]{PilaudSantos-quotientopes}}
	\label{fig:Tamari4}
\end{figure}
\end{example}

Lattice congruences naturally yield quotient fans, which turn out to be polytopal.

\begin{theorem}[\cite{Reading-HopfAlgebras}]
\label{thm:quotientFanGlueing}
Any lattice congruence~$\equiv$ of the weak order on~$\fS_n$ defines a complete fan~$\Fan_\equiv$, called \defn{quotient fan}, whose chambers are obtained gluing together the chambers~$C(\sigma)$ of the braid fan~$\Fan_n$ corresponding to permutations~$\sigma$ that belong to the same congruence class of~$\equiv$.
\end{theorem}

\begin{theorem}[\cite{PilaudSantos-quotientopes}]
\label{thm:quotientopes}
For any lattice congruence~$\equiv$ of the weak order on~$\fS_n$, the quotient fan~$\Fan_\equiv$ is the normal fan of a polytope~$P_\equiv$, called \defn{quotientope}.
\end{theorem}

By construction, the Hasse diagram of the quotient of the weak order by~${\equiv}$ is given by the dual graph of the quotient fan~$\Fan_\equiv$, or by the graph of the quotientope~$P_\equiv$, oriented in the direction~$\b{\alpha}$.

\begin{example}
For the sylvester congruence~$\equiv_\textrm{sylv}$ of \cref{exm:sylvesterCongruence}, the quotient fan~$\Fan_\textrm{sylv}$ has
\begin{itemize}
\item a chamber $C(T) = \set{\b{x} \in \hyp}{\b{x}_i \le \b{x}_j \text{ if } i \text{ is a descendant of } j \text{ in } T}$ for each binary tree~$T$, 
\item a ray~$C(I)$ for each proper interval~$I = [i,j] \subsetneq [n]$.
\end{itemize}
\cref{fig:shards3,fig:shards4} (right) illustrate the quotient fans~$\Fan_\textrm{sylv}$ for~$n = 3$ and~$n = 4$.
The quotient fan~$\Fan_\textrm{sylv}$ is the normal fan of the classical associahedron~$\Asso$ defined equivalently as:
\begin{itemize}
\item the convex hull of the points~$\sum_{j \in [n]} \ell(T,j) \, r(T,j) \, \b{e}_j$ for all binary trees~$T$ on~$n$ nodes, where $\ell(T,j)$ and~$r(T,j)$ respectively denote the numbers of leaves in the left and right subtrees of the node $j$ of~$T$ (labeled in inorder), see~\cite{Loday},
\item the intersection of the hyperplane~$\Hyp$ with the halfspaces~$\bigset{\b{x} \in \R^n}{\sum_{i \le k \le j} \b{x}_k \ge \binom{j-i+2}{2}}$ for all intervals~$1 \le i \le j \le n$, see~\cite{ShniderSternberg}.
\end{itemize}
See \cref{fig:Tamari4} (right).
\end{example}

\subsubsection{Shards}

An alternative description of the quotient fan~$\Fan_\equiv$ defined in \cref{thm:quotientFanGlueing} is given by its walls, each of which can be seen as the union of some preserved walls of the braid arrangement.
The conditions in the definition of lattice congruences impose strong constraints on the set of preserved walls: deleting some walls forces to delete others.
Shards were introduced by N.~Reading in~\cite{Reading-posetRegions} (see also~\cite{Reading-PosetRegionsChapter, Reading-FiniteCoxeterGroupsChapter}) to understand the possible sets of preserved walls. 

For any~$1 \le i < j \le n$, let~$[i,j] \eqdef \{i, \dots, j\}$ and~${]i,j[} \eqdef \{i+1, \dots, j-1\}$.
For any~$S \subseteq {]i,j[}$, the \defn{shard}~$\shard(i,j,S)$ is the cone
\[
\shard(i,j,S) \eqdef \set{\b{x} \in \R^n}{\b{x}_i = \b{x}_j, \; \b{x}_i \ge \b{x}_h \text{ for all } h \in S, \; \b{x}_i \le \b{x}_k \text{ for all } k \in {]i,j[} \ssm S}.
\]
The \defn{length} of~$\shard(i,j,S)$ is~$j-i$.
Denote by~$\shards_n \eqdef \set{\shard(i,j,S)}{1 \le i < j \le n \text{ and } S \subseteq {]i,j[}}$ the set of all shards of~$\HA_n$.

Throughout the paper, we use a convenient notation for shards borrowed from N.~Reading's work on arc diagrams~\cite{Reading-arcDiagrams}: we consider $n$ dots on the horizontal axis, and we represent the shard~$\shard(i,j,S)$ by an arc joining the $i$th dot to the $j$th dot and passing above (resp.~below) the $k$th dot when~$k \in S$ (resp.~when~$k \notin S$).
For instance, the arc~\raisebox{-.16cm}{\includegraphics[scale=.8]{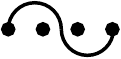}} represents the shard~$\shard(1,4,\{2\})$.
We say that~$\shard(i,j,S)$ is an \defn{up} (resp.~\defn{down}) shard if the corresponding arc passes above (resp.~below) all points of~$]i,j[$, that is, if~$S = {]i,j[}$ (resp.~$S = \varnothing$).
We say that~$\shard(i,j,S)$ is \defn{mixed} if the corresponding arc crosses the horizontal axis, that is, if~$\varnothing \ne S \subsetneq {]i,j[}$.

\begin{figure}[b]
	\capstart
	\centerline{\includegraphics[scale=.9]{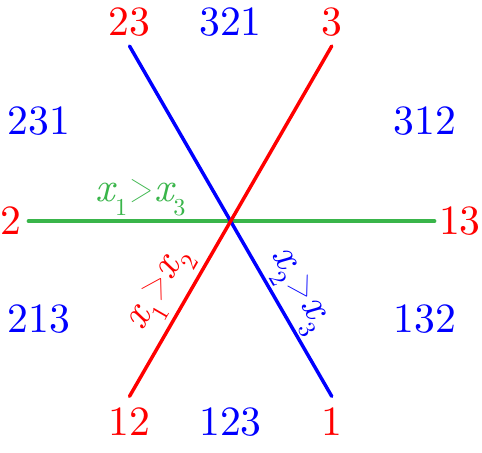} \qquad \includegraphics[scale=.9]{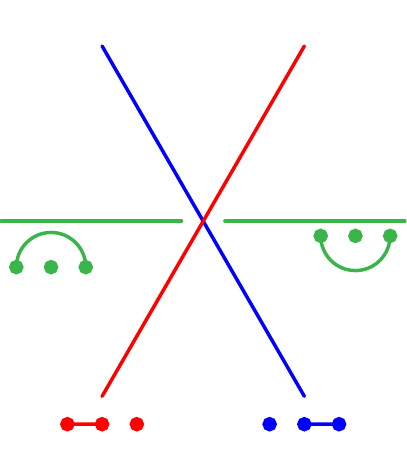} \qquad \includegraphics[scale=.9]{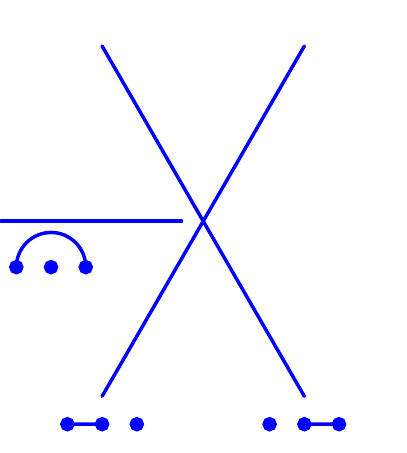}}
	\caption{The braid fan~$\Fan_3$ (left), the corresponding shards (middle), and the quotient fan of the sylvester congruence~$\equiv_\textrm{sylv}$~(right). \cite[Fig.~3]{PilaudSantos-quotientopes}}
	\label{fig:shards3}
\end{figure}

\begin{figure}
	\capstart
	\centerline{\includegraphics[scale=.5]{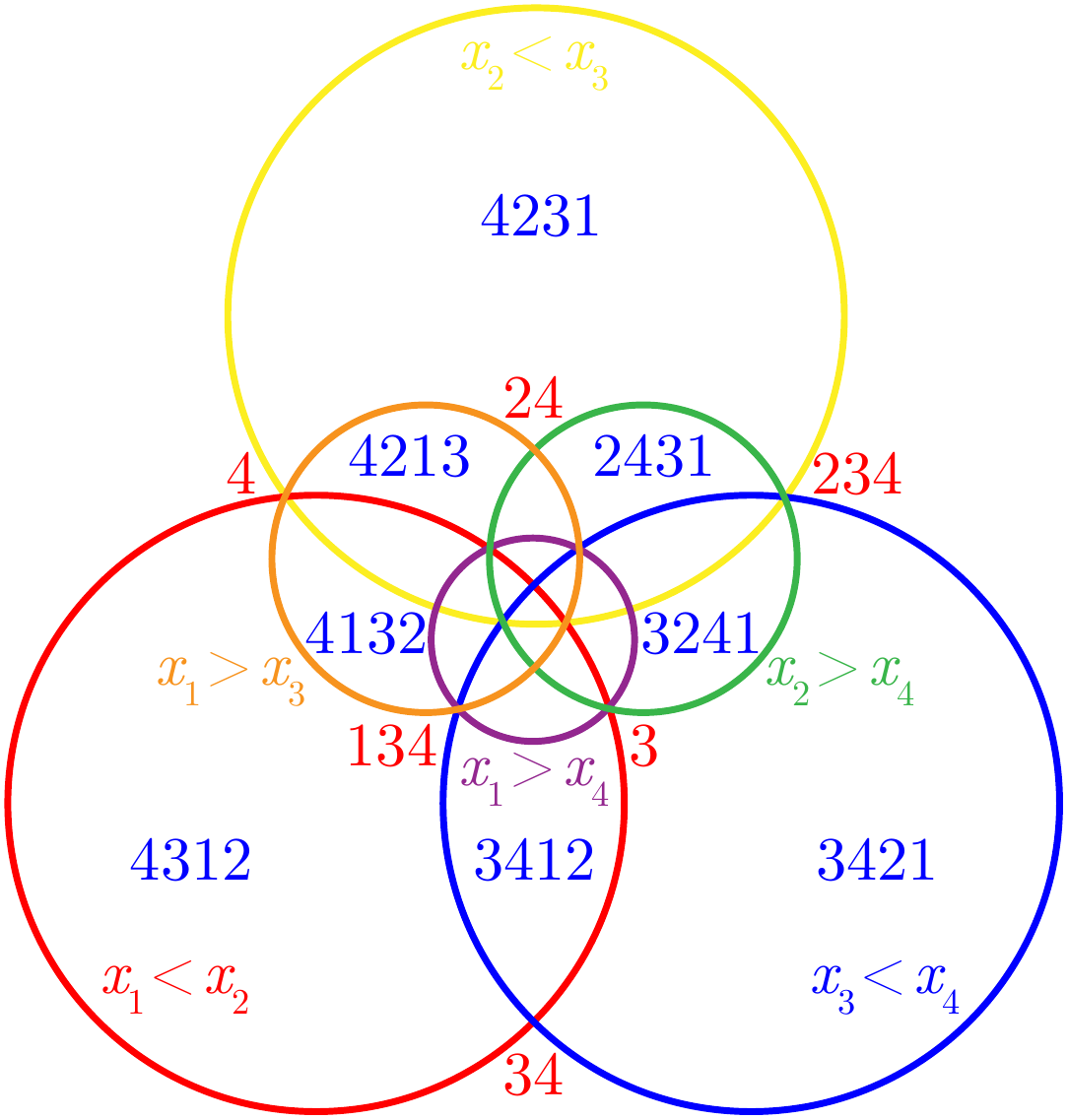} \; \includegraphics[scale=.5]{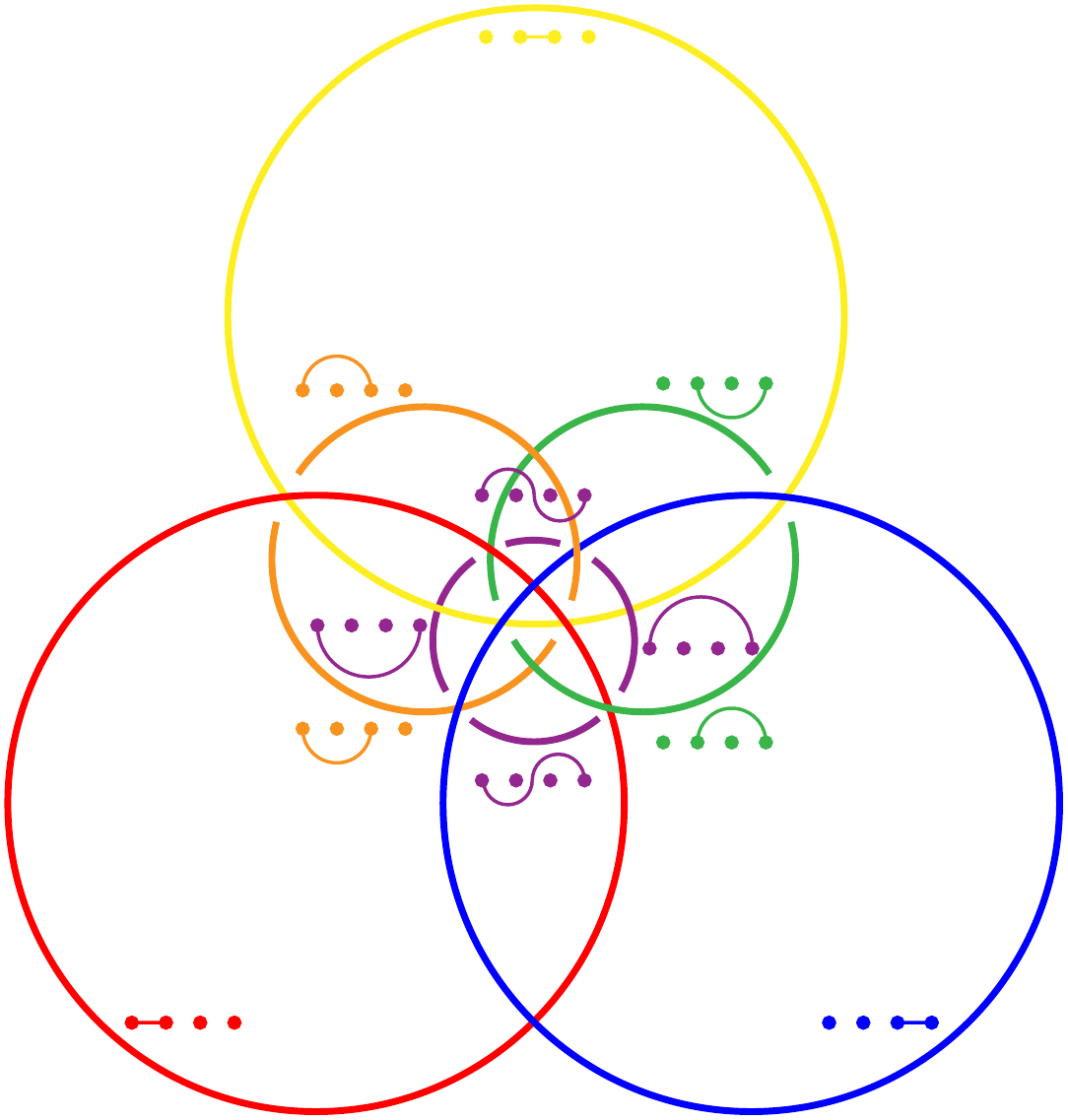} \; \includegraphics[scale=.5]{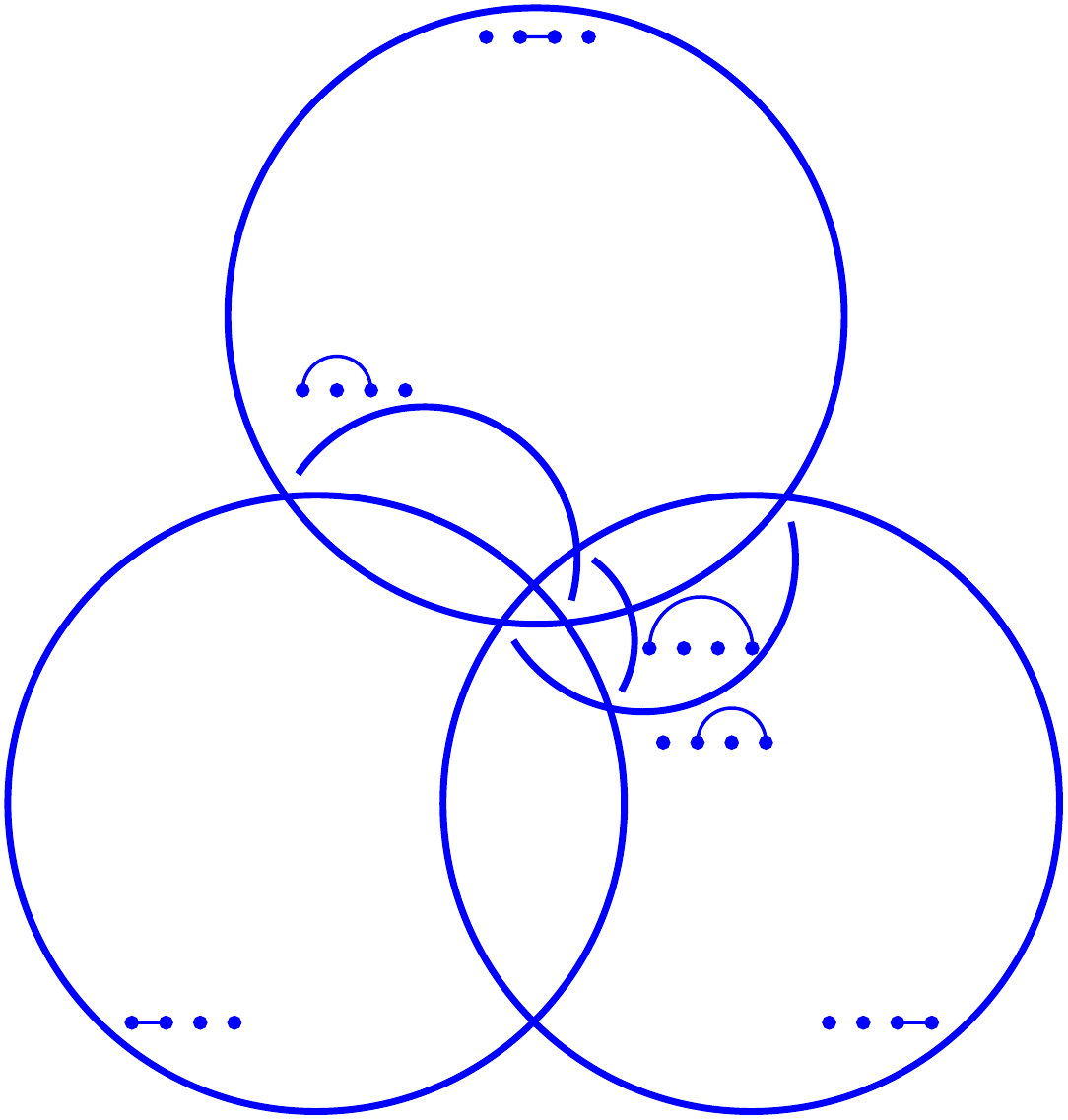}}
	\caption{A stereographic projection of the braid fan~$\Fan_4$ (left) from the pole~$4321$, the corresponding shards (middle), and the quotient fan of the sylvester congruence~$\equiv_\textrm{sylv}$~(right). \mbox{\cite[Fig.~4]{PilaudSantos-quotientopes}}}
	\label{fig:shards4}
\end{figure}

\cref{fig:shards3,fig:shards4} illustrate the braid fans~$\Fan_n$ and their shards~$\shards_n$ when~$n = 3$ and~$n = 4$ respectively.
As the $3$-dimensional fan~$\Fan_4$ is difficult to visualize (as in \cref{fig:weakOrder4}\,(middle)), we use another classical representation in \cref{fig:shards4}\,(left): we intersect~$\Fan_4$ with a unit sphere and we stereographically project the resulting arrangement of great circles from the pole~$4321$ to the plane.
Each circle then corresponds to a hyperplane~$\b{x}_i = \b{x}_j$ with~$i < j$, separating a disk where~$\b{x}_i < \b{x}_j$ from an unbounded region where~$\b{x}_i > \b{x}_j$.
In both \cref{fig:shards3,fig:shards4}, the left picture shows the braid fan~$\Fan_n$ (where chambers are labeled with blue permutations of~$[n]$ and rays are labeled with red proper subsets of~$[n]$), the middle picture shows the shards~$\shards_n$ (labeled by arcs), and the right picture represents the quotient fan~$\Fan_\textrm{sylv}$ by the sylvester congruence.

It turns out that the shards are precisely the pieces of the hyperplanes of~$\HA_n$ needed to delimit the cones of the quotient fan~$\Fan_\equiv$.

\begin{theorem}[{\cite[Sect.~10.5]{Reading-FiniteCoxeterGroupsChapter}}]
\label{thm:quotientFanShards}
	For any lattice congruence~$\equiv$ of the weak order on~$\fS_n$, there is a subset of shards~$\shards_\equiv \subseteq \shards_n$ such that the interiors of the chambers of the fan~$\Fan_\equiv$ are precisely the connected components of~$\hyp \ssm \bigcup \shards_\equiv$.
\end{theorem}

\begin{figure}[b]
	\capstart
	\centerline{\includegraphics[scale=.7,valign=c]{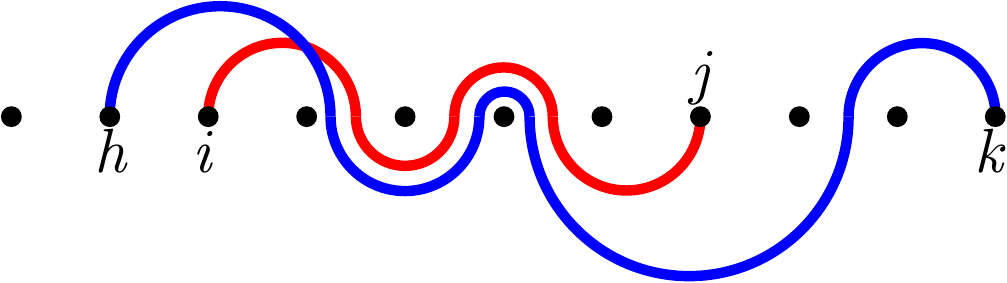} \hspace{1cm} \includegraphics[scale=.7,valign=c]{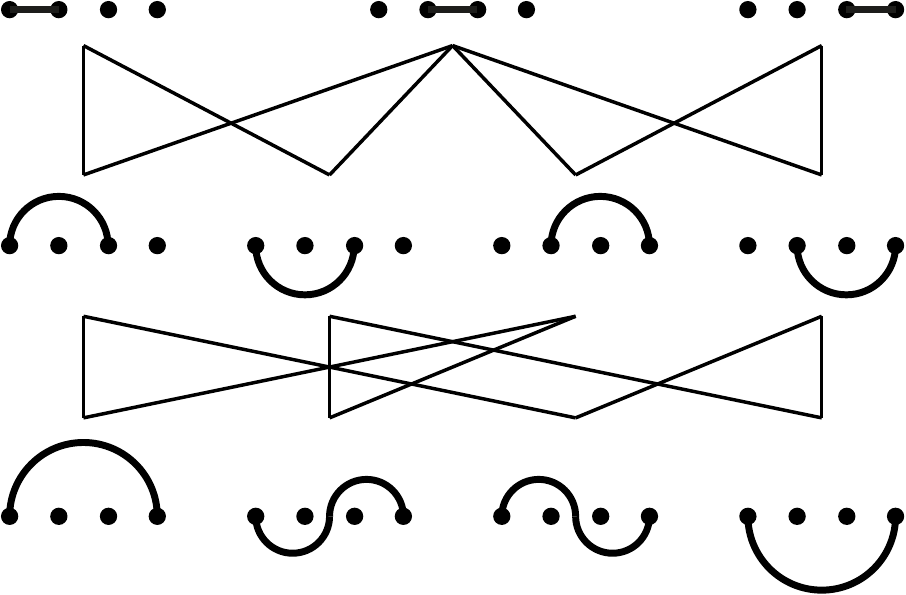}}
	\caption{The forcing relation on shards (left) and the shard poset for~$n = 4$ (right). The red shard~$\shard(i,j,S)$ forces the blue shard~$\shard(h,k,T)$ since~$h \le i < j \le k$ and~$S = T \cap {]i,j[}$. \mbox{\cite[Fig.~5]{PilaudSantos-quotientopes}}}
	\label{fig:forcingOrder}
\end{figure}

Finally, we can describe the set of lattice congruences of the weak order on~$\fS_n$ using the following poset on the shards~$\shards_n$.
A shard~$\shard(i,j,S)$ is said to \defn{force} a shard~$\shard(h,k,T)$ if~$h \le i < j \le k$ and~$S = T \cap {]i,j[}$.
We denote this relation by~$\shard(i,j,S) \succ \shard(h,k,T)$.
In terms of the corresponding arcs, the arc~$\alpha$ of~$\shard(i,j,S)$ is a subarc of the arc~$\beta$ of~$\shard(h,k,T)$, meaning that the endpoints of~$\alpha$ are in between the endpoints of~$\beta$, and the arc~$\alpha$ agrees with the arc~$\beta$ between~$i$ and~$j$.
We call \defn{shard poset} the poset~$(\shards_n, \prec)$ of all shards ordered by forcing.
The forcing relation and the shard poset on~$\shards_4$ are illustrated on \cref{fig:forcingOrder}.

\begin{theorem}[{\cite[Sect.~10.5]{Reading-FiniteCoxeterGroupsChapter}}]
\label{thm:shardIdeals}
The map~${\equiv} \mapsto \shards_\equiv$ is a bijection between the lattice congruences of the weak order on~$\fS_n$ and the upper ideals of the shard poset~$(\shards_n,\prec)$.
\end{theorem}

\begin{example}
For the sylvester congruence~$\equiv_\textrm{sylv}$, the corresponding shard ideal is the ideal~$\shards_\textrm{sylv} = \set{\shard(i, j, ]i,j[)}{1 \le i < j \le n}$ of all up shards, \ie those whose corresponding arcs pass above all dots in between their endpoints.
\cref{fig:shards3,fig:shards4}\,(right) represent the quotient fans~$\Fan_{\equiv_\textrm{sylv}}$ corresponding to the sylvester congruences~$\equiv_\textrm{sylv}$ on~$\fS_3$ and~$\fS_4$ respectively.
It is obtained 
\begin{itemize}
\item either by gluing the chambers~$C(\sigma)$ of the permutations~$\sigma$ in the same sylvester class,
\item or by cutting the space with the shards of~$\shards_{\equiv_\textrm{sylv}} = \set{\shard(i, j, ]i,j[)}{1 \le i < j \le 4}$.
\end{itemize}
\end{example}

To conclude these recollections on lattice congruences of the weak order, let us recall that the quotient fan~$\Fan_\equiv$ is essential if and only if the identity permutation is alone in its $\equiv$-congruence class, or equivalently if~$\shards_\equiv$ contains all basic shards~$\shard(i, i+1, \varnothing)$ for~$i \in [n-1]$ (this follows \eg from~\cite[Thm.~6.9]{Reading-latticeCongruences}).
We say that such a congruence is \defn{essential}.
If~$\shards_\equiv$ does not contain the shard~$\shard(i, i+1, \varnothing)$, then the quotient~$\fS_n/{\equiv}$ is isomorphic to the Cartesian product of the quotients~$\fS_i/{\equiv'}$ and~$\fS_{n-i}/{\equiv''}$ where~$\equiv'$ and~$\equiv''$ are the restrictions of~$\equiv$ to~$[1,i]$ and~$[i+1,n]$ respectively.
Any lattice congruence can thus be understood from its essential restrictions and we therefore focus on essential congruences.


\subsection{Deformed permutahedra and removahedra}
\label{subsec:deformedPermutahedraRemovahedra}

This section discusses the type cone of the braid fan~$\Fan_n$.
As they belong to the deformation cone of the permutahedron, we call the resulting polytopes \defn{deformed permutahedra}.
We also discuss a subfamily of specific deformed permutahedra called \defn{removahedra} as they are obtained by deleting facets (instead of moving them).

\subsubsection{Linear dependences in the braid fan}

We start with classical considerations on the geometry of the braid fan.
Remember that we have chosen a representative vector~$\ray(I) \eqdef |I|\one - n\one_I$ for the ray corresponding to each proper subset~$\varnothing \ne I \subsetneq [n]$ (where~$\one \eqdef \sum_{i \in [n]} \b{e}_i$ and~$\one_I \eqdef \sum_{i \in I} \b{e}_i$).
We also set~$\ray(\varnothing) = \ray([n]) = 0$ by convention.

\begin{lemma}
\label{lem:linearDependence1}
For any two proper subsets~$\varnothing \ne I, J \subsetneq [n]$, the representative vectors satisfy the linear dependence~$\ray(I) + \ray(J) = \ray(I \cap J) + \ray(I \cup J)$.
\end{lemma}

\begin{lemma}
\label{lem:linearDependence2}
Let~$\sigma, \tau$ be two adjacent permutations. Let~$\varnothing \ne I \subsetneq [n]$ (resp.~$\varnothing \ne J \subsetneq [n]$) be such that~$\ray(I)$ (resp.~$\ray(J)$) is the ray of~$C(\sigma)$ not in~$C(\tau)$ (resp.~of~$C(\tau)$ not in~$C(\sigma)$). Then the linear dependence among the rays of the cones~$C(\sigma)$ and~$C(\tau)$ is~$\ray(I) + \ray(J) = \ray(I \cap J) + \ray(I \cup J)$.
\end{lemma}

For example, the linear dependence among the rays in the adjacent cones~$C(123)$ and~$C(213)$ of~$\Fan_3$ is~$\ray(\{1\}) + \ray(\{2\}) = \ray(\{12\})$, while the linear dependence among the rays in the adjacent cones~$C(123)$ and~$C(132)$ of~$\Fan_3$ is~$\ray(\{1,2\}) + \ray(\{1,3\}) = \ray(\{1\})$. See \fref{fig:shards3}\,(left).
The first non-degenerate linear dependencies (\ie where~$I \cap J \ne \varnothing$ and~$I \cup J \ne [n]$) arise in~$\Fan_4$: for instance, the linear dependence among the rays in the adjacent cones~$C(4132)$ and~$C(4312)$ of~$\Fan_4$ is~$\ray(\{3,4\}) + \ray(\{1,4\}) = \ray(\{4\}) + \ray(\{1,3,4\})$. See \fref{fig:shards4}\,(left).

\subsubsection{Deformed permutahedra}

We now consider the type cone of the braid fan, or in other words the deformation cone of the permutahedron, studied in details in~\cite{Postnikov, PostnikovReinerWilliams}.
The following classical statement is a consequence of \cref{prop:characterizationPolytopalFan,lem:linearDependence2}.
We naturally identify a vector~$\b{h}$ with coordinates indexed by the rays of the braid fan~$\Fan_n$ with a height function~$h : 2^{[n]} \to \R_{\ge0}$ with~$h(\varnothing) = h([n]) = 0$.

\begin{proposition}
The type cone~$\ctypeCone(\Fan_n)$ of the braid fan~$\Fan_n$ (or deformation cone of the permutahedron~$\Perm$) is (isomorphic to) the set of functions~$h : 2^{[n]} \to \R$ satisfying ${h(\varnothing) = h([n]) = 0}$ and the \defn{submodular inequalities} $h(I) + h(J) \ge h(I \cap J) + h(I \cup J)$ for any~${I, J \subseteq [n]}$.
The facets of~$\ctypeCone(\Fan_n)$ correspond to those submodular inequalities where~$|I \ssm J| = |J \ssm I| = 1$.
\end{proposition}

For instance, the height function of the permutahedron~$\Perm$ is given by
\[
h_\circ(I) = \max_{\sigma \in \fS_n} \dotprod{\ray(I)}{\sigma} = |I|n(n+1)/2 - n|I|(|I|+1)/2 = n|I|(n-|I|)/2.
\]
It is clearly submodular since~$h_\circ(I) + h_\circ(J) - h_\circ(I \cap J) - h_\circ(I \cup J) = 2n|I \ssm J||J \ssm I| \ge 0$.

A \defn{deformed permutahedron} is a polytope whose normal fan coarsens the braid fan~$\Fan_n$.
It can be written as~$\Defo \eqdef \bigset{\b{x} \in \Hyp}{\dotprod{\ray(I)}{\b{x}} \le h(I) \text{ for all } \varnothing \ne I \subsetneq [n]}$ for some submodular function~$h : 2^{[n]} \to \R$.
We prefer the term deformed permutahedron to the term generalized permutahedron used by A.~Postnikov in~\cite{Postnikov,PostnikovReinerWilliams} (in particular because it also generalizes to other Coxeter groups~\cite{ArdilaCastilloEurPostnikov}).
Observe in particular that any quotientope~$P_\equiv$ of \cref{thm:quotientopes} is a deformed permutahedron since the quotient fan~$\Fan_\equiv$ coarsens the braid fan~$\Fan_n$ by definition.

\subsubsection{Removahedra}

A \defn{removahedron} is a deformed permutahedron obtained by deleting inequalities in the facet description of the permutahedron~$\Perm$.
In other words, it can be written as~$\Remo \eqdef \bigset{\b{x} \in \Hyp}{\dotprod{\ray(I)}{\b{x}} \le h_\circ(I) \text{ for all } I \in \cI}$ for a subset~$\cI$ of proper subsets of~$[n]$.
Examples of removahedra include the permutahedron~$\Perm$ itself (remove no inequality), the associahedron~$\Asso$ (remove the inequalities that do not correspond to intervals), the graph associahedron~$\Asso[G]$~\cite{CarrDevadoss, Devadoss} if and only if the graph~$G$ is chordful~\cite{Pilaud-removahedra}, and the permutreehedra~\cite{PilaudPons-permutrees} described below.

We say that a fan~$\Fan[G]$ with rays~$\set{\ray(I)}{I \in \cI}$ is \defn{removahedral} if $\Fan[G]$ is the normal fan of the removahedron~$\Remo$.
We say that a lattice congruence~$\equiv$ of the weak order is removahedral if its quotient fan~$\Fan_\equiv$ is.
The following example shows that some lattice congruence are not removahedral.

\begin{example}
Consider the congruence~$\equiv$ of~$\fS_4$ whose shard ideal is~$\shards_\equiv = \shards_4 \ssm \{\shard(1,4,\{2\})\}$.
Since the only removed shard contains no ray in its interior, the rays of the quotient fan~$\Fan_\equiv$ are all rays of the braid fan~$\Fan_\equiv$, so that the corresponding removahedron is the permutahedron~$\Perm[4]$ which does not realize~$\Fan_\equiv$.
\end{example}


\subsection{Permutrees}
\label{subsec:permutrees}

We now recall the permutrees of~\cite{PilaudPons-permutrees} that generalize the binary trees and will be especially important in this paper.
The presentation and pictures are borrowed from~\cite{PilaudPons-permutrees}.

\subsubsection{Combinatorics of permutrees}

In an oriented tree~$T$, we call \defn{parents} (resp.~\defn{children}) of a node~$j$ the outgoing (resp.~incoming) neighbors of~$j$, and \defn{ancestor} (resp.~\defn{descendant}) subtrees of~$j$ the connected components of the parents (resp.~children) of~$j$ in~$T \ssm \{j\}$.
A \defn{permutree} is an oriented tree~$T$ with nodes~$[n]$, such that 
\begin{itemize}
\item any node has either one or two parents and either one or two children,
\item if node~$j$ has two parents (resp.~children), then all nodes in its left ancestor (resp.~descendant) subtree are smaller than~$j$ while all nodes in its right ancestor (resp.~descendant) subtree are larger than~$j$.
\end{itemize}
\cref{fig:permutrees} provides four examples of permutrees.
We use the following conventions:
\begin{itemize}
\item All edges are oriented bottom-up and the nodes appear from left to right in natural order.
\item We decorate the nodes with \noneCirc{}, \downCirc{}, \upCirc{}, \upDownCirc{} depending on their number of parents and children. The sequence of these symbols, from left to right, is the \defn{decoration}~$\decoration$ of~$T$.
\item We draw a vertical red wall below (resp.~above) the nodes of~$\decoration^- \eqdef \set{j \in [n]}{\decoration_j = \downCirc{} \text{ or } \upDownCirc{}}$ (resp.~of~$\decoration^+ \eqdef \set{j \in [n]}{\decoration_j = \upCirc{} \text{ or } \upDownCirc{}}$) to mark the separation between the left and right descendant (resp.~ancestor) subtrees.
\end{itemize}
\begin{figure}[b]
	\capstart
	\centerline{
	\includegraphics[scale=.8]{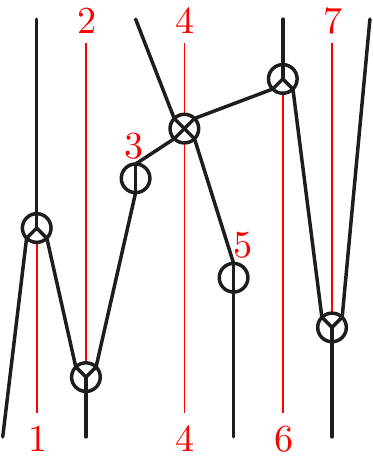} \qquad
	\includegraphics[scale=.8]{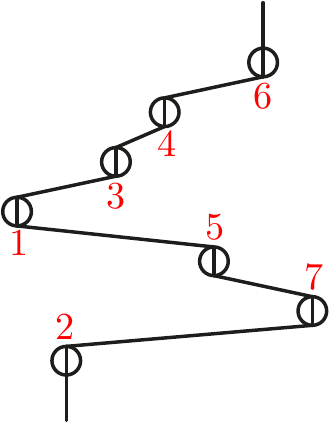} \qquad
	\includegraphics[scale=.8]{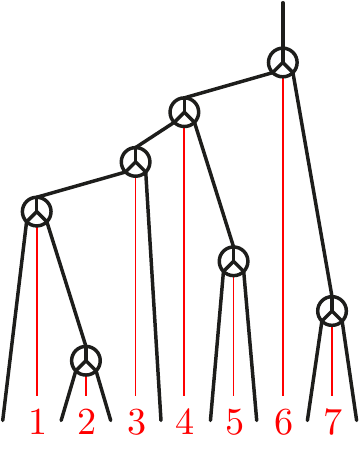} \qquad
	\includegraphics[scale=.8]{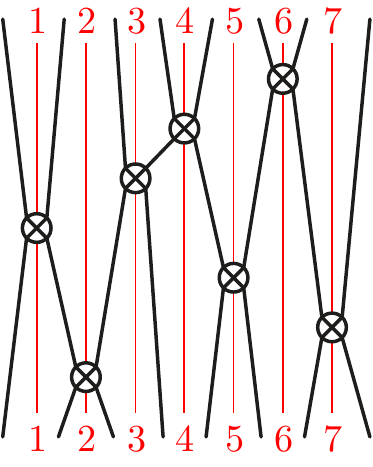}
	}
	\caption{Four examples of permutrees. While the first is generic, the last three use decorations corresponding to permutations, binary trees, and binary sequences. \mbox{\cite[Fig.~2 \& 3]{PilaudPons-permutrees}}}
	\label{fig:permutrees}
\end{figure}

As illustrated in \cref{fig:permutrees}, $\decoration$-permutrees extend and interpolate between various combinatorial families, including permutations when~$\decoration = \noneCirc^n$, binary trees when~$\decoration = \downCirc^n$, and binary sequences when~$\decoration = \upDownCirc^n$.
In fact, permutrees arose by pushing further the combinatorics of Cambrian trees developed in~\cite{ChatelPilaud} to provide combinatorial models to the type~$A$ Cambrian lattices~\cite{Reading-CambrianLattices}.

An \defn{edge cut} in a permutree~$T$ is the partition~$\edgecut{I}{J}$ of the nodes of~$T$ into the set~$I$ of nodes in the source set and the set~$J = [n] \ssm I$ of nodes in the target set of an oriented edge~of~$T$.
Edge cuts play an important role in the geometry of the permutree fan defined below.

As with the classical rotation operation on binary trees, there is a local operation on $\decoration$-permutrees which only exchanges the orientation of an edge and rearranges the endpoints of two other edges.
Namely, consider an edge~$i \to j$ in a $\decoration$-permutree~$T$.
Let~$D$ denote the only (resp.~the right) descendant subtree of node~$i$ if~$i \notin \decoration^-$ (resp.~if~$i \in \decoration^-$) and let~$U$ denote the only (resp.~the left) ancestor subtree of node~$j$ if~$j \notin \decoration^+$ (resp.~if~$j \in \decoration^+$).
Let~$S$ be the oriented tree obtained from~$T$ by just reversing the orientation of~$i \to j$ and attaching the subtree~$U$ to~$i$ and the subtree~$D$ to~$j$.
The transformation from~$T$ to~$S$ is the \defn{rotation} of the edge~$i \to j$.

\begin{proposition}[\cite{PilaudPons-permutrees}]
\label{prop:rotationPermutree}
The result~$S$ of the rotation of an edge~$i \to j$ in a $\decoration$-permutree~$T$ is a \mbox{$\decoration$-permutree}. Moreover, $S$ is the unique $\decoration$-permutree with the same edge cuts as~$T$, except the cut defined by~$i \to j$.
\end{proposition}

Define the \defn{increasing rotation graph} as the directed graph whose vertices are the $\decoration$-permutrees and whose arcs are increasing rotations~$T \to S$, \ie where the edge~$i \to j$ in~$T$ is reversed to the edge~$i \leftarrow j$ in~$S$ for~$i < j$. See \cref{fig:permutreeLattices}.

\begin{proposition}[\cite{PilaudPons-permutrees}]
\label{prop:permutreeLattice}
The transitive closure of the increasing rotation graph on~$\decoration$-permutrees is a lattice, called \defn{$\decoration$-permutree lattice}.
\end{proposition}

The $\decoration$-permutree lattice specializes to the weak order when~$\decoration = \noneCirc^n$, the Tamari lattice when~$\decoration = \downCirc^n$, the Cambrian lattices~\cite{Reading-CambrianLattices} when~$\decoration \in \{\downCirc, \upCirc\}^n$ and the boolean lattice when~${\decoration = \upDownCirc^n}$.
\begin{figure}
	\capstart
	\centerline{\includegraphics[scale=.5]{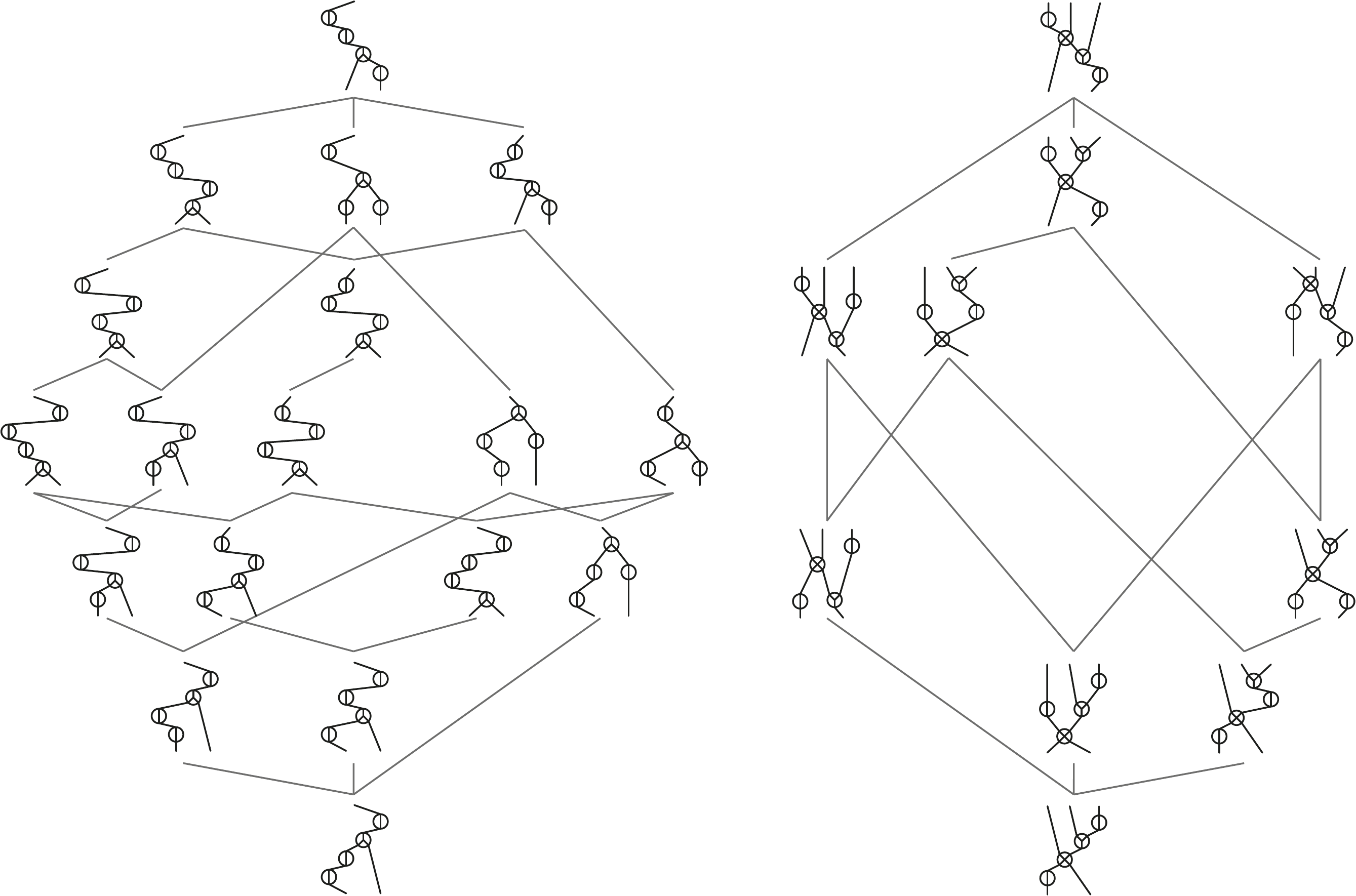}}
	\caption{The $\decoration$-permutree lattices for $\decoration \! = \! \noneCirc{}\noneCirc{}\downCirc{}\noneCirc{}$ (left) and $\decoration \! = \! \noneCirc{}\upDownCirc{}\upCirc{}\noneCirc{}$~(right).~\mbox{\cite[Fig.~11]{PilaudPons-permutrees}}}
	\label{fig:permutreeLattices}
\end{figure}

In fact, all permutree lattices are lattice quotients of the weak order, exactly as the Tamari lattice is the quotient of the weak order by the sylvester congruence.
The \defn{$\decoration$-permutree congruence}~$\equiv_\decoration$ is the equivalence relation on~$\fS_n$ defined equivalently as:
\begin{itemize}
\item the relation whose equivalence classes are the sets of linear extensions of the $\decoration$-permutrees,
\item the relation whose equivalence classes are the fibers of the $\decoration$-permutree insertion, similar to the binary tree insertion and presented in detail in~\cite{PilaudPons-permutrees},
\item the transitive closure of the rewriting rules~$U i k V j W \equiv_\decoration U k i V j W$ when~$j \in \decoration^-$ and $U j V i k W \equiv_\decoration U j V k i W$ when $j \in \decoration^+$, for some letters~$i < j < k$ and words~$U,V,W$ on~$[n]$,
\item the congruence with ideal~$\shards_\decoration \eqdef \set{\shard(i,j,S)}{1 \le i < j \le n, \; \decoration^- \cap {]i,j[} \subseteq S \subseteq {]i,j[} \ssm \decoration^+}$. In other words, the corresponding arcs do not pass below any~$k \in \decoration^-$ nor above any~$k \in \decoration^+$.
\end{itemize}
This extends the trivial congruence when~$\decoration = \noneCirc^n$, the sylvester congruence~\cite{HivertNovelliThibon-algebraBinarySearchTrees} when~${\decoration = \downCirc^n}$, the Cambrian congruences~\cite{Reading-latticeCongruences, Reading-CambrianLattices, ChatelPilaud} when~$\decoration \in \{\downCirc, \upCirc\}^n$, and the hypoplactic congruence~\cite{KrobThibon-NCSF4, Novelli-hypoplactic} when~$\decoration = \upDownCirc^n$.
The corresponding shard ideals are represented in \cref{fig:arcIdealsPermutrees}.

\begin{figure}
	\capstart
	\centerline{
	\includegraphics[scale=1.35]{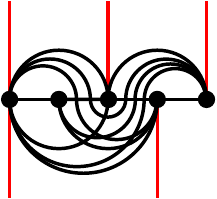} \qquad
	\includegraphics[scale=1.35]{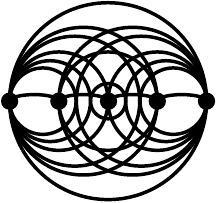} \qquad
	\includegraphics[scale=1.35]{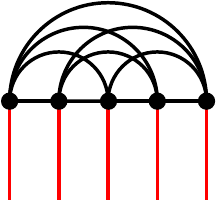} \qquad
	\includegraphics[scale=1.35]{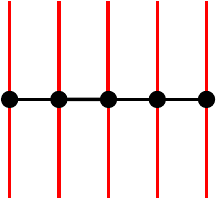}
	}
	\caption{Four examples of permutree shard ideals. While the first is generic, the last three use decorations corresponding to permutations, binary trees, and binary sequences. \mbox{\cite[Fig.~3]{Pilaud-arcDiagramAlgebra}}}
	\label{fig:arcIdealsPermutrees}
\end{figure}

\begin{proposition}[\cite{PilaudPons-permutrees}]
The $\decoration$-permutree congruence~$\equiv_\decoration$ is a lattice congruence of the weak order and the quotient~$\fS_n / {\equiv_\decoration}$ is (isomorphic to) the $\decoration$-permutree lattice.
\end{proposition}

Note that the description of the shard ideal~$\shards_\decoration$ above gives the following characterization of permutree congruences.

\begin{proposition}
\label{prop:characterizationPermutreeCongruences}
A lattice congruence~$\equiv$ is a permutree congruence if and only if all generators of the lower ideal~$\shards_n \ssm \shards_\equiv$ are of length~$2$, \ie of the form~$\shard(i-1, i+1, \varnothing)$ or~$\shard(i-1, i+1, \{i\})$.
\end{proposition}

We conclude these recollections on combinatorics of permutrees with a natural order on all permutree congruences.
For two decorations ${\decoration, \decoration' \in \Decorations^n}$, we say that~$\decoration$ \defn{refines}~$\decoration'$ and we write~$\decoration \cle \decoration'$ if $\decoration_i \cle \decoration'_i$ for all~$i \in [n]$ for the order~$\noneCirc{} \cle \{\downCirc{}, \upCirc{}\} \cle \upDownCirc{}$.
In this case, the $\decoration$-permutree congruence refines the $\decoration'$-permutree congruence: $\sigma \equiv_\decoration \tau$ implies~$\sigma \equiv_{\decoration'} \tau$ for any two permutations~$\sigma, \tau \in \fS_n$.

\subsubsection{Geometry of permutrees}
\label{subsubsec:geometryPermutrees}

We finally recall the geometric constructions of~\cite{PilaudPons-permutrees} realizing the $\decoration$-permutree lattice.

The \defn{$\decoration$-permutree fan}~$\Fan_\decoration$ is the quotient fan of the $\decoration$-permutree congruence.
It has 
\begin{itemize}
\item a chamber~$C(T)$ for each $\decoration$-permutree~$T$, which can be defined either as the union of the chambers~$C(\sigma)$ for all linear extensions~$\sigma$ of~$T$, or by the inequalities~$\b{x}_i \le \b{x}_j$ for all edges~$i \to j$ in~$T$, or by the generators~$|I| \one_J - |J| \one_I$ for all edge cuts~$\edgecut{I}{J}$ of~$T$,
\item a ray~$C(I)$ for each proper subset~$\varnothing \ne I \subsetneq [n]$ such that for all~$a < b < c$, if~$a,c \in I$ then~$b \notin \decoration^- \ssm I$, and if~$a,c \notin I$ then~$b \notin \decoration^+ \cap I$. We denote by~$\cI_\decoration$ the collection of these proper subsets. See \cref{prop:raysPermutreeFan}.
\end{itemize}
Two examples of $\decoration$-permutree fans are represented in \cref{fig:permutreeFan}.
\begin{figure}[t]
	\capstart
	\centerline{\includegraphics[scale=.5]{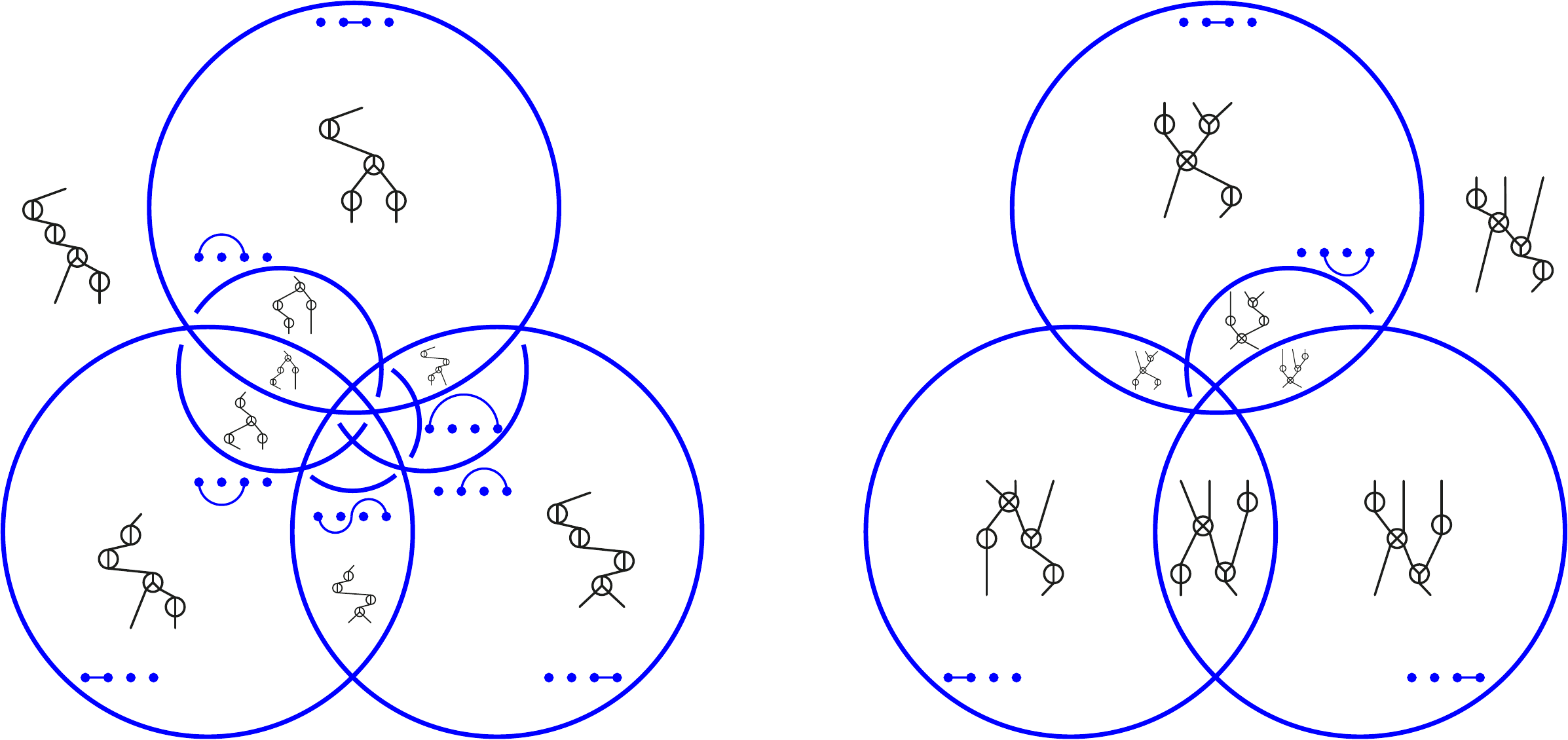}}
	\caption{The permutree Fans~$\Fan_{\noneCirc{}\noneCirc{}\downCirc{}\noneCirc{}}$ (left) and~$\Fan_{\noneCirc{}\upDownCirc{}\upCirc{}\noneCirc{}}$ (right). Each shard is labeled with its coresponding arc, and some chambers are labeled with their corresponding permutrees.}
	\label{fig:permutreeFan}
\end{figure}
The $\decoration$-permutree fan~$\Fan_\decoration$ specializes to the braid fan when~$\decoration = \noneCirc{}^n$, the (type~$A$) Cambrian fans of N.~Reading and D.~Speyer~\cite{ReadingSpeyer} when~$\decoration \in \{\downCirc{}, \upCirc{}\}^n$, the fan defined by the hyperplane arrangement~$\b{x}_i = \b{x}_{i+1}$ for each~$i \in [n-1]$ when~$\decoration = \upDownCirc{}^n$, and the fan defined by the hyperplane arrangement~$\b{x}_i = \b{x}_j$ for each~$1 \le i < j \le n$ such that~$\decoration_k = \noneCirc{}$ for all~$i < k < j$ when~$\decoration \in \{\noneCirc{},\upDownCirc{}\}^n$.
Decoration refinements translate to permutree fan refinements: if~$\decoration \cle \decoration'$, then the $\decoration$-permutree fan~$\Fan_\decoration$ refines the $\decoration'$-permutree~fan~$\Fan_{\decoration'}$.

\enlargethispage{.3cm}
The \defn{$\decoration$-permutreehedron}~$\Permutreehedron$ is the polytope defined equivalently as:
\begin{itemize}
\item the convex hull of the points~$\smash{\sum_{j \in [n]} \big( 1 + d(T,j) + \underline{\ell}(T,j) \, \underline{r}(T,j) - \overline{\ell}(T,j) \, \overline{r}(T,j) \big) \, \b{e}_j}$ for all $\decoration$-permutrees~$T$, where $d(T,j), \underline{\ell}(T,j), \underline{r}(T,j), \overline{\ell}(T,j), \overline{r}(T,j)$ respectively denote the numbers of nodes in the descendant, left descendant, right descendant, left ancestor, right ancestor subtrees of~$j$ in~$T$,
\item the intersection of the hyperplane~$\Hyp$ with the halfspaces~$\bigset{\b{x} \in \R^n}{\sum_{i \in I} \b{x}_i \ge \binom{|I|+2}{2}}$ for all subsets~$I$ in~$\cI_\decoration$.
\end{itemize}

Two examples of $\decoration$-permutreehedra are represented in \cref{fig:permutreehedron}.
\begin{figure}[t]
	\capstart
	\centerline{\includegraphics[scale=.5]{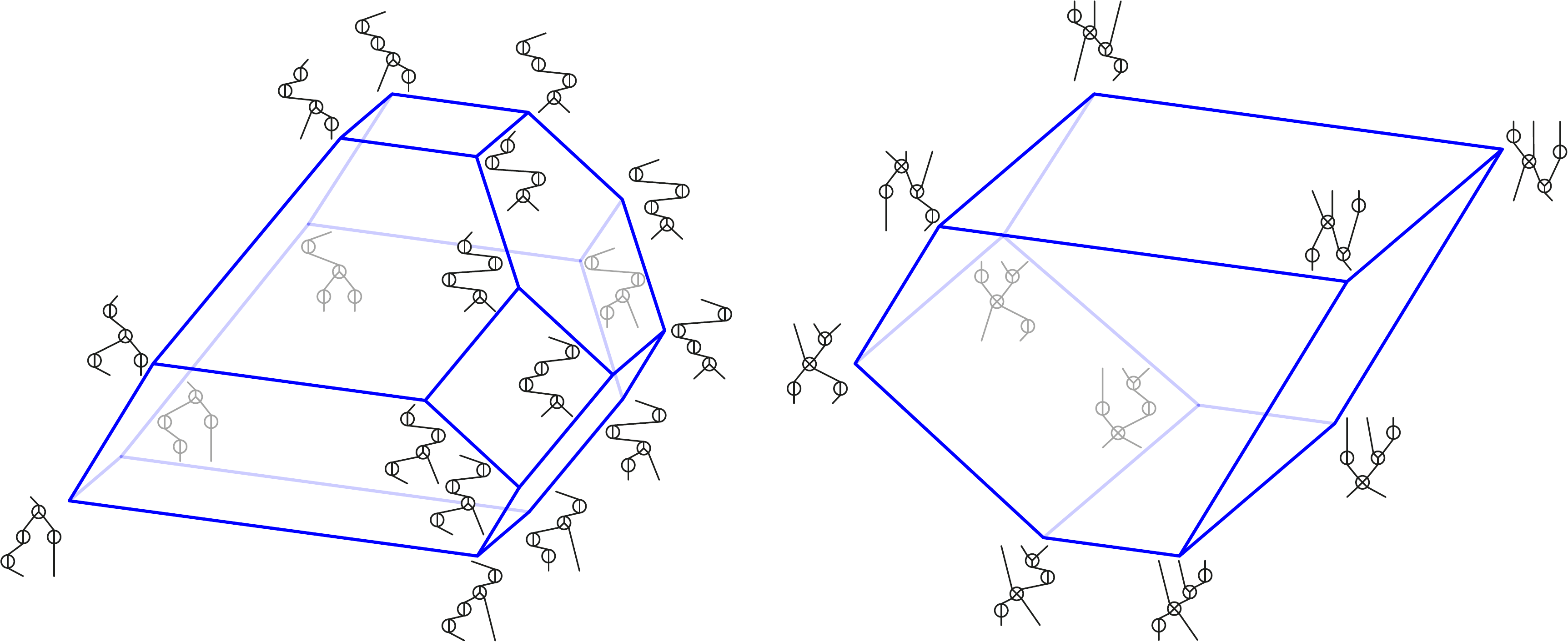}}
	\caption{The permutreehedra~$\Permutreehedron[\noneCirc{}\noneCirc{}\downCirc{}\noneCirc{}]$ (left) and~$\Permutreehedron[\noneCirc{}\upDownCirc{}\upCirc{}\noneCirc{}]$ (right). \cite[Fig.~15]{PilaudPons-permutrees}}
	\label{fig:permutreehedron}
\end{figure}
The $\decoration$-permutreehedron~$\Permutreehedron$ specializes to the permutahedron~$\Perm$ when~${\decoration = \noneCirc{}^n}$, J.-L.~Loday's associahedron~$\Asso$~\cite{ShniderSternberg, Loday} when~$\decoration = \downCirc{}^n$, C.~Hohlweg and C.~Lange's associahedra~$\Asso[\decoration]$~\cite{HohlwegLange, LangePilaud} when~${\decoration \in \{\downCirc{}, \upCirc{}\}^n}$, the parallelepiped with directions~$\b{e}_i - \b{e}_{i+1}$ for each~$i \in [n-1]$ when~$\decoration = \upDownCirc{}^n$, the graphical zonotope~$\Zono[\decoration]$ generated by the vectors~$\b{e}_i - \b{e}_{j}$ for each~$1 \le i < j \le n$ such that~$\decoration_k = \noneCirc$ for all~$i < k < j$ when~$\decoration \in \{\noneCirc{},\upDownCirc{}\}^n$.

\begin{figure}[t]
	\capstart
	\centerline{\includegraphics[width=\textwidth]{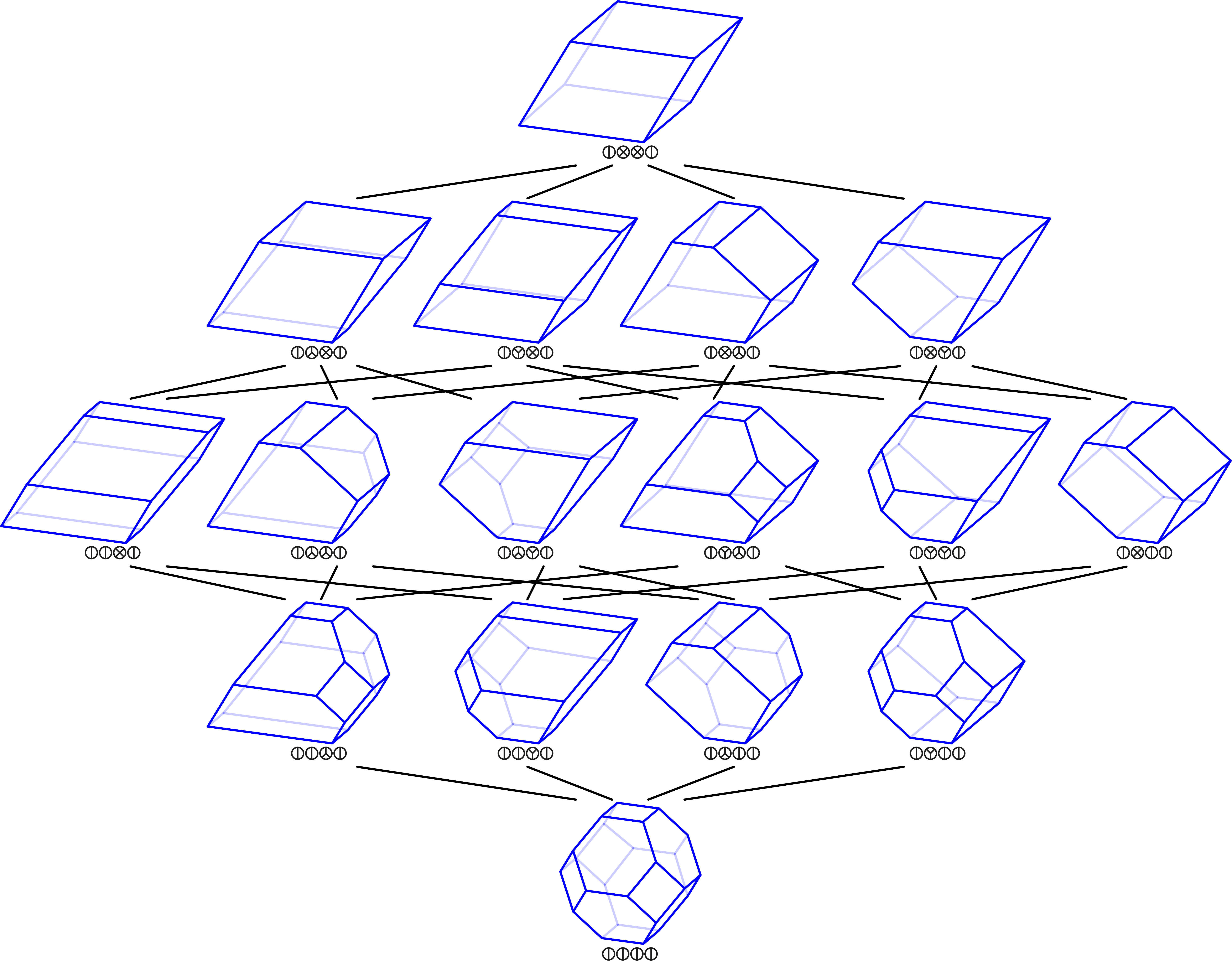}}
	\caption{The $\decoration$-permutreehedra, for all decorations~$\decoration \in \noneCirc{} \cdot \Decorations^2 \cdot \noneCirc{}$. \cite[Fig.~16]{PilaudPons-permutrees}}
	\label{fig:permutreehedra}
\end{figure}

\begin{theorem}[\cite{PilaudPons-permutrees}]
\label{thm:permutreehedron}
The $\decoration$-permutree fan~$\Fan_\decoration$ is the normal fan of the $\decoration$-permutreehedron~$\Permutreehedron$.
\end{theorem}

Note that decoration refinements translate to permutreehedra inclusions: if~$\decoration \cle \decoration'$, then the permutreehedron~$\Permutreehedron[\decoration']$ is obtained by deleting inequalities in the facet description of the permutreehedron~$\Permutreehedron[\decoration']$, so that~$\Permutreehedron[\decoration] \subseteq \Permutreehedron[\decoration']$ as illustrated in \cref{fig:permutreehedra}.
In particular, all facet-defining inequalities of the permutreehedron~$\Permutreehedron$ are facet-defining inequalities of the permutahedron~$\Perm$, which can be rephrased as follows.

\begin{corollary}
\label{coro:permutreeCongruencesRemovahedral}
All permutree congruences are removahedral.
\end{corollary}


\section{Removahedral congruences}
\label{sec:removahedralCongruences}

The main goal of this section is to show \cref{thm:main}.
An important step of the proof is the description of the rays of the quotient fans, provided in \cref{subsec:raysQuotientFan}.
While elementary, we are not aware that a characterization of these rays appeared in the literature.
We then proceed to prove \cref{thm:main}\,(i) in \cref{subsec:removahedralCongruences} and \cref{thm:main}\,(ii) in \cref{subsec:permutreesStronglyRemovahedral}.


\subsection{Rays of the quotient fan}
\label{subsec:raysQuotientFan}

Fix a ray~$\ray(I)$ of the braid fan~$\Fan_n$ corresponding to a proper subset~$\varnothing \ne I \subsetneq [n]$.
The shards of~$\shards_n$ containing $\ray(I)$ were characterized by a simple combinatorial criterion in~\cite[Lem.~5]{PilaudSantos-quotientopes}.
Here, we need to characterize which shards of~$\shards_n$ contain~$\ray(I)$ in their relative interior.
We associate with~$I$ the set~$\shards_I$ of $n-2$ shards containing the $|I|-1$ down shards joining two consecutive elements of~$I$ and the $n-|I|-1$ up shards joining two consecutive elements of~$[n] \ssm I$, that is
\[
\shards_I \eqdef \bigset{\shard(i, j, \varnothing)}{i, j \in I \text{ and } {]i,j[} \cap I = \varnothing} \cup \bigset{\shard(i, j, ]i,j[)}{i, j \notin I \text{ and } {]i,j[} \subseteq I}.
\]

\begin{lemma}
\label{lem:rayInInteriorShard}
The $n-2$ shards of~$\shards_I$ are the only shards containing~$\ray(I)$ in their relative interior.
\end{lemma}

\begin{proof}
Recall that
\begin{itemize}
\item the ray~$\ray(I)$ is given by~$|I|\one - n\one_I$, thus~$\b{x}_i < \b{x}_j$ for~$i \in I$ and~$j \notin I$,
\item the relative interior of~$\shard(i,j,S)$ is given by~$\b{x}_h < \b{x}_i = \b{x}_j < \b{x}_k$ for all~$h \in S$ and~${k \notin {]i,j[} \ssm S}$.
\end{itemize}
Therefore, $\ray(I)$ satisfies these inequalities if and only if either~$i,j \in I$, $S = \varnothing$ and~$]i,j[ \cap I = \varnothing$, or $i,j \notin I$, $S = {]i,j[}$, and~${]i,j[} \subseteq I$.
Therefore, $\ray(I)$ is contained in the relative interior of~$n-2$ shards.
\end{proof}

Note that this implies that a mixed shard (\ie of the form~$\shard(i,j,S)$ with~$S \notin \{\varnothing, {]i,j[}\}$, or said differently whose arc crosses the horizontal axis) contains no ray of the braid fan in its relative interior. See for example the shards~$\shard(1,4,\{2\})$ and~$\shard(1,4,\{3\})$ in \cref{fig:shards4}.

From \cref{lem:rayInInteriorShard}, we derive the following description of the rays of the quotient fan~$\Fan_\equiv$.
As an illustration, this description is specialized to permutree fans in \cref{prop:raysPermutreeFan}.

\begin{lemma}
\label{lem:raysQuotientFan}
The ray~$\ray(I)$ is a ray of the quotient fan~$\Fan_\equiv$ if and only if~$\shards_\equiv$ contains $\shards_I$.
\end{lemma}

\begin{proof}
If $\shards_\equiv$ contains~$\shards_I$, then the quotient fan~$\Fan_\equiv$ contains $n-2$ shards which intersect along~$\ray(I)$, so that~$\ray(I)$ is a ray of~$\Fan_\equiv$.
The converse can be derived from~\cite[Prop.~5.10]{Reading-HopfAlgebras} or~\cite[Prop.~7.7]{Reading-shardIntersectionOrder}.
Let us just provide a sketchy argument.
Let~$\fS_I$ be the interval of permutations~$\sigma$ of~$[n]$ such that cone~$C(\sigma)$ contains~$\ray(I)$ (or equivalently $\sigma([|I|]) = I$).
Let~$\equiv_I$ denote the subcongruence of~$\equiv$ induced by~$\fS_I$.
The basic shards of~$\equiv_I$ are the shards of~$\shards_I$.
Since~$\shards_\equiv$ does not contain~$\shards_I$, the congruence~$\equiv_I$ is not essential, so that~$\ray(I)$ is not a ray~of~$\Fan_\equiv$.
\end{proof}


\subsection{Removahedral congruences}
\label{subsec:removahedralCongruences}

We are now ready to prove the following statement, which is a more precise reformulation of \cref{thm:main}\,(i).

\begin{theorem}
\label{thm:removahedralCongruences}
The only essential removahedral congruences are the permutree congruences.
\end{theorem}

Note that we focus here on essential congruences.
However, as already mentioned, non-essential congruences can be understood from their essential restrictions.
For arbitrary congruences, \cref{thm:removahedralCongruences} says that the essential restrictions of removahedral congruences are permutree congruences.

We learned from \cref{coro:permutreeCongruencesRemovahedral} that permutree congruences are essential removahedral congruences, and we want to prove the opposite direction.
We thus assume by contradiction that there exists an essential removahedral congruence~$\equiv$ which is not a permutree congruence.

\begin{lemma}
The generating set~$\b{\Pi}_\equiv$ of the lower ideal~$\shards_n \ssm \shards_\equiv$ contains at least one shard of the form~$(i, j, \varnothing)$ or~$(i, j, ]i,j[)$ for~$i \le j-3$.
\end{lemma}

\begin{proof}
Since~$\equiv$ is essential, $\b{\Pi}_\equiv$ contains no shard of length~$1$.
Since~$\equiv$ is not a permutree congruence, $\b{\Pi}_\equiv$ must contain a shard of length distinct from~$2$ by \cref{prop:characterizationPermutreeCongruences}.
Decompose the set of shards~$\b{\Pi}_\equiv$ into the subset~$\b{\Pi}_\equiv^{=2}$ of shards of length~$2$ and the subset~$\b{\Pi}_\equiv^{>2}$ of shards of length greater than~$3$.
Let~$\equiv_\decoration$ denote the permutree congruence defined by~$\b{\Pi}_{\equiv_\decoration} = \b{\Pi}_\equiv^{=2}$.
If~$\b{\Pi}_\equiv^{>2}$ contains only mixed shards, then $\equiv$ and~$\equiv_\decoration$ contain the same up and down shards, since mixed shards only force mixed shards.
This implies by \cref{lem:raysQuotientFan} that the quotient fans~$\Fan_\equiv$ and~$\Fan_\decoration$ have the same rays.
The corresponding removahedron is thus the permutreehedron~$\Permutreehedron$ which does not realize the quotient fan~$\Fan_\equiv$ so that~$\equiv$ is not removahedral.
The statement follows.
\end{proof}

We assume now that there are~$i \le j-3$ such that the shard~$\shard(i, j, ]i,j[)$ is one of the generators of the lower ideal~$\shards_n \ssm \shards_\equiv$. 
The proof for the shard~$\shard(i,j,\varnothing)$ is symmetric.
We consider the following five subsets of~$[n]$:
\[
I = {]i+1,j[},
\qquad
J = {]i, j-1[},
\qquad
K = {[1,j[}
\qquad
L = {]i,n]}
\qquad\text{and}\qquad
M = {]i+1, j-1[}.
\]
Note that~$M$ might be empty, in which case it will not appear in all the computations below.
Before showing that these subsets provide a contradiction, let us consider a minimal example.

\begin{example}
\label{exm:pbRays}
Consider the lattice congruence~$\equiv$ for~$n = 4$ whose only deleted shard is~$\shard(1,4,\{2,3\})$.
Consider the subsets
\[
I = \{3\},
\qquad
J = \{2\},
\qquad
K = \{1,2,3\}
\qquad\text{and}\qquad
L = \{2,3,4\}.
\]
(Here, $M = \varnothing$ is irrelevant.)
The cones~$C \eqdef \R_{\ge0}\{\ray(I), \ray(K), \ray(L)\}$ and~$D \eqdef \R_{\ge0}\{\ray(J), \ray(K), \ray(L)\}$ are two adjacent chambers of the quotient fan~$\Fan_\equiv$, whose corresponding linear dependence is ${\ray(I) + \ray(J) = \ray(K) + \ray(L)}$.
However, the height function~$h_\circ(I) = 4|I|(4-|I|)/2$ of the permutahedron~$\Perm[4]$ satisfies~$h_\circ(I) + h_\circ(J) = h_\circ(K) + h_\circ(L)$, thus violating the wall-crossing inequality~$h_\circ(I) + h_\circ(J) > h_\circ(K) + h_\circ(L)$.
Therefore, the lattice congruence~$\equiv$ is not removahedral.
\end{example}

We now come back to the general situation.
We need the following three observations.

\begin{lemma}
\label{lem:raysAreHere}
The rays~$\ray(I)$, $\ray(J)$, $\ray(K)$, $\ray(L)$, and $\ray(M)$ are all rays of the quotient fan~$\Fan_\equiv$.
\end{lemma}

\begin{proof}
By \cref{lem:rayInInteriorShard}, the only non-basic shards containing these rays in their interior are
\begin{itemize}
\item $\shard(i+1, j, ]i+1, j[)$ for~$I$,
\item $\shard(i, j-1, ]i, j-1[)$ for~$J$,
\item $\shard(i+1, j-1, ]i+1, j-1[)$ for~$M$.
\end{itemize}
Since all these shards are in~$\shards_\equiv$ by minimality of~$\shard(i, j, ]i,j[)$, the result follows by \cref{lem:raysQuotientFan}.
\end{proof}

\begin{lemma}
\label{lem:raysAreInAdjacentChambers}
The quotient fan~$\Fan_\equiv$ contains two chambers~$C$ and~$D$ adjacent along the hyperplane~$\b{x}_{i+1} = \b{x}_{j-1}$, such that~$\ray(I) \in C$ while $\ray(J) \in D$ and~$\{\ray(K), \ray(L), \ray(M)\} \subset C \cap D$.
\end{lemma}

\begin{proof}
Observe that for two proper subsets~$\varnothing \ne I, J \subsetneq [n]$, the rays~$\ray(I)$ and~$\ray(J)$ are separated by the hyperplane of equation~$\b{x}_i = \b{x}_j$ if and only if~$i \in I \ssm J$ and~$j \in J \ssm I$ or \viceversa.
Therefore, $I \subseteq J$ implies that~$\ray(I)$ and~$\ray(J)$ belong to a common chamber of the braid arrangement~$\Fan_n$ (in fact to the chamber~$C(\sigma)$ for any permutation~$\sigma$ such that~$\sigma([|I|]) = I$ and~$\sigma([|J|]) = J$).
Since~$M \subseteq I, J \subseteq K, L$, the only rays separated by an hyperplane are:
\begin{itemize}
\item $\ray(I)$ and~$\ray(J)$ which are separated by the hyperplane of equation~$\b{x}_{i+1} = \b{x}_{j-1}$,
\item $\ray(K)$ and~$\ray(L)$ which are separated by all hyperplanes of equation~$\b{x}_k = \b{x}_\ell$ for~$k \le i$ and~$\ell \ge j$. However, $\ray(K)$ and~$\ray(L)$ belong to the same chamber of the quotient fan~$\Fan_\equiv$ since we remove all shards~$\shard(k, \ell, ]k,\ell[)$ for~$k \le i$ and~$\ell \ge j$.
\end{itemize}

For completeness, let us provide a less intuitive but more formal alternative argument.
Consider a sequence of permutations starting with~$\sigma \eqdef [i+2, \dots, j-2, j-1, i+1, i, \dots, 1, j, \dots, n]$ and ending with~$\tau \eqdef [i+2, ..., j-2, j-1, i+1, j, ..., n, i, ..., 1]$, and obtained by transposing at each step two values~$k \le i$ and~$\ell \ge j$ at consecutive positions.
In other words, all the permutations in the sequence start by~$[i+2, \dots, j-2, j-1, i+1]$ and end with a shuffle of~$[i, \dots, 1]$ with~$[j, \dots, n]$.
At each step, the interval~$]k,\ell[$ between the two transposed values always appears before the position of the transposition.
Therefore, the chambers corresponding to the two permutations before and after the transposition are separated by the shard~$\shard(k,\ell,]k,\ell[)$, which does not belong to~$\shards_\equiv$ since it is forced by~$\shard(i,j,]i,j[)$.
It follows that the cones of all these permutations, and in particular those of~$\sigma$ and~$\tau$, belong to the same chamber~$C$ of the quotient fan~$\Fan_\equiv$.
This chamber~$C$ contains the rays~$\ray(I)$, $\ray(K)$, $\ray(L)$ and~$\ray(M)$ since the subsets~$I$ and~$M$ (resp.~$K$, resp.~$L$) are initial intervals of all permutations in the sequence (resp.~of~$\sigma$, resp.~of~$\tau$).
We prove similarly that~$\ray(J)$, $\ray(K)$, $\ray(L)$ and~$\ray(M)$ belong to a chamber~$D$ by considering a sequence of permutations starting with~$[i+2, \dots, j-2, i+1, j-1, i, \dots, 1, j, \dots, n]$ and ending with ${[i+2, \dots, j-2, i+1, j-1, j, \dots, n, i, \dots, 1]}$.
\end{proof}

\begin{lemma}
\label{lem:pbHeights}
Let~$h_\circ(I) = n|I|(n-|I|)/2$ be the height function of the permutahedron~$\Perm$.~Then
\begin{itemize}
\item $\ray(I) + \ray(J) = \ray(K) + \ray(L) + \ray(M)$,
\item $h_\circ(I) + h_\circ(J) - h_\circ(K) - h_\circ(L) - h_\circ(M) = ni(j-n) + n(1-i) \le 0$.
\end{itemize}
\end{lemma}

\begin{proof}
Immediate computations from the cardinalities
\[
|I| = |J| = j-i-2,
\qquad
|K| = j-1,
\qquad
|L| = n-i
\qquad\text{and}\qquad
|M| = j-i-3.
\qedhere
\]
\end{proof}

Observe that the inequality~$h_\circ(I) + h_\circ(J) \le h_\circ(K) + h_\circ(L) + h_\circ(M)$ is an equality if and only if~$i = 1$ and~$j = n$, as was the case in \cref{exm:pbRays}.
We have now all ingredients to conclude the proof of \cref{thm:removahedralCongruences}.

\begin{proof}[Proof of \cref{thm:removahedralCongruences}]
Combining \cref{lem:raysAreHere,lem:raysAreInAdjacentChambers,lem:pbHeights}, we obtain that the height function~$h_\circ$ of the permutahedron violates the wall-crossing inequality~$h_\circ(I) + h_\circ(J) > h_\circ(K) + h_\circ(L) + h_\circ(M)$ of \cref{prop:characterizationPolytopalFanNonSimplicial} corresponding to the linear dependence~$\ray(I) + \ray(J) = \ray(K) + \ray(L) + \ray(M)$ between the rays of the adjacent chambers~$C$ and~$D$ of the quotient fan~$\Fan_\equiv$, so that~$\equiv$ is not removahedral.
\end{proof}


\subsection{Permutree congruences are strongly removahedral}
\label{subsec:permutreesStronglyRemovahedral}

We now prove that the permutree congruences are removahedral in a stronger sense.
Namely, we show that we obtain polytopes realizing the permutree fans by deleting inequalities in the facet description of any polytope realizing the braid fan, not only the classical permutahedron~$\Perm$, as stated in \cref{thm:main}\,(ii).
We start by a structural observation on the exchanges in permutree fans.

\begin{proposition}
\label{prop:wallCrossingInequalitiesPermutreeFan}
Consider two adjacent chambers~$\R_{\ge0}\rays$ and~$\R_{\ge0}\rays[S]$ of the $\decoration$-permutree fan~$\Fan_\decoration$ with~$\rays \ssm \rays[S] = \{\ray(I)\}$ and~$\rays[S] \ssm \rays = \{\ray(J)\}$.
Then the rays~$\ray(I \cap J)$ and~$\ray(I \cup J)$ are also rays of the $\decoration$-permutree fan~$\Fan_\decoration$ and belong to~$\rays \cap \rays[S]$.
Therefore, all wall-crossing inequalities of the $\decoration$-permutree fan~$\Fan_\decoration$ are of the form
\[
h(I) + h(J) > h(I \cap J) + h(I \cup J),
\]
with the usual convention that~$h(\varnothing) = h([n]) = 0$.
\end{proposition}

\vspace{-.5cm}
\parpic(5.8cm,2.7cm)(0pt,60pt)[rb]{\quad	\begin{tikzpicture}[baseline=-.3cm, scale=.9]
		\node[circle, draw, inner sep=2pt] (i) at (-.5,-.5) {$i$};
		\node[circle, draw, inner sep=1pt] (j) at (.5,.5) {$j$};
		\node[inner sep=1pt] (U) at (0,1) {$U$};
		\node[inner sep=1pt] (D) at (0,-1) {$D$};
		\node[inner sep=1pt] (ld) at (-1,-1) {$\underline{L}$};
		\node[inner sep=1pt] (lu) at (-1,0) {$\overline{L}$};
		\node[inner sep=1pt] (rd) at (1,0) {$\underline{R}$};
		\node[inner sep=1pt] (ru) at (1,1) {$\overline{R}$};
		\draw (D) -- (i);
		\draw (ld) -- (i);
		\draw (i) -- (lu);
		\draw[red, thick] (i) -- (j);
		\draw (j) -- (U);
		\draw (rd) -- (j);
		\draw (j) -- (ru);
	\end{tikzpicture}
	$\underset{\text{rotation}}{\longleftrightarrow}$
	\begin{tikzpicture}[baseline=-.3cm, scale=.9]
		\node[circle, draw, inner sep=2pt] (i) at (-.5,.5) {$i$};
		\node[circle, draw, inner sep=1pt] (j) at (.5,-.5) {$j$};
		\node[inner sep=1pt] (U) at (0,1) {$U$};
		\node[inner sep=1pt] (D) at (0,-1) {$D$};
		\node[inner sep=1pt] (ld) at (-1,0) {$\underline{L}$};
		\node[inner sep=1pt] (lu) at (-1,1) {$\overline{L}$};
		\node[inner sep=1pt] (rd) at (1,-1) {$\underline{R}$};
		\node[inner sep=1pt] (ru) at (1,0) {$\overline{R}$};
		\draw (D) -- (j);
		\draw (rd) -- (j);
		\draw (j) -- (ru);
		\draw[red, thick] (j) -- (i);
		\draw (i) -- (U);
		\draw (ld) -- (i);
		\draw (i) -- (lu);
	\end{tikzpicture}
}
\begin{proof}
Consider the two $\decoration$-permutrees~$T$ and~$S$ whose chambers are~$C(T) = \R_{\ge0}\rays$ and~$C(S) = \R_{\ge0}\rays[S]$.
Let~${i \to j}$ denote the edge of~$T$ that is rotated to the edge~$j \to i$ in~$S$.
Up to swapping the roles of~$I$ and~$J$, we can assume that~$i < j$.
We denote by~$U$, $D$, $\underline{L}$, $\overline{L}$, $\underline{R}$, $\overline{R}$ the subtrees of~$T$ and~$S$ as illustrated on the figure on the right.
Note that some of the subtrees~$\underline{L}$, $\overline{L}$, $\underline{R}$, $\overline{R}$ might not exist if~$\decoration_i \ne \upDownCirc$ or~$\decoration_j \ne \upDownCirc$.
We then just assume that they are empty.
Since the rays of the cone~$C(T)$ are given by~$|I|\one_J - |J|\one_I$ for all edge cuts of~$\edgecut{I}{J}$ of~$T$, and the unique edge cut that differs from~$T$ to~$S$ is that corresponding to the edge~$i~\textbf{--}~j$ by \cref{prop:rotationPermutree}, we obtain that
\[
I = \{i\} \cup D \cup \underline{L} \cup \overline{L}
\qquad\text{and}\qquad
J = \{j\} \cup D \cup \underline{R} \cup \overline{R}
\]
and therefore that
\[
I \cap J = D
\qquad\text{and}\qquad
I \cup J = \{i,j\} \cup D \cup \underline{L} \cup \overline{L} \cup \underline{R} \cup \overline{R} = [n] \ssm U.
\]
But~$\edgecut{D}{[n] \ssm D}$ and~$\edgecut{[n] \ssm U}{U}$ are edge cuts in the $\decoration$-permutrees~$T$ and~$S$, so that the rays~$\ray(I \cap J)$ and~$\ray(I \cup J)$ are rays of the $\decoration$-permutree fan~$\Fan_\decoration$ and belong to~$\rays \cap \rays[S]$.
We therefore obtain that the unique (up to rescaling) linear dependence among the rays of~$\rays \cup \rays[S]$ is ${\ray(I) + \ray(J) = \ray(I \cap J) + \ray(I \cup J)}$.
The corresponding wall-crossing inequality is thus given by~$h(I) + h(J) > h(I \cap J) + h(I \cup J)$.
\end{proof}

\begin{corollary}
For any strictly submodular function~$h : 2^{[n]} \to \R$ (\ie such that~${h(\varnothing) = h([n]) = 0}$ and~$h(I) + h(J) > h(I \cap J) + h(I \cup J)$ for all~${I, J \subseteq [n]}$ with~$I \not\subseteq J$ and~$I \not\supseteq J$), and any decoration~$\decoration \in \Decorations^n$, the $\decoration$-permutree fan~$\Fan_\decoration$ is the normal fan of the polytope
\[
\Permutreehedron^h \eqdef \bigset{\b{x} \in \Hyp}{\dotprod{\ray(I)}{x} \le h(I) \text{ for all } I \in \cI_\decoration},
\]
where~$\cI_\decoration = \set{\varnothing \ne I \subsetneq [n]}{\ray(I) \text{ is a ray of } \Fan_\decoration}$ (see \cref{prop:raysPermutreeFan} for a characterization).
In other words, we obtain a polytope realizing the $\decoration$-permutree fan~$\Fan_\decoration$ by deleting inequalities in the facet description of any polytope realizing the braid arrangement~$\Fan_n$.
\end{corollary}


\section{Type cones of permutree fans}
\label{sec:typeConesPermutreeFans}

In this section, we provide a complete facet description of the type cone~$\ctypeCone(\Fan_\decoration)$ of the $\decoration$-permutree fan~$\Fan_\decoration$.
We first describe the rays of~$\Fan_\decoration$ in \cref{subsec:raysPermutreeFan}, then the pairs of exchangeable rays of~$\Fan_\decoration$ in \cref{subsec:exchangeablePairsPermutreeFan}, and finally the facets of the type cone~$\ctypeCone(\Fan_\decoration)$ in \cref{subsec:extremalExchangeablePairsPermutreeFan}.
We conclude by a description of kinematic permutreehedra for the permutree fans~$\Fan_\decoration$ whose type cone is simplicial in \cref{subsec:kinematicPermutreehedra}.


\subsection{Rays of permutree fans}
\label{subsec:raysPermutreeFan}

Specializing \cref{lem:raysQuotientFan} to the $\decoration$-permutree congruence~$\equiv_\decoration$, we obtain the following description of the rays of the $\decoration$-permutree fan~$\Fan_\decoration$ announced in \cref{subsubsec:geometryPermutrees}.

\begin{proposition}
\label{prop:raysPermutreeFan}
A ray~$\ray(I)$ is a ray in the $\decoration$-permutree fan~$\Fan_\decoration$ if and only if for all~$a < b < c$, if~$a, c \in I$ then~$b \notin \decoration^- \ssm I$,
and if~$a, c \notin I$ then~$b \notin \decoration^+ \cap I$.
\end{proposition}

\begin{example}
For the decorations of \cref{fig:permutreeLattices,fig:permutreeFan,fig:permutreehedron}, the rays of~$\Fan_{\noneCirc{}\noneCirc{}\downCirc{}\noneCirc{}}$ correspond to the subsets~$1, 2, 3, 4, 12, 13, 23, 34, 123, 134, 234$ while the rays of~$\Fan_{\noneCirc{}\upDownCirc{}\upCirc{}\noneCirc{}}$ correspond to the subsets~$1, 4, 12, 34, 123, 124, 234$.
\end{example}

\pagebreak
\begin{example}
Specializing \cref{prop:raysPermutreeFan}, we recover the following classical descriptions:
\begin{itemize}
\item when~$\decoration = \noneCirc^n$, the rays of the braid fan~$\Fan_{\noneCirc^n}$ are all proper subsets~$\varnothing \ne I \subsetneq [n]$,
\item when~$\decoration = \downCirc^n$, the rays of~$\Fan_{\downCirc^n}$ are all proper intervals~$[i,j]$ of~$[n]$, (equivalently, one can think of the interval~$[i,j]$ as corresponding to the internal diagonal $(i-1,j+1)$ of a polygon with vertices labeled~$0, \dots, n+1$),
\item when~$\decoration = \upDownCirc^n$, the rays of~$\Fan_{\noneCirc^n}$ are all proper initial intervals~$[1,i]$ or final intervals~$[i,n]$.
\end{itemize}
\end{example}

\begin{proof}[Proof of \cref{prop:raysPermutreeFan}]
We have~$a, c \in I$ and~$b \in \decoration^- \ssm I$ if and only if the shard~$\shard(i, j, \varnothing)$ is in~$\shards_I$ but not in~$\shards_\decoration$, where~$i \eqdef \max ({[a,b[} \cap I)$ and~$j \eqdef \min ({]b,c]} \cap I)$.
Similarly, $a, c \notin I$ and~${b \in \decoration^+ \cap I}$ if and only if the shard~$\shard(i, j, ]i,j[)$ is in~$\shards_I$ but not in~$\shards_\decoration$, where~$i \eqdef \max ({[a,b[} \ssm I)$ and~$j \eqdef \min ({]b,c]} \ssm I)$.
The statement thus follows from \cref{lem:raysQuotientFan}.

An alternative argument would be to observe directly that these conditions are necessary and sufficient to allow the construction of a $\decoration$-permutree with an edge whose cut is~$\edgecut{I}{[n] \ssm I}$.
\end{proof}

\begin{corollary}
\label{coro:numberRaysPermutreeFan}
The number~$\rho(\decoration)$ of rays of the $\decoration$-permutree fan~$\Fan_\decoration$ is
\[
\rho(\decoration) = n - 1 + \sum_{\substack{1 \le i < j \le n \\ \forall \, i < k < j, \, \decoration_k \ne \upDownCirc}} \!\!\!\!\! 2^{|\set{i < k < j}{\decoration_k = \noneCirc}|}.
\]
\end{corollary}

\begin{example}
For the decorations of \cref{fig:permutreeLattices,fig:permutreeFan,fig:permutreehedron}, $\rho(\noneCirc{}\noneCirc{}\downCirc{}\noneCirc{}) = 11$ and~$\rho(\noneCirc{}\upDownCirc{}\upCirc{}\noneCirc{}) = 7$.
\end{example}

\begin{example}
Specializing the formula of \cref{coro:numberRaysPermutreeFan}, we recover the following classical numbers:
\begin{itemize}
\item when~$\decoration = \noneCirc^n$, the braid fan~$\Fan_{\noneCirc^n}$ has~$2^n-2$ rays,
\item when~$\decoration = \downCirc^n$, the fan~$\Fan_{\downCirc^n}$ has $\binom{n+1}{2}-1$ rays (equalling the number of internal diagonals of the $(n+2)$-gon),
\item when~$\decoration = \upDownCirc^n$, the fan~$\Fan_{\upDownCirc^n}$ has~$2n-2$ rays.
\end{itemize}
\end{example}

\begin{proof}[Proof of \cref{coro:numberRaysPermutreeFan}]
To choose a ray~$\ray(I)$ of~$\Fan_\decoration$, we proceed as follows:
\begin{itemize}
\item We choose the last position~$i$ (resp.~first position~$j$) such that~$1, \dots, i$ (resp.~$j, \dots, n$) all belong to~$I$ or all belong to~$[n] \ssm I$. Note that~$1 \le i < j \le n$.
\item For any~$i < k < j$, since~$|\{i, i+1\} \cap I| = 1$ and~$|\{j-1, j\} \cap I| = 1$, the characterization of \cref{prop:raysPermutreeFan} imposes that~$k \in I$ if~$k \in \decoration^-$, and~$k \notin I$ if~$k \in \decoration^+$. This is impossible if~$\decoration_k = \upDownCirc$ (explaining the condition over the sum), and leaves two choices if~$\decoration_k = \noneCirc$ (explaining the power of~$2$).
\item If~$i+1 < j$, then the presence of~$i+1$ (resp.~$j-1$) in~$I$ requires the absence of~$i$ (resp.~$j$) in~$I$ and \viceversa, so there is no choice left.
\item If~$i+1 = j$, we have counted only one ray while both subsets~$[1,i]$ and~$[j,n]$ indeed correspond to rays of~$\Fan_\decoration$.
\qedhere
\end{itemize}
\end{proof}

\begin{corollary}
\label{coro:numberRaysPermutreeFanNoI}
If~$\decoration \in \{\downCirc, \upCirc, \upDownCirc\}^n$, we have~$\rho(\decoration) =  n - 1 + |\set{1 \le i < j \le n}{\forall \, i < k < j, \, \decoration_k \ne \upDownCirc}|$.
\end{corollary}


\subsection{Exchangeable rays of permutree fans}
\label{subsec:exchangeablePairsPermutreeFan}

An immediate corollary of \cref{prop:wallCrossingInequalitiesPermutreeFan} is that the linear dependence between the rays of two adjacent chambers~$C \eqdef \R_{\ge0}\rays$ and~$D \eqdef \R_{\ge0}\rays[S]$ of~$\Fan_\decoration$ with~$\rays \ssm \rays[S] = \{\ray\}$ and~$\rays[S] \ssm \rays = \{\ray[s]\}$ only depend on the rays~$\ray$ and~$\ray[s]$, not on the chambers~$C$ and~$D$.
This property is called \defn{unique exchange relation property} in~\cite{PadrolPaluPilaudPlamondon} and allows to describe the type cone by inequalities associated with exchangeable rays rather than with walls.

We therefore proceed in identifying the pairs of exchangeable rays of~$\Fan_\decoration$.
We consider two subsets~$I, J \in \cI_\decoration$, \ie proper subsets~$\varnothing \ne I, J \subsetneq [n]$ such that~$\ray(I)$ and~$\ray(J)$ are rays of the $\decoration$-permutree fan~$\Fan_\decoration$, as characterized in \cref{prop:raysPermutreeFan}.

\begin{proposition}
\label{prop:exchangeablePairsPermutreeFan}
The rays~$\ray(I)$ and~$\ray(J)$ are exchangeable in the $\decoration$-permutree fan~$\Fan_\decoration$ if and only if, up to swapping the roles of~$I$ and~$J$,
\begin{enumerate}[(i)]
\item $i \eqdef \max(I \ssm J) < \min(J \ssm I) \defeq j$,
\item $I \ssm J = \{i\}$ or $\decoration_i \ne \noneCirc$ \quad and \quad $J \ssm I = \{j\}$ or $\decoration_j \ne \noneCirc$,
\item ${]i,j[} \cap \decoration^- \subseteq I \cap J$ \quad and \quad ${]i,j[} \cap \decoration^+ \cap I \cap J = \varnothing$.
\end{enumerate}
\end{proposition}

\begin{example}
For the decorations of \cref{fig:permutreeLattices,fig:permutreeFan,fig:permutreehedron}, the pairs of exchangeable rays of~$\Fan_{\noneCirc{}\noneCirc{}\downCirc{}\noneCirc{}}$ correspond to the pairs of subsets~$\{1, 2\}$, $\{1, 3\}$, $\{1, 34\}$, $\{12, 13\}$, $\{12, 134\}$, $\{12, 23\}$, $\{12, 234\}$, $\{123, 134\}$, $\{123, 234\}$, $\{123, 4\}$, $\{13, 23\}$, $\{13, 34\}$, $\{13, 4\}$, $\{134, 234\}$, $\{2, 3\}$, $\{2, 34\}$, $\{23, 34\}$, $\{23, 4\}$, $\{3, 4\}$, while the pairs of exchangeable rays of~$\Fan_{\noneCirc{}\upDownCirc{}\upCirc{}\noneCirc{}}$ correspond to the pairs of subsets~$\{1, 234\}$, $\{12, 34\}$, $\{12, 4\}$, $\{123, 124\}$, $\{123, 4\}$, $\{124, 34\}$.
\end{example}

\begin{example}
\label{exm:exchangeablePairsPermutreeFan}
Specializing \cref{prop:exchangeablePairsPermutreeFan}, we recover that the pairs of exchangeable rays in~$\Fan_\decoration$ correspond to the pairs of proper subsets $\{I, J\}$ where
\begin{itemize}
\item when~$\decoration = \noneCirc^n$, we have~$I = K \cup \{i\}$ and~$J = K \cup \{j\}$ for~$1 \le i < j \le n$ and~$K \subseteq [n] \ssm \{i,j\}$,
\item when~$\decoration = \downCirc^n$, we have~$I = {[h,j[}$ and~$J = {]i,k]}$ for some~$1 \le h \le i < j \le k \le n$, (equivalently, the internal diagonals $(h-1,j)$ and~$(i,k+1)$ of the $(n+2)$-gon intersect),
\item when~$\decoration = \upDownCirc^n$, we have~$I = {[1,i]}$ and~$J = {]i,n]}$ for some~$1 \le i < n$.
\end{itemize}
\end{example}

\begin{proof}[Proof of \cref{prop:exchangeablePairsPermutreeFan}]
We first prove that the conditions of \cref{prop:exchangeablePairsPermutreeFan} are necessary for the rays~$\ray(I)$ and~$\ray(J)$ to be exchangeable in the $\decoration$-permutree fan~$\Fan_\decoration$.
We keep the notations of the proof of \cref{prop:wallCrossingInequalitiesPermutreeFan}.
Remember that we had~$I = \{i\} \cup D \cup \underline{L} \cup \overline{L}$ and~$J = \{j\} \cup D \cup \underline{R} \cup \overline{R}$.
Since~$\underline{L} \cup \overline{L} < i < j < \underline{R} \cup \overline{R}$, we obtain that~$i = \max(I \ssm J)$ and~$j = \min(J \ssm I)$ indeed satisfy~(i).
Moreover, $I \ssm J = \{i\} \cup \underline{L} \cup \overline{L}$ is restricted to~$\{i\}$ if~$\decoration_i = \noneCirc$ (because the subtrees~$\underline{L}$ and~$\overline{L}$ must then be empty), and similarly $J \ssm I = \{j\} \cup \underline{R} \cup \overline{R}$ is restricted to~$\{j\}$ if~$\decoration_j = \noneCirc$, wich shows~(ii).
Finally, if there is~$i < k < j$ such that~$k \in \decoration^- \ssm (I \cap J)$ (resp.~$k \in \decoration^+ \cap (I \cap J)$), then the edge~$i~\textbf{--}~j$ crosses the red wall below~$k$ (resp.~above~$k$), which shows~(iii).

Assume now that~$I$ and~$J$ satisfy the conditions of \cref{prop:raysPermutreeFan,prop:exchangeablePairsPermutreeFan}.
We construct two $\decoration$-permutrees~$T$ and~$S$, connected by the rotation of the edge~$i~\textbf{--}~j$ whose edge cut in~$T$ is~$I$ and in~$S$ is~$J$.
For this, we first pick an arbitrary permutree, that we denote by~$D$ (resp.~$U$, resp.~$L$, resp.~$R$), for the restriction of the decoration~$\decoration$ to the subset~$I \cap J$ (resp.~$[n] \ssm (I \cup J)$, resp.~$I \ssm J \ssm \{i\}$, resp.~$J \ssm I \ssm \{j\}$).
We then construct an oriented tree~$T$ on~$[n]$ starting with an edge~$i \to j$ and placing
\begin{itemize}
\item $D$ as the only (resp.~the right) descendant subtree of~$i$ if~$i \notin \decoration^-$ (resp.~if~$i \in \decoration^-$),
\item $U$ as the only (resp.~the left) ancestor subtree of~$j$ if~$j \notin \decoration^+$ (resp.~if~$j \in \decoration^+$),
\item $L$ as the left descendant (resp.~ancestor) subtree of~$i$ if~$i \in \decoration^-$ (resp.~if~$i \notin \decoration^-$),
\item $R$ as the right descendant (resp.~ancestor) subtree of~$j$ if~$j \in \decoration^-$ (resp.~if~$j \notin \decoration^-$).
\end{itemize}
Note that there is only one way to place these subtrees.
For instance, to place~$D$, we connect the leftmost upper blossom of~$D$ to the only (resp.~the right) lower blossom of~$i$ if~$i \notin \decoration^-$ (resp.~if~$i \in \decoration^-$).

We claim that the conditions of \cref{prop:raysPermutreeFan,prop:exchangeablePairsPermutreeFan} ensure that~$T$ is a $\decoration$-permutree.
Observe first that it indeed forms a tree since the permutree~$L$ (resp.~$R$) is empty if~$\decoration_i = \noneCirc$ (resp.~$\decoration_j = \noneCirc$) by \cref{prop:exchangeablePairsPermutreeFan}\,(ii).
Hence, we just need to show that no edge of~$T$ crosses a red wall below a node~$i \in \decoration^-$ or above a node~$i \in \decoration^+$.
Since all nodes of~$L$ (resp.~$R$) are smaller than~$i$ (resp.~$j$) by \cref{prop:exchangeablePairsPermutreeFan}\,(i), and there is no red wall above (resp.~below) the nodes of~$D$ (resp.~$U$) between~$i$ and~$j$ by \cref{prop:exchangeablePairsPermutreeFan}\,(iii), the edge~$i \to j$ crosses no red wall.
Consider now an edge~${\ell\;\textbf{--}\;\ell'}$ of~$L \cup \{i\}$ with~$\ell < \ell'$.
It cannot cross a red wall emanating from a node~$r$ of~$R \cup \{j\}$ since~$\ell' \le i < j \le r$, nor from a node~$u$ of~$U$ since otherwise we would have~$\ell < u < \ell'$ with~$\ell, \ell' \in I$ and~$u \in \decoration^- \ssm I$ contradicting \cref{prop:raysPermutreeFan}, nor from a node~$d$ of~$D$ since otherwise we would have~$\ell < d < \ell'$ with~$\ell, \ell' \notin J$ and~$d \in \decoration^+ \cap J$ contradicting \cref{prop:raysPermutreeFan}.
We prove similarly that no edge in~$D \cup \{i\}$, nor in $U \cup \{j\}$, nor in~$R \cup \{j\}$ crosses a red wall.
This closes the proof that~$T$ is a $\decoration$-permutree.

Finally, denote by~$S$ the $\decoration$-permutree obtained by the rotation of the edge~$i \to j$ in~$T$.
Observe that the construction is done so that the edge~$i \to j$ in~$T$ has cut~$\edgecut{I}{[n] \ssm I}$ while the edge~$j \to i$ in~$S$ has cut~$\edgecut{J}{[n] \ssm J}$.
It follows that the rays~$\ray(I)$ and~$\ray(J)$ are exchangeable in the adjacent chambers~$C(T)$ and~$C(S)$ of the $\decoration$-permutree fan~$\Fan_\decoration$.
\end{proof}

\pagebreak
\begin{corollary}
\label{coro:numberExchangeablePairsPermutreeFan}
The number~$\chi(\decoration)$ of pairs of exchangeable rays in the $\decoration$-permutree fan~$\Fan_\decoration$~is
\[
\chi(\decoration) = \sum_{\substack{1 \le i < j \le n \\ \forall \, i < k < j, \, \decoration_k \ne \upDownCirc}} \!\!\!\!\! \Omega(\decoration_1 \dots \decoration_{i-1})^{\decoration_i \ne \upDownCirc} \cdot 2^{|\set{i < k < j}{\decoration_k = \noneCirc}|} \cdot \Omega(\decoration_n \dots \decoration_{j+1})^{\decoration_j \ne \upDownCirc},
\]
where~$\Omega(\decoration_1 \dots \decoration_k)$ is defined inductively from~$\Omega(\varepsilon) = 1$ by
\[
\Omega(\decoration_1 \dots \decoration_k) = \begin{cases} 2 \cdot \Omega(\decoration_1 \dots \decoration_{k-1}) & \text{if } \decoration_k = \noneCirc, \\ 1 + \Omega(\decoration_1 \dots \decoration_{k-1}) & \text{if } \decoration_k \in \{\downCirc, \upCirc\}, \\ 2 & \text{if } \decoration_k = \upDownCirc. \end{cases}
\] 
\end{corollary}

\begin{example}
For the decorations of \cref{fig:permutreeLattices,fig:permutreeFan,fig:permutreehedron}, $\chi(\noneCirc{}\noneCirc{}\downCirc{}\noneCirc{}) = 19$ and~$\chi(\noneCirc{}\upDownCirc{}\upCirc{}\noneCirc{}) = 6$.
\end{example}

\begin{example}
Specializing the formula of \cref{coro:numberExchangeablePairsPermutreeFan}, we recover the following classical numbers:
\begin{itemize}
\item when~$\decoration = \noneCirc^n$, the braid fan~$\Fan_{\noneCirc^n}$ has~$2^{n-2} \binom{n}{2}$ pairs of exchangeable rays,
\item when~$\decoration = \downCirc^n$, the fan~$\Fan_{\downCirc^n}$ has~$\binom{n+2}{4}$ pairs of exchangeable rays (equalling the number of quadruples of vertices of the $(n+2)$-gon),
\item when~$\decoration = \upDownCirc^n$, the fan~$\Fan_{\upDownCirc^n}$ has~$n-1$ pairs of exchangeable rays.
\end{itemize}
\end{example}

\begin{proof}[Proof of \cref{coro:numberExchangeablePairsPermutreeFan}]
A pair of exchangeable rays in the $\decoration$-permutree fan~$\Fan_\decoration$ is a pair of proper subsets ${\varnothing \ne I, J \subsetneq [n]}$ satisfying the conditions of \cref{prop:raysPermutreeFan,prop:exchangeablePairsPermutreeFan}.
We choose such a pair of subsets as~follows:
\begin{itemize}
\item We first choose~$1 \le i < j \le n$ and will have~$i = \max(I \ssm J)$ and~$j = \min(J \ssm I)$ (to fulfill \cref{prop:exchangeablePairsPermutreeFan}\,(i)).
\item For all~$i < k < j$, we must have~$k \in I \cap J$ if $k \in \decoration^-$ and~$k \in [n] \ssm (I \cup J)$ if~$k \in \decoration^+$ (to fulfill \cref{prop:exchangeablePairsPermutreeFan}\,(iii)). This is impossible if~$\decoration_k = \upDownCirc$ (explaining the condition over the sum), and leaves two choices if~$\decoration_k = \noneCirc$ (explaining the power of~$2$).
\item For all~$k < i$, we must have~$k \in I$ if~$\decoration_i \in \decoration^+$, and~$k \notin J$ if~$\decoration_i \in \decoration^-$ (to fulfill \cref{prop:raysPermutreeFan}). Moreover, $k \in I \ssm J$ implies~$\decoration_i \ne \noneCirc$ (to fulfill \cref{prop:exchangeablePairsPermutreeFan}\,(ii)). Thus, $k$ must lie in
	\begin{itemize}
	\item $I \ssm J$ if~$\decoration_i = \upDownCirc$,
	\item $I \cap J$ or~$I \ssm J$ if~$\decoration_i = \upCirc$,
	\item $[n] \ssm (I \cup J)$  or~$I \ssm J$ if~$\decoration_i = \downCirc$,
	\item $I \cap J$ or~$[n] \ssm (I \cap J)$ if~$\decoration_i = \noneCirc$.
	\end{itemize}
Moreover, the choices in the last three cases are handled by the function~$\Omega$.
\item For~$j < k$, the argument is symmetric.
\qedhere
\end{itemize}
\end{proof}


\subsection{Facets of types cones of permutree fans}
\label{subsec:extremalExchangeablePairsPermutreeFan}

In view of the unique exchange property of the $\decoration$-permutree fan~$\Fan_\decoration$, each pair of exchangeable rays of~$\Fan_\decoration$ yields a wall-crossing inequality for the type cone~$\ctypeCone(\Fan_\decoration)$.
However, not all pairs of exchangeable rays yield facet-defining inequalities of~$\ctypeCone(\Fan_\decoration)$.
The characterization of the facets of~$\ctypeCone(\Fan_\decoration)$ is very similar to that of the exchangeable rays, only point (ii) slightly differs.

\begin{proposition}
\label{prop:extremalExchangeablePairsPermutreeFan}
The rays~$\ray(I)$ and~$\ray(J)$ define a facet of the type cone~$\ctypeCone(\Fan_\decoration)$ if and only if, up to swapping the roles of~$I$ and~$J$,
\begin{enumerate}[(i)]
\item $i \eqdef \max(I \ssm J) < \min(J \ssm I) \defeq j$,
\item $I \ssm J = \{i\}$ or $\decoration_i = \upDownCirc$ \quad and \quad $J \ssm I = \{j\}$ or $\decoration_j = \upDownCirc$,
\item ${]i,j[} \cap \decoration^- \subseteq I \cap J$ \quad and \quad ${]i,j[} \cap \decoration^+ \cap I \cap J = \varnothing$.
\end{enumerate}
\end{proposition}

\begin{example}
For the decorations of \cref{fig:permutreeLattices,fig:permutreeFan,fig:permutreehedron}, the facets of~$\ctypeCone(\Fan_{\noneCirc{}\noneCirc{}\downCirc{}\noneCirc{}})$ correspond to the pairs of subsets~$\{1, 2\}$, $\{1, 3\}$, $\{12, 13\}$, $\{12, 23\}$, $\{123, 134\}$, $\{123, 234\}$, $\{13, 23\}$, $\{13, 34\}$, $\{134, 234\}$, $\{2, 3\}$, $\{23, 34\}$, $\{3, 4\}$, while the facets of~$\ctypeCone(\Fan_{\noneCirc{}\upDownCirc{}\upCirc{}\noneCirc{}})$ correspond to the pairs of subsets~$\{1, 234\}$, $\{12, 4\}$, $\{123, 124\}$, $\{124, 34\}$.
\end{example}

\begin{example}
Specializing \cref{prop:extremalExchangeablePairsPermutreeFan}, we recover that all pairs of exchangeable rays of~$\Fan_\decoration$ described in \cref{exm:exchangeablePairsPermutreeFan} define facets of the type cone~$\ctypeCone(\Fan_\decoration)$ when~$\decoration = \noneCirc^n$ or~$\decoration = \upDownCirc^n$. In contrast, when~${\decoration = \downCirc^n}$, only the pairs of intervals~$\{{[i,j[}, {]i,j]}\}$ for some~$1 \le i < j \le n$ correspond to facets of the type cone~$\ctypeCone(\Fan_{\downCirc^n})$ (equivalently, the internal diagonals~$(i-1,j)$ and~$(i,j+1)$ of the $(n+2)$-gon that just differ by a shift).
\end{example}

\begin{remark}
\label{rem:extremalExchangeablePairsPermutreeFan}
Observe that combining the characterization of \cref{prop:raysPermutreeFan} with the condition of \cref{prop:extremalExchangeablePairsPermutreeFan}\,(ii) implies that 
\begin{itemize}
\item $[1,i[$ is included in $I \ssm J$ if~$\decoration_i = \upDownCirc$, in $[n] \ssm (I \cup J)$ if~$\decoration_i = \downCirc$, and in~$I \cap J$ if~$\decoration_i = \upCirc$, 
\item $]j,n]$ is included in $J \ssm I$ if~$\decoration_j = \upDownCirc$, in $[n] \ssm (I \cup J)$ if~$\decoration_j = \downCirc$, and in~$I \cap J$ if~$\decoration_j = \upCirc$.
\end{itemize}
In particular, $I \ssm J$ is either~$\{i\}$ or~$[1,i]$ (and not both except if~$i=1$), and $J \ssm I$ is either~$\{j\}$ or~$[j,n]$ (and not both except if~$j=n$).
The latter is however not equivalent to \cref{prop:extremalExchangeablePairsPermutreeFan}\,(ii), since for instance the subsets~$I = [1,2]$ and~$J = [3]$ do not define a facet~$\ctypeCone(\Fan_{\downCirc^3})$.
\end{remark}

\begin{proof}[Proof of \cref{prop:extremalExchangeablePairsPermutreeFan}]
We consider two exchangeable rays~$\ray(I)$ and~$\ray(J)$, so that~$I$ and~$J$ satisfy the conditions of \cref{prop:raysPermutreeFan,prop:exchangeablePairsPermutreeFan}. We will show that they satisfy the additional condition of \cref{prop:extremalExchangeablePairsPermutreeFan}\,(ii) if and only if the wall-crossing inequality corresponding to the exchange of~$\ray(I)$ and~$\ray(J)$ defines a facet of the type cone~$\ctypeCone(\Fan_\decoration)$.

\medskip
Assume first that~$I$ and~$J$ do not satisfy \cref{prop:extremalExchangeablePairsPermutreeFan}\,(ii).
Since they satisfy \cref{prop:exchangeablePairsPermutreeFan}\,(ii), we have~$|I \ssm J| > 1$ and~$\decoration_i \in \{\downCirc, \upCirc\}$, or~$|J \ssm I| > 1$ and~$\decoration_j \in \{\downCirc, \upCirc\}$.
Let us detail the case~$|I \ssm J| > 1$ and~$\decoration_i = \downCirc$, the other situations being symmetric.
As usual, we define
\[
i \eqdef \max(I \ssm J), \;\; j \eqdef \min(J \ssm I), \;\; D \eqdef I \cap J, \;\; U \eqdef [n] \ssm (I \cup J), \;\; L \eqdef I \ssm J \ssm \{i\}, \;\; R \eqdef J \ssm I \ssm \{j\}.
\]
For each of the subsets~$D$, $U$, $L$ and~$R$, we choose an arbitrary permutree, and we construct $\decoration$-permutrees~$T$ and~$S$ as in the proof of \cref{prop:exchangeablePairsPermutreeFan} and as illustrated in \cref{fig:rotationPermutreesFacetsTypeCone}, so that rotation from~$T$ to~$S$ exchanges~$\ray(I)$ to~$\ray(J)$.
Note that we voluntarily placed~$R$ at the same level as node~$j$, as it can be the right ancestor or descendant subtree of~$j$, depending on~$\decoration_j$.
We know that the wall-crossing inequality corresponding to this rotation is
\begin{equation}
\label{eq:wallCrossingInequality1}
h(I) + h(J) > h(I \cap J) + h(I \cup J).
\end{equation}

Consider now the $\decoration$-permutrees~$V$ and~$W$ of \cref{fig:rotationPermutreesFacetsTypeCone}.
In the tree~$V$ (resp.~$W$), the leftmost lower blossom of~$L$ is connected to~$j$ (resp.~$i$), and the rightmost upper blossom of~$L$ is connected to the blossom of~$U$ to which~$j$ was connected in~$T$.
Note that these are indeed $\decoration$-permutrees since~$\decoration_i = \downCirc$.
The rotation between~$V$ and~$W$ yields the wall-crossing inequality
\begin{equation}
\label{eq:wallCrossingInequality2}
h(I') + h(J) > h(I \cap J) + h(I' \cup J),
\end{equation}
where~$I' \eqdef \{i\} \cup D = \{i\} \cup (I \cap J)$.
Note that we could also have checked that~$I'$ and~$J$ satisfy the conditions of \cref{prop:raysPermutreeFan,prop:exchangeablePairsPermutreeFan}.

Consider now the $\decoration$-permutrees~$X$ and~$Y$ of \cref{fig:rotationPermutreesFacetsTypeCone}.
We have already seen that~$Y$ is indeed a $\decoration$-permutree because it is just equal to~$V$.
The same arguments show that~$X$ is also a $\decoration$-permutree.
Consider now the rotation of the edge joining~$L$ to~$j$.
If~$i$ and~$j$ are connected to the same node in~$L$, then this rotation gives the $\decoration$-permutree~$Y$.
Otherwise, this rotation moves a part of~$L$ in between~$j$ and~$U$ and leaves the remaining part of~$L$ in between~$i$ and~$j$.
Rotating again and again the edge between this remaining part of~$L$ and~$j$, we finally obtain the $\decoration$-permutree~$Y$.
More formally, we can consider the sequence of rotations between the $\decoration$-permutrees~$X_k$ and~$Y_k$ illustrated in \cref{fig:rotationPermutreesFacetsTypeCone}, where at each step we use~$X_{k+1} = Y_k$.
In these $\decoration$-permutrees, we have~$\underline{L}_k \sqcup \overline{L}_k = L$.
Moreover, $\underline{L}_{k+1}$ is obtained from~$\underline{L}_k$ by deleting the node connected to~$j$ in~$X_k$ and all its left ancestors and descendants.
Hence, starting with~$X_0 = X$, we will end with~$Y_p = Y$.
Summing the wall-crossing inequalities corresponding to the rotations between~$X_k$ and~$Y_k$, we thus obtain the inequality
\begin{equation}
\label{eq:wallCrossingInequality3}
h(I) + h(J') > h(I \cap J') + h(I \cup J),
\end{equation}
where~$J' \eqdef \{i\} \cup J$.

Finally, observe that~$I' = I \cap J'$ and~$J' = I' \cup J$.
We therefore obtain that the wall-crossing inequality \eqref{eq:wallCrossingInequality1} can be expressed as the sum of inequalities \eqref{eq:wallCrossingInequality2} and \eqref{eq:wallCrossingInequality3}, showing that~$I$ and~$J$ do not define a facet of the type cone~$\ctypeCone(\Fan_\decoration)$.

\begin{figure}[t]
	\capstart
	\centerline{\begin{tabular}{c@{\qquad}c@{\qquad}c@{\qquad}c}
	$T = $
	\begin{tikzpicture}[baseline=-.2cm, scale=.9]
		\node[circle, draw, inner sep=2pt] (i) at (-.5,-.5) {$i$};
		\node[circle, draw, inner sep=1pt] (j) at (.5,.5) {$j$};
		\node[inner sep=1pt] (U) at (0,1) {$U$};
		\node[inner sep=1pt] (D) at (0,-1) {$D$};
		\node[inner sep=1pt] (L) at (-1,-1) {$L$};
		\node[inner sep=1pt] (R) at (1.1,.5) {$R$};
		\draw (D) -- (i);
		\draw (L) -- (i);
		\draw[red, thick] (i) -- (j);
		\draw (j) -- (U);
		\draw (j) -- (R);
	\end{tikzpicture}
	&
	$V = $
	\begin{tikzpicture}[baseline=-.1cm, scale=.9]
		\node[circle, draw, inner sep=2pt] (i) at (-.5,-.5) {$i$};
		\node[circle, draw, inner sep=1pt] (j) at (.5,.5) {$j$};
		\node[inner sep=1pt] (U) at (0,1.5) {$U$};
		\node[inner sep=1pt] (D) at (0,-1) {$D$};
		\node[inner sep=1pt] (L) at (-1,1) {$L$};
		\node[inner sep=1pt] (l) at (-1,-1) {\phantom{$L$}};
		\node[inner sep=1pt] (R) at (1.1,.5) {$R$};
		\draw (D) -- (i);
		\draw (l) -- (i);
		\draw[red, thick] (i) -- (j);
		\draw (j) -- (L);
		\draw (j) -- (R);
		\draw (L) -- (U);
	\end{tikzpicture}
	&
	$X = $
	\begin{tikzpicture}[baseline=-.3cm, scale=.9]
		\node[circle, draw, inner sep=2pt] (i) at (-.5,-1) {$i$};
		\node[circle, draw, inner sep=1pt] (j) at (.5,.5) {$j$};
		\node[inner sep=1pt] (U) at (0,1) {$U$};
		\node[inner sep=1pt] (D) at (0,-1.5) {$D$};
		\node[inner sep=1pt] (L) at (-1,-.5) {$L$};
		\node[inner sep=1pt] (l) at (-1,-1.5) {\phantom{$L$}};
		\node[inner sep=1pt] (R) at (1.1,.5) {$R$};
		\draw (D) -- (i);
		\draw (l) -- (i);
		\draw (i) -- (L);
		\draw[red, thick] (L) -- (j);
		\draw (j) -- (U);
		\draw (j) -- (R);
	\end{tikzpicture}
	&
	$X_k = $
	\begin{tikzpicture}[baseline=-.1cm, scale=.9]
		\node[circle, draw, inner sep=2pt] (i) at (-.5,-1) {$i$};
		\node[circle, draw, inner sep=1pt] (j) at (.5,.5) {$j$};
		\node[inner sep=1pt] (U) at (0,1.5) {$U$};
		\node[inner sep=1pt] (D) at (0,-1.5) {$D$};
		\node[inner sep=1pt] (oLk) at (-1,1) {$\overline{L}_k$};
		\node[inner sep=1pt] (uLk) at (-1,-.5) {$\underline{L}_k$};
		\node[inner sep=1pt] (l) at (-1,-1.5) {\phantom{$L$}};
		\node[inner sep=1pt] (R) at (1.1,.5) {$R$};
		\draw (D) -- (i);
		\draw (l) -- (i);
		\draw (i) -- (uLk);
		\draw[red, thick] (uLk) -- (j);
		\draw (j) -- (oLk);
		\draw (j) -- (R);
		\draw (oLk) -- (U);
	\end{tikzpicture}
	\\
	\\
	$\Big\updownarrow$
	&
	$\Big\updownarrow$
	&
	$\Big\updownarrow$
	&
	$\Big\updownarrow$
	\\
	\\
	$S = $
	\begin{tikzpicture}[baseline=-.2cm, scale=.9]
		\node[circle, draw, inner sep=2pt] (i) at (-.5,.5) {$i$};
		\node[circle, draw, inner sep=1pt] (j) at (.5,-.5) {$j$};
		\node[inner sep=1pt] (U) at (0,1) {$U$};
		\node[inner sep=1pt] (D) at (0,-1) {$D$};
		\node[inner sep=1pt] (L) at (-1,0) {$L$};
		\node[inner sep=1pt] (R) at (1.1,-.5) {$R$};
		\draw (D) -- (j);
		\draw (j) -- (R);
		\draw[red, thick] (j) -- (i);
		\draw (i) -- (U);
		\draw (L) -- (i);
	\end{tikzpicture}
	&
	$W = $
	\begin{tikzpicture}[baseline=-.1cm, scale=.9]
		\node[circle, draw, inner sep=2pt] (i) at (-.5,.5) {$i$};
		\node[circle, draw, inner sep=1pt] (j) at (.5,-.5) {$j$};
		\node[inner sep=1pt] (U) at (0,1.5) {$U$};
		\node[inner sep=1pt] (D) at (0,-1) {$D$};
		\node[inner sep=1pt] (l) at (-1,0) {\phantom{$L$}};
		\node[inner sep=1pt] (L) at (-1,1) {$L$};
		\node[inner sep=1pt] (R) at (1.1,-.5) {$R$};
		\draw (D) -- (j);
		\draw (j) -- (R);
		\draw[red, thick] (j) -- (i);
		\draw (l) -- (i);
		\draw (i) -- (L);
		\draw (L) -- (U);
	\end{tikzpicture}
	&
	$Y = $
	\begin{tikzpicture}[baseline=-.1cm, scale=.9]
		\node[circle, draw, inner sep=2pt] (i) at (-.5,-.5) {$i$};
		\node[circle, draw, inner sep=1pt] (j) at (.5,.5) {$j$};
		\node[inner sep=1pt] (U) at (0,1.5) {$U$};
		\node[inner sep=1pt] (D) at (0,-1) {$D$};
		\node[inner sep=1pt] (L) at (-1,1) {$L$};
		\node[inner sep=1pt] (l) at (-1,-1) {\phantom{$L$}};
		\node[inner sep=1pt] (R) at (1.1,.5) {$R$};
		\draw (D) -- (i);
		\draw (l) -- (i);
		\draw (i) -- (j);
		\draw[red, thick] (j) -- (L);
		\draw (j) -- (R);
		\draw (L) -- (U);
	\end{tikzpicture}
	&
	$Y_k = $
	\begin{tikzpicture}[baseline=-.1cm, scale=.9]
		\node[circle, draw, inner sep=2pt] (i) at (-.5,-1) {$i$};
		\node[circle, draw, inner sep=1pt] (j) at (.5,.5) {$j$};
		\node[inner sep=1pt] (U) at (0,1.5) {$U$};
		\node[inner sep=1pt] (D) at (0,-1.5) {$D$};
		\node[inner sep=1pt] (oLk) at (-1,1) {$\overline{L}_{k+1}$};
		\node[inner sep=1pt] (uLk) at (-1,-.5) {$\underline{L}_{k+1}$};
		\node[inner sep=1pt] (l) at (-1,-1.5) {\phantom{$L$}};
		\node[inner sep=1pt] (R) at (1.1,.5) {$R$};
		\draw (D) -- (i);
		\draw (l) -- (i);
		\draw (i) -- (uLk);
		\draw (uLk) -- (j);
		\draw[red, thick] (j) -- (oLk);
		\draw (j) -- (R);
		\draw (oLk) -- (U);
	\end{tikzpicture}
	\\
\end{tabular}}
	\caption{Some rotations in $\decoration$-permutrees. The third column is a combination of rotations of the form given in the fourth column.}
	\label{fig:rotationPermutreesFacetsTypeCone}
\end{figure}
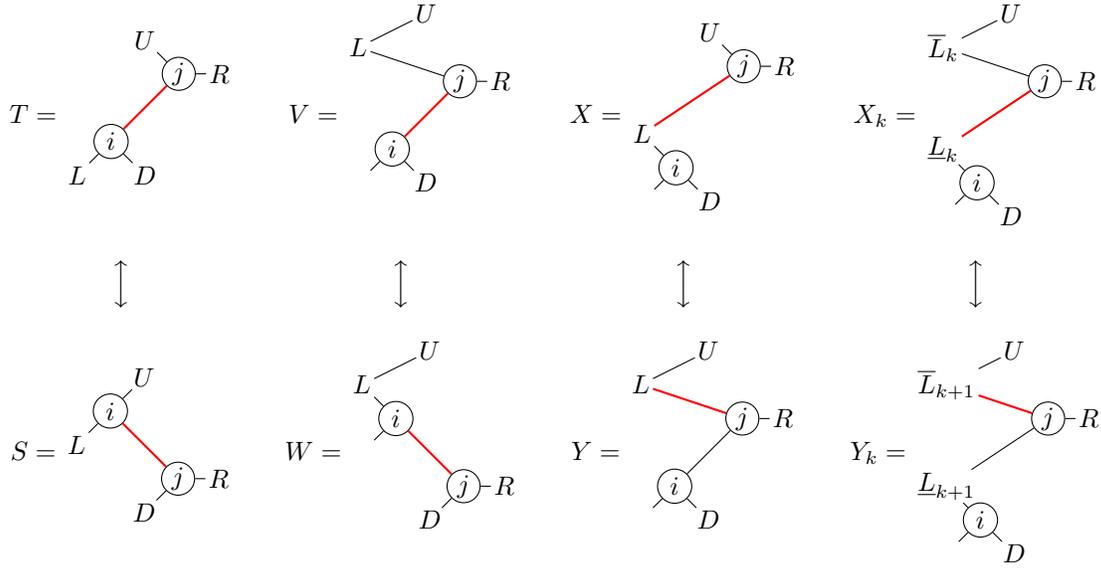

\medskip
Conversely, assume that~$I$ and~$J$ satisfy the conditions of \cref{prop:extremalExchangeablePairsPermutreeFan}.
To prove that the wall-crossing inequality associated with~$\{I,J\}$ indeed defines a facet of the type cone~$\ctypeCone(\Fan_\decoration)$, we exhibit a point~$\b{p}$ that satisfies the wall-crossing inequality associated with any pair~$\{K,L\}$ fulfilling the conditions of \cref{prop:extremalExchangeablePairsPermutreeFan}, except for the one associated with the pair~$\{I,J\}$.

For this, we need some notations and definitions.
We denote by~$\wp(A) \eqdef \{X \subseteq A\}$ the set of subsets of a set~$A$, and define~$\nabla(A,B) \eqdef \wp(A \cup B) \ssm \big( \wp(A) \cup \wp(B) \big)$ for two sets~$A, B$.
For a pair~$\{A,B\}$ with~$A \ssm B \ne \varnothing \ne B \ssm A$, observe that
\begin{itemize}
\item $A \cup B$ is the inclusion-maximal element of~$\nabla(A,B)$,
\item $A$ and $B$ are the inclusion-maximal subset of~$A \cup B$ not in~$\nabla(A,B)$,
\item the pairs in~$\nabla(A,B)$ are precisely the pairs~$\{a,b\}$ for~$a \in A \ssm B$ and~$b \in B \ssm A$.
\end{itemize}
This implies that for two pairs~$\{A,B\}$ and~$\{C,D\}$ with~$A \ssm B \ne \varnothing \ne B \ssm A$ and~$C \ssm D \ne \varnothing \ne D \ssm C$,
\begin{itemize}
\item if~$\nabla(A,B) = \nabla(C,D)$, then~$\{A,B\} = \{C,D\}$,
\item if~$\nabla(A,B) \subseteq \nabla(C,D)$, then~$A \cup B \subseteq C \cup D$, and up to reversing the roles of~$C$ and~$D$, $A \ssm B \subseteq C \ssm D$ and~$B \ssm A \subseteq D \ssm C$.
\end{itemize}

We denote by~$(\b{g}_M)_{M \in \cI}$ the standard basis of the space~$\R^\cI$ indexed by the rays~$\cI$ of~$\Fan_\decoration$.
Let~$\b{n}(I,J) \eqdef \b{g}_I + \b{g}_J - \b{g}_{I \cap J} - \b{g}_{I \cup J}$ denote the normal vector of the wall-crossing inequality associated with our pair~$\{I,J\}$ of subsets.
We are thus looking for a point~$\b{p} \in \R^\cI$ such that~$\dotprod{\b{p}}{\b{n}(I,J)} < 0$ while~$\dotprod{\b{p}}{\b{n}(K,L)} \ge 0$ for all pairs~$\{K,L\}$ that fulfill the conditions of \cref{prop:extremalExchangeablePairsPermutreeFan} and are distinct from~$\{I,J\}$.
To find such a point, we define three vectors~$\b{x}, \b{y}, \b{z} \in \R^\cI$ given by
\[
\b{x} \eqdef - \sum_{M \in \cI} |\wp(M) \ssm \nabla(I,J)| \, \b{g}_M,
\qquad
\b{y} \eqdef - \sum_{M \in \cI} |\wp(M) \cap \nabla(I,J)| \, \b{g}_M,
\qquad\text{and}\qquad
\b{z} \eqdef -\b{n}(I,J),
\]
and we consider the point~$\b{p} \in \R^\cI$ defined by
\[
\b{p} \eqdef \lambda \, \b{x} + \mu \, \b{y} + \b{z},
\]
where~$\lambda$ and~$\mu$ are arbitrary scalars such that~$\lambda > |\dotprod{\b{z}}{\b{n}(K,L)}|$ for any pair~$\{K,L\}$ of subsets and~$0 < \mu \, |\nabla(I,J)| < \dotprod{\b{z}}{\b{z}}$.
We will prove that the point~$\b{p}$ satisfies the desired inequalities.

For this, consider a pair~$\{K,L\}$ that fulfills the conditions of \cref{prop:extremalExchangeablePairsPermutreeFan}.
Observe that
\begin{align*}
& \dotprod{\b{x}}{\b{n}(K,L)} = \b{x}_K + \b{x}_L - \b{x}_{K \cap L} - \b{x}_{K \cup L} \\ & \qquad = - |\wp(K) \ssm \nabla(I,J)| - |\wp(L) \ssm \nabla(I,J)| + |\wp(K \cup L) \ssm \nabla(I,J)| + |\wp(K \cap L) \ssm \nabla(I,J)| \\ & \qquad = |\nabla(K,L) \ssm \nabla(I,J)|,
\end{align*}
by inclusion-exclusion principle.
Similarly
\[
\dotprod{\b{y}}{\b{n}(K,L)} = |\nabla(K,L) \cap \nabla(I,J)|.
\]
Finally,
\[
\dotprod{\b{z}}{\b{n}(K,L)} = - \dotprod{\b{n}(I,J)}{\b{n}(K,L)} = -\dotprod{\b{g}_I \! + \! \b{g}_J \! - \! \b{g}_{I \cup J} \! - \! \b{g}_{I \cap J}}{\b{g}_K \! + \! \b{g}_L \! - \! \b{g}_{K \cup L} \! - \! \b{g}_{K \cap L}}
\]
is a signed sum of Kronecker deltas~$\delta_{X,Y}$ where~$X$ ranges over~$\{I, J, I \cup J, I \cap J\}$ and~$Y$ ranges over~$\{K, L, K \cup L, K \cap L\}$.
We leave to the end of the proof the claim that these two sets are disjoint when~$\nabla(K,L) \subsetneq \nabla(I,J)$, so that~$\dotprod{\b{z}}{\b{n}(K,L)} = 0$.

We therefore obtain that
\[
\dotprod{\b{p}}{\b{n}(I,J)} = \mu \, |\nabla(I,J)| - \dotprod{\b{z}}{\b{z}} < 0
\]
since~$\mu \, |\nabla(I,J)| < \dotprod{\b{z}}{\b{z}}$, while
\[
\dotprod{\b{p}}{\b{n}(K,L)} = \lambda \, |\nabla(K,L) \ssm \nabla(I,J)| + \mu \, |\nabla(K,L) \cap \nabla(I,J)| - \dotprod{\b{z}}{\b{n}(K,L)} \ge 0
\]
for any other pair~$\{K,L\}$ that fulfills the conditions of \cref{prop:extremalExchangeablePairsPermutreeFan}.
Indeed, 
\begin{itemize}
\item if~$\nabla(K,L) \not\subseteq \nabla(I,J)$, then~$\lambda \, |\nabla(K,L) \ssm \nabla(I,J)| - \dotprod{\b{z}}{\b{n}(K,L)} \ge \lambda - \dotprod{\b{z}}{\b{n}(K,L)} > 0$ and~$\mu \, |\nabla(K,L) \cap \nabla(I,J)| \ge 0$,
\item if~$\nabla(K,L) \subsetneq \nabla(I,J)$, then~$\lambda \, |\nabla(K,L) \ssm \nabla(I,J)| = 0$ while~$\mu \, |\nabla(K,L) \cap \nabla(I,J)| > \mu$ and we claimed that~$\dotprod{\b{z}}{\b{n}(K,L)} = 0$.
\end{itemize}

We conclude by proving our claim that if~$\nabla(K,L) \subsetneq \nabla(I,J)$, then the sets~$\{I, J, I \cup J, I \cap J\}$ and~$\{K, L, K \cup L, K \cap L\}$ are disjoint.
Up to reversing the roles of~$I$ and~$J$ (resp.~$K$ and~$L$), we assume that~$i \eqdef \max(I \ssm J) < \min(J \ssm I) \defeq j$ (resp.~$k \eqdef \max(K \ssm L)< \min(L \ssm K) \defeq \ell$).
As observed earlier, the inclusion~$\nabla(K,L) \subsetneq \nabla(I,J)$ implies the inclusions~$K \ssm L \subseteq I \ssm J$ and~$L \ssm K \subseteq J \ssm I$, so that~$k \le i < j \le \ell$.
It turns out that the conditions of \cref{prop:extremalExchangeablePairsPermutreeFan} actually imply that $K \ssm L = I \ssm J$ and~$L \ssm K = J \ssm I$, so that~$k = i$ and~$j = \ell$.
Indeed,
\begin{itemize}
\item if~$I \ssm J = \{i\}$, then~$K \ssm L = \{i\}$ (since~$\varnothing \ne K \ssm L \subseteq I \ssm J$),
\item if~$\decoration_i = \upDownCirc$, then~$K \ssm L = I \ssm J = [1,i]$ (as otherwise~$i \in {]k, \ell[} \cap \decoration^- \cap \decoration^+$).
\end{itemize}
(The argument is symmetric for the equality~$L \ssm K = J \ssm I$.)
Observe now that
\begin{itemize}
\item if~$I \cap J = K \cap L$, then~$I = (I \cap J) \cup (I \ssm J) = (K \cap L) \cup (K \ssm L) = K$ and similarly~$J = L$,
\item if~$I \cup J = K \cup L$, then~$I = (I \cup J) \ssm (J \ssm I) = (K \cup L) \ssm (L \ssm K) = K$ and similarly~$J = L$,
\item if~$I = K$, then~$J = (J \ssm I) \cup (I \ssm (I \ssm J)) = (L \ssm K) \cup (K \ssm (K \ssm L)) = L$,
\item similarly, if~$J = L$, then~$I = K$.
\end{itemize}
In these four cases, we obtained that~$I = K$ and~$J = L$, which contradicts our assumption that~$\nabla(K,L) \ne \nabla(I,J)$.
Finally, observe that 
\begin{itemize}
\item $I \notin \{L, K \cup L, K \cap L\}$ since~$i \in I \ssm L$ and~$j \in L \ssm I$. Similarly, $J \notin \{K, K \cup L, K \cap L\}$, $K \notin \{J, I \cup J, I \cap J\}$ and $L \notin \{J, I \cup J, I \cap J\}$.
\item $I \cup J \ne K \cap L$ since~$i \in I \ssm L$. Similarly, $I \cap J \ne K \cup L$.
\end{itemize}
We have thus checked that the sets~$\{I, J, I \cup J, I \cap J\}$ and~$\{K, L, K \cup L, K \cap L\}$ are disjoint, which ends the proof.
\end{proof}

The proof of the following consequence of \cref{prop:extremalExchangeablePairsPermutreeFan} is almost identical to that of \cref{coro:numberExchangeablePairsPermutreeFan}.

\begin{corollary}
\label{coro:numberExtremalExchangeablePairsPermutreeFan}
The number~$\phi(\decoration)$ of facets of the type cone~$\ctypeCone(\Fan_\decoration)$ of the $\decoration$-permutree fan~$\Fan_\decoration$~is
\[
\phi(\decoration) = \sum_{\substack{1 \le i < j \le n \\ \forall \, i < k < j, \, \decoration_k \ne \upDownCirc}} \!\!\!\!\! \Omega(\decoration_1 \dots \decoration_{i-1})^{\decoration_i = \noneCirc} \cdot 2^{|\set{i < k < j}{\decoration_k = \noneCirc}|} \cdot \Omega(\decoration_n \dots \decoration_{j+1})^{\decoration_j = \noneCirc},
\]
where~$\Omega(\decoration_1 \dots \decoration_k)$ is defined inductively from~$\Omega(\varepsilon) = 1$ by
\[
\Omega(\decoration_1 \dots \decoration_k) = \begin{cases} 2 \cdot \Omega(\decoration_1 \dots \decoration_{k-1}) & \text{if } \decoration_k = \noneCirc, \\ 1 + \Omega(\decoration_1 \dots \decoration_{k-1}) & \text{if } \decoration_k \in \{\downCirc, \upCirc\}, \\ 2 & \text{if } \decoration_k = \upDownCirc. \end{cases}
\] 
\end{corollary}

\begin{example}
For the decorations of \cref{fig:permutreeLattices,fig:permutreeFan,fig:permutreehedron}, $\phi(\noneCirc{}\noneCirc{}\downCirc{}\noneCirc{}) = 12$ and~$\phi(\noneCirc{}\upDownCirc{}\upCirc{}\noneCirc{}) = 4$.
\end{example}

\begin{example}
Specializing the formula of \cref{coro:numberExtremalExchangeablePairsPermutreeFan}, we recover the following classical numbers:
\begin{itemize}
\item when~$\decoration = \noneCirc^n$, the type cone~$\ctypeCone(\Fan_{\noneCirc^n})$ has~$2^{n-2} \binom{n}{2}$ facets,
\item when~$\decoration = \downCirc^n$, the type cone~$\ctypeCone(\Fan_{\downCirc^n})$ has~$\binom{n}{2}$ facets (equalling the number of squares of the form~$(i-1,i,j,j+1)$ in the $(n+2)$-gon),
\item when~$\decoration = \upDownCirc^n$, the type cone~$\ctypeCone(\Fan_{\upDownCirc^n})$ has~$n-1$ facets.
\end{itemize}
\end{example}

\begin{corollary}
\label{coro:numberExtremalExchangeablePairsPermutreeFanNoI}
If~$\decoration \in \{\downCirc, \upCirc, \upDownCirc\}^n$, we have~$\phi(\decoration) = |\set{1 \le i < j \le n}{\forall \, i < k < j, \, \decoration_k \ne \upDownCirc}|$.
\end{corollary}

\begin{corollary}
The type cone~$\ctypeCone(\Fan_\decoration)$ is simplicial if and only if~$\decoration_k \ne \noneCirc{}$ for any~$k \in {]1,n[}$.
\end{corollary}

\begin{proof}
The type cone~$\ctypeCone(\Fan_\decoration)$ is simplicial if and only if the number~$\rho(\decoration)$ of rays of~$\Fan_\decoration$ and the number~$\phi(\decoration)$ of facets of~$\ctypeCone(\Fan_\decoration)$ satisfy the equality $\rho(\decoration) = \phi(\decoration) + n - 1$.
Considering the formulas of \cref{coro:numberRaysPermutreeFan} for~$\rho(\decoration)$ and of \cref{coro:numberExtremalExchangeablePairsPermutreeFan} for~$\phi(\decoration)$, we observe that $\rho(\decoration) = \phi(\decoration) + n - 1$ if and only if~$\psi(i,j) \eqdef \Omega(\decoration_1 \dots \decoration_{i-1})^{\decoration_i = \noneCirc} \cdot \Omega(\decoration_n \dots \decoration_{j+1})^{\decoration_j = \noneCirc}$ is always equal to~$1$ in the formula for~$\phi(\decoration)$.
It is clearly the case if~$\decoration_k \ne \noneCirc{}$ for any~$k \in {]1,n[}$.
Conversely, if~$\decoration_k = \noneCirc{}$ for some~$k \in {]1,n[}$, then we have~$\psi(k-1,k) \ge 2$ (and also~$\psi(k,k+1) \ge 2$), so that the equality does not hold.
\end{proof}


\subsection{Kinematic permutreehedra}
\label{subsec:kinematicPermutreehedra}

\newcommandx{\Fd}{\mathfrak{F}}
\newcommandx{\Rd}{\mathfrak{R}}
\newcommandx{\ldij}[3][1=i, 2=j]{\smash{p^{#3}_{#1,#2}}}
\newcommandx{\rdij}[3][1=i, 2=j]{\smash{q^{#3}_{#1,#2}}}
Applying \cref{prop:simplicialTypeCone}, we obtain the following realizations of the $\decoration$-permutree fans in the kinematic space~\cite{ArkaniHamedBaiHeYan} when~$\decoration \in \{\downCirc, \upCirc, \upDownCirc\}^n$.
To simplify our statement, we assume that~${\decoration_1 = \decoration_n = \upDownCirc}$ (this assumption does not lose generality as the decorations~$\decoration_1$ and~$\decoration_n$ are irrelevant in all constructions).
Consider the sets
\[
\Fd \eqdef \set{1 \le i < j \le n}{\forall \, i < k < j, \, \decoration_k \ne \upDownCirc}
\quad\text{and}\quad
\Rd \eqdef \{0, 1\} \times [n]^2 \times \{0, 1\}
\]
and define~$\ldij{\varepsilon}$ and~$\rdij{\varepsilon}$ for~$(i,j) \in \Fd$ and~$\varepsilon \in \{+, -\}$ by
\[
\ldij{\varepsilon} \! \eqdef \! \begin{cases} \min \big( \{j\} \cup \big( {]i,j[} \cap \decoration^\varepsilon \big) \big) - 1  & \text{if } i \in \decoration^\varepsilon, \\ i-1 & \text{if } i \notin \decoration^\varepsilon, \end{cases}
\qquad
\rdij{\varepsilon} \! \eqdef \! \begin{cases} \max \big( \{i\} \cup \big( {]i,j[} \cap \decoration^\varepsilon \big) \big) + 1  & \text{if } j \in \decoration^\varepsilon, \\ j+1 & \text{if } j \notin \decoration^\varepsilon. \end{cases}
\]
Using these notations, we obtain the following realizations of the $\decoration$-permutree fan.

\begin{corollary}
\label{coro:kinematicSpace}
Let~$\decoration \in \{\downCirc, \upCirc, \upDownCirc\}^n$ with~${\decoration_1 = \decoration_n = \upDownCirc}$, and the notations introduced above.
Then, for any~$\b{u} \in \R_{>0}^{\Fd}$, the polytope~$Q_\decoration(\b{u})$ defined by
\[
\set{\b{z} \in \R_{\ge0}^{\Rd}\!}{\begin{array}{@{}l@{}} \b{z}_{(\ell, \, p, \, q, \, r)} = 0 \text{ if } (p, q) \notin \Fd, \quad \b{z}_{(\ell, \, p, \, q, \, r)} = \b{z}_{(\ell', \, p, \, q, \, r')} \text{ if } p + 1 \ne q, \quad\text{and for all } (i,j) \in \Fd, \\ \b{z}_{(1, \, \ldij{+}, \, \rdij{-}, \, 0)} \! + \b{z}_{(0, \, \ldij{-}, \, \rdij{+}, \, 1)} \! - \b{z}_{(i \notin \decoration^-, \, \ldij[i][j+1]{-}, \, \rdij[i-1][j]{-}, \, j \notin \decoration^-)} \! - \b{z}_{(i \in \decoration^+, \, \ldij[i][j+1]{+}, \, \rdij[i-1][j]{+}, \, j \in \decoration^+)} \! = \b{u}_{(i, \, j)} \end{array}}
\]
is a $\decoration$-permutreehedron, whose normal fan is the $\decoration$-permutree fan~$\Fan_\decoration$.
Moreover, the polytopes~$Q_\decoration(\b{u})$ for~$\b{u} \in \R_{>0}^{\Fd}$ describe all polytopal realizations of the $\decoration$-permutree fan~$\Fan_\decoration$.
\end{corollary}

\begin{example}
Specializing the construction of \cref{coro:kinematicSpace}, we obtain:
\begin{itemize}
\item when~$\decoration = \upDownCirc\downCirc^{n-2}\upDownCirc$, we have
\[
\qquad
\ldij{-} = \begin{cases} j-1 & \text{if } i = 1, \\ i-1 & \text{if } i \ne 1, \end{cases}
\qquad
\ldij{+} = i,
\qquad
\rdij{-} = \begin{cases} i+1 & \text{if } j = n, \\ j+1 & \text{if } j \ne n, \end{cases}
\qquad\text{and}\qquad
\rdij{+} = j,
\]
so that the polytope~$Q_\decoration(\b{u})$ is equivalent to the \defn{kinematic associahedron} of~\cite{ArkaniHamedBaiHeYan}:
\[
\qquad
\set{\b{y} \in \R^{\binom{[0,n+1]}{2}}\!}{\begin{array}{@{}l@{}} \b{y} \ge 0, \qquad \b{y}_{(i, \, j)} = 0 \text{ if } i + 1 = j, \qquad \b{y}_{(0,n+1)} = 0, \qquad\text{and} \\ \b{y}_{(i, \, j+1)} \! + \b{y}_{(i-1, \, j)} \! - \b{y}_{(i-1, \, j+1)} \! - \b{y}_{(i, \, j)} \! = \b{u}_{(i, \, j)} \text{ for all } (i,j) \in \binom{[n]}{2} \end{array}}.
\]
(The map is given by~$\b{y}_{(0, \, j)} = \b{z}_{(1, \, j-1, \, j, \, 0)}$, $\b{y}_{(i, \, n+1)} = \b{z}_{(0, \, i, \, i+1, \, 1)}$ and~$\b{y}_{(i, \, j)} = \b{z}_{(\ell, \, i, \, j, \, r)}$ for any~$\ell, r \in \{0,1\}$.)
\item when~$\decoration = \upDownCirc^n$, we have
\[
\ldij{-} = \ldij{+} = i
\qquad\text{and}\qquad
\rdij{-} = \rdij{+} = j,
\]
so that the polytope~$Q_\decoration(\b{u})$ is equivalent to the following \defn{kinematic cube}:
\[
\qquad
\set{\b{y} \in \R^{\{0,1\} \times [n-1]}\!}{\b{y} \ge 0 \quad\text{and}\quad \b{y}_{(0, \, i)} \! + \b{y}_{(1, \, i)} \! = \b{u}_{(i, \, i+1)} \text{ for all } i \in [n-1]}
\]
(The map is given by~$\b{y}_{(0, \, i)} = \b{z}_{(0, \, i, \, i+1, \, 1)}$ and~$\b{y}_{(1, \, i)} = \b{z}_{(1, \, i, \, i+1, \, 0)}$.)
\end{itemize}
\end{example}

\begin{proof}[Proof of \cref{coro:kinematicSpace}]
According to \cref{coro:numberRaysPermutreeFanNoI,coro:numberExtremalExchangeablePairsPermutreeFanNoI}, we can parametrize 
\begin{enumerate}
\item the rays of the fan~$\Fan_\decoration$ by~$\Rd$ where we identify~$(0, p, q, 0)$, $(1, p, q, 0)$, $(0, p, q, 1)$ and $(1, p, q, 1)$ when~$p+1 \ne q$, and ignore~$(0, p, q, 0)$, $(1, p, q, 0)$, $(0, p, q, 1)$ and $(1, p, q, 1)$ for~$(p,q) \notin \Fd$.
\item the facets of its type cone~$\ctypeCone(\Fan_\decoration)$ by~$\Fd$.
\end{enumerate}
To be more precise, with any~$(\ell, p, q, r) \in \Rd$, we associate the subset~$R_{(\ell, p, q, x)}$ of~$[n]$ defined as follows:
\begin{itemize}
\item if~$(p, q) \notin \Fd$ then~$R_{(0, p, q, 0)} \eqdef R_{(1, p, q, 0)} \eqdef R_{(0, p, q, 1)} \eqdef \varnothing$ while~$R_{(1, p, q, 1)} \eqdef [n]$,
\item if~$p + 1 = q$, then~$R_{(0, p, q, 0)} \eqdef \varnothing$, $R_{(1, p, q, 0)} \eqdef [1,p]$, $R_{(0, p, q, 1)} \eqdef [q,n]$ and $R_{(1,p,q,1)} \eqdef [n]$,
\item if~$(p,q) \in \Fd$ and~$p+1 \ne q$, then independently of the values of~$\ell$ and~$r$, the set~$R_{(\ell, p, q, r)}$ is the unique proper subset~$\varnothing \ne R \subsetneq [n]$ which fulfills the conditions of \cref{prop:raysPermutreeFan} and with the property that~$i$ is the last position (resp.~$j$ is the first position) such that~$1, \dots, i$ (resp.~$j, \dots, n$) all belong to~$R$ or all belong to~$[n] \ssm R$, (see the proof of \cref{coro:numberRaysPermutreeFan} for the uniqueness~of~$R$).
\end{itemize}
Now with any~$(i,j) \in \Fd$, we associate the unique pair~$(I_{i,j}, J_{i,j})$ of proper subsets satisfying the conditions of \cref{prop:raysPermutreeFan,prop:extremalExchangeablePairsPermutreeFan} and such that~${\max(I_{i,j} \ssm J_{i,j}) = i}$ and ${\min(J_{i,j} \ssm I_{i,j}) = j}$ (see \cref{rem:extremalExchangeablePairsPermutreeFan} to argue that such a pair is unique).
We claim that these sets are given by
\[
I_{i,j} = R_{(1, \, \ldij{+}, \, \rdij{-}, \, 0)}
\qquad\text{and}\qquad
J_{i,j} = R_{(0, \, \ldij{-}, \, \rdij{+}, \, 1)},
\]
and their intersection and union  are given by
\[
I_{i,j} \cap J_{i,j} = R_{(i \notin \decoration^-, \, \ldij[i][j+1]{-}, \, \rdij[i-1][j]{-}, \, j \notin \decoration^-)}
\qquad\text{and}\qquad
I_{i,j} \cup J_{i,j} = R_{(i \in \decoration^+, \, \ldij[i][j+1]{+}, \, \rdij[i-1][j]{+}, \, j \in \decoration^+)},
\]
which implies the result by \cref{prop:wallCrossingInequalitiesPermutreeFan,prop:simplicialTypeCone}.

\bigskip
To show this claim, let us first prove that~$I_{i,j} = R_{(1, \, \ldij{+}, \rdij{-}, 0)}$.
Recall from \cref{prop:extremalExchangeablePairsPermutreeFan}\,(iii) that ${]i,j[} \cap \decoration^- \subseteq I_{i,j}$ and ${]i,j[} \cap \decoration^+ \subseteq [n] \ssm I_{i,j}$.
Moreover, we obtain by \cref{rem:extremalExchangeablePairsPermutreeFan} that
\begin{itemize}
\item if~$i \in \decoration^+$, then~${[1,i]} \subseteq I_{i,j}$ and the last position~$p$ for which $1, \dots, p$ all belong to~$I_{i,j}$ is the position just before the first element in~${]i,j[} \cap \decoration^+$, or~$j$ if there is no such element,
\item otherwise,~${[1,i[} \subseteq [n] \ssm I_{i,j}$ while~$i \in I_{i,j}$, so that the last position~$p$ for which $1, \dots, p$ all belong to~$[n] \ssm I_{i,j}$ is~$i-1$.
\end{itemize}
This shows that~$\ldij{+}$ indeed gives the last position~$p$ for which~$1, \dots, p$ all belong to~$I_{i,j}$ or all belong to~$[n] \ssm I_{i,j}$.
A symmetric argument shows that~$\rdij{-}$ gives the first position~$q$ for which~$q, \dots, n$ all belong to~$I_{i,j}$ or all belong to~$[n] \ssm I_{i,j}$.
This ensures that~$1 \le \ldij{+} < \rdij{-} \le n$.
Moreover, we have~$\decoration_k \ne \upDownCirc$ for~$\ldij{+} < k < \rdij{-}$, since~$i-1 \le \ldij{+}$ with equality only when~$i \notin \decoration^+$ and~$\rdij{-} \le j+1$ with equality only when~$j \notin \decoration^-$.
Thus, $(\ldij{+}, \rdij{-}) \in \Fd$.

Observe now that if~$i \notin \decoration^+$, then~$\ldij{+} = i-1$ while~$\rdij{-} \ge i+1$ so that~$\ldij{+} + 1 < \rdij{-}$.
We therefore obtain that
\begin{itemize}
\item if~$\ldij{+} + 1 = \rdij{-}$, then~$i \in \decoration^+$ so that~$[1,i] \subseteq I_{i,j}$ and thus~$I_{i,j} = [1,\ldij{+}] = R_{(1, \, \ldij{+}, \, \rdij{-}, \, 0)}$,
\item if~$\ldij{+} + 1 < \rdij{-}$, then $I_{i,j} = R_{(1, \, \ldij{+}, \, \rdij{-}, \, 0)}$ as they are both fully determined by~$\ldij{+}$ and~$\rdij{-}$.
\end{itemize}
This concludes our proof of~$I_{i,j} \! = \! R_{(1, \, \ldij{+}, \, \rdij{-}, \, 0)}$.
A symmetric argument shows~$J_{i,j} \! = \! R_{(0, \, \ldij{-}, \, \rdij{+}, \, 1)}$.

\bigskip
For the intersection, note first that if~$I_{i,j} \cap J_{i,j} = \varnothing$, then~$i,j \in \decoration^-$ by \cref{rem:extremalExchangeablePairsPermutreeFan} while ${{]i,j[} \subseteq \decoration^+}$ by \cref{prop:extremalExchangeablePairsPermutreeFan}\,(iii). Therefore, $\ldij[i][j+1]{-} = j-1$ and~$\rdij[i-1][j]{-} = i+1$, which implies that~$(\ldij[i][j+1]{-} \not< \rdij[i-1][j]{-}) \notin \Fd$ or~$\ldij[i][j+1]{-} + 1 = \rdij[i-1][j]{-}$.
Since~$i,j \in \decoration^-$, this yields in both situations that~$R_{(i \notin \decoration^-, \, \ldij[i][j+1]{-}, \, \rdij[i-1][j]{-}, \, j \notin \decoration^-)} = \varnothing = I_{i,j} \cap J_{i,j}$.

Assume now that~$I_{i,j} \cap J_{i,j} \ne \varnothing$.
We then obtain by \cref{rem:extremalExchangeablePairsPermutreeFan} that
\begin{itemize}
\item if~$i \in \decoration^+$, then~${[1,i]} \subseteq [n] \ssm (I_{i,j} \cap J_{i,j})$ and the last position~$p$ for which $1, \dots, p$ all belong to~$[n] \ssm (I_{i,j} \cap J_{i,j})$ is the position just before the first element in~${]i,j[} \cap \decoration^-$, or~$j+1$ if there is no such element,
\item otherwise,~${[1,i[} \subseteq I_{i,j} \cap J_{i,j}$ while~$i \notin I_{i,j} \cap J_{i,j}$, so that the last position~$p$ for which $1, \dots, p$ all belong to~$I_{i,j} \cap J_{i,j}$ is~$i-1$.
\end{itemize}
This shows that~$\ldij[i][j+1]{-}$ indeed gives the last position~$p$ for which~$1, \dots, p$ belongs all to~$I_{i,j} \cap J_{i,j}$ or all to~$[n] \ssm (I_{i,j} \cap J_{i,j})$.
A symmetric argument shows that~$\rdij[i-1][j]{-}$ gives the first position~$q$ for which~$q, \dots, n$ belongs all to~$I_{i,j} \cap J_{i,j}$ or all to~$[n] \ssm (I_{i,j} \cap J_{i,j})$.
As above, this ensures that~$(\ldij[i][j+1]{-}, \rdij[i-1][j]{-}) \in \Fd$.

We now distinguish three cases:
\begin{itemize}
\item if~$\ldij[i][j+1]{-} + 1 = \rdij[i-1][j]{-}$ and~$[1,i[ \subseteq I_{i,j} \cap J_{i,j}$, then~$i \in \decoration^+$ and~$j \in \decoration^-$ by \cref{rem:extremalExchangeablePairsPermutreeFan}, so that we get~$I_{i,j} \cap J_{i,j} = [1, \ldij[i][j+1]{-}] = R_{(0, \, \ldij[i][j+1]{-}, \, \rdij[i-1][j]{-}, \, 1)} = R_{(i \notin \decoration^-, \, \ldij[i][j+1]{-}, \, \rdij[i-1][j]{-}, \, j \notin \decoration^-)}$.
\item if~$\ldij[i][j+1]{-} + 1 = \rdij[i-1][j]{-}$ and~$[1,i] \subseteq [n] \ssm (I_{i,j} \cap J_{i,j})$, then~$i \in \decoration^-$ and~$j \in \decoration^+$ by \cref{rem:extremalExchangeablePairsPermutreeFan}, so that we get~$I_{i,j} \cap J_{i,j} = [\rdij[i-1][j]{-}, n] = R_{(1, \, \ldij[i][j+1]{-}, \, \rdij[i-1][j]{-}, \, 0)} = R_{(i \notin \decoration^-, \, \ldij[i][j+1]{-}, \, \rdij[i-1][j]{-}, \, j \notin \decoration^-)}$.
\item if~$\ldij[i][j+1]{-} + 1 < \rdij[i-1][j]{-}$, then $I_{i,j} \cap J_{i,j} = R_{(i \notin \decoration^-, \, \ldij[i][j+1]{-}, \, \rdij[i-1][j]{-}, \, j \notin \decoration^-)}$ as they are both fully determined by~$\ldij[i][j+1]{-}$ and~$\rdij[i-1][j]{-}$.
\end{itemize}
This concludes the proof for the intersection~$I_{i,j} \cap J_{i,j} = R_{(i \notin \decoration^-, \, \ldij[i][j+1]{-}, \, \rdij[i-1][j]{-}, \, j \notin \decoration^-)}$.
The proof for the union $I_{i,j} \cup J_{i,j} = R_{(i \in \decoration^+, \, \ldij[i][j+1]{+}, \, \rdij[i-1][j]{+}, \, j \in \decoration^+)}$ is identical.
\end{proof}


\section*{Acknowledgments}

We are grateful to J.-C. Novelli for multiple discussions and suggestions, to N.~Reading for bibliographic inputs, to A.~Padrol for comments on a preliminary version, and to an anonymous referee for various suggestions on the presentation of this paper.


\bibliographystyle{alpha}
\bibliography{quotientopesPermutreehedraRemovahedra}

\begin{thebibliography}{AHBHY18}

\bibitem[ACEP20]{ArdilaCastilloEurPostnikov}
Federico Ardila, Federico Castillo, Christopher Eur, and Alexander Postnikov.
\newblock Coxeter submodular functions and deformations of {C}oxeter
  permutahedra.
\newblock {\em Adv. Math.}, 365:107039, 36, 2020.

\bibitem[AHBHY18]{ArkaniHamedBaiHeYan}
Nima Arkani-Hamed, Yuntao Bai, Song He, and Gongwang Yan.
\newblock Scattering forms and the positive geometry of kinematics, color and
  the worldsheet.
\newblock {\em J. High Energy Phys.}, (5):096, front matter+75, 2018.

\bibitem[CD06]{CarrDevadoss}
Michael~P. Carr and Satyan~L. Devadoss.
\newblock Coxeter complexes and graph-associahedra.
\newblock {\em Topology Appl.}, 153(12):2155--2168, 2006.

\bibitem[CFZ02]{ChapotonFominZelevinsky}
Fr{\'e}d{\'e}ric Chapoton, Sergey Fomin, and Andrei Zelevinsky.
\newblock Polytopal realizations of generalized associahedra.
\newblock {\em Canad. Math. Bull.}, 45(4):537--566, 2002.

\bibitem[CP17]{ChatelPilaud}
Gr\'egory Chatel and Vincent Pilaud.
\newblock {C}ambrian {H}opf {Algebras}.
\newblock {\em Adv. Math.}, 311:598--633, 2017.

\bibitem[Dev09]{Devadoss}
Satyan~L. Devadoss.
\newblock A realization of graph associahedra.
\newblock {\em Discrete Math.}, 309(1):271--276, 2009.

\bibitem[DRS10]{DeLoeraRambauSantos}
Jesus~A. {De Loera}, J\"org Rambau, and Francisco Santos.
\newblock {\em Triangulations: Structures for Algorithms and Applications},
  volume~25 of {\em Algorithms and {C}omputation in Mathematics}.
\newblock Springer Verlag, 2010.

\bibitem[GKZ08]{GelfandKapranovZelevinsky}
Israel Gelfand, Mikhail Kapranov, and Andrei Zelevinsky.
\newblock {\em Discriminants, resultants and multidimensional determinants}.
\newblock Modern Birkh\"auser Classics. Birkh\"auser Boston Inc., Boston, MA,
  2008.
\newblock Reprint of the 1994 edition.

\bibitem[HL07]{HohlwegLange}
Christophe Hohlweg and Carsten Lange.
\newblock Realizations of the associahedron and cyclohedron.
\newblock {\em Discrete Comput.~Geom.}, 37(4):517--543, 2007.

\bibitem[HNT05]{HivertNovelliThibon-algebraBinarySearchTrees}
Florent Hivert, Jean-Christophe Novelli, and Jean-Yves Thibon.
\newblock The algebra of binary search trees.
\newblock {\em Theoret. Comput. Sci.}, 339(1):129--165, 2005.

\bibitem[KT97]{KrobThibon-NCSF4}
Daniel Krob and Jean-Yves Thibon.
\newblock Noncommutative symmetric functions. {IV}. {Q}uantum linear groups and
  {H}ecke algebras at {$q=0$}.
\newblock {\em J. Algebraic Combin.}, 6(4):339--376, 1997.

\bibitem[Lod04]{Loday}
Jean-Louis Loday.
\newblock Realization of the {S}tasheff polytope.
\newblock {\em Arch.~Math.~(Basel)}, 83(3):267--278, 2004.

\bibitem[LP18]{LangePilaud}
Carsten Lange and Vincent Pilaud.
\newblock Associahedra via spines.
\newblock {\em Combinatorica}, 38(2):443--486, 2018.

\bibitem[LR98]{LodayRonco}
Jean-Louis Loday and Mar{\'{\i}}a~O. Ronco.
\newblock Hopf algebra of the planar binary trees.
\newblock {\em Adv. Math.}, 139(2):293--309, 1998.

\bibitem[McM73]{McMullen-typeCone}
Peter McMullen.
\newblock Representations of polytopes and polyhedral sets.
\newblock {\em Geometriae Dedicata}, 2:83--99, 1973.

\bibitem[Mey74]{Meyer}
Walter Meyer.
\newblock Indecomposable polytopes.
\newblock {\em Trans. Amer. Math. Soc.}, 190:77--86, 1974.

\bibitem[Nov00]{Novelli-hypoplactic}
Jean-Christophe Novelli.
\newblock On the hypoplactic monoid.
\newblock {\em Discrete Math.}, 217(1-3):315--336, 2000.
\newblock Formal power series and algebraic combinatorics (Vienna, 1997).

\bibitem[Pil17]{Pilaud-removahedra}
Vincent Pilaud.
\newblock Which nestohedra are removahedra?
\newblock {\em Rev. Colombiana Mat.}, 51(1):21--42, 2017.

\bibitem[Pil19]{Pilaud-arcDiagramAlgebra}
Vincent Pilaud.
\newblock Hopf algebras on decorated noncrossing arc diagrams.
\newblock {\em J. Combin. Theory Ser. A}, 161:486--507, 2019.

\bibitem[Pos09]{Postnikov}
Alexander Postnikov.
\newblock Permutohedra, associahedra, and beyond.
\newblock {\em Int. Math. Res. Not. IMRN}, (6):1026--1106, 2009.

\bibitem[PP18]{PilaudPons-permutrees}
Vincent Pilaud and Viviane Pons.
\newblock Permutrees.
\newblock {\em Algebraic Combinatorics}, 1(2):173--224, 2018.

\bibitem[PPPP19]{PadrolPaluPilaudPlamondon}
Arnau Padrol, Yann Palu, Vincent Pilaud, and Pierre-Guy Plamondon.
\newblock Associahedra for finite type cluster algebras and minimal relations
  between $\b{g}$-vectors.
\newblock Preprint,
  \href{http://arxiv.org/abs/1906.06861}{\texttt{arXiv:1906.06861}}, 2019.

\bibitem[PPR20]{PadrolPilaudRitter}
Arnau Padrol, Vincent Pilaud, and Julian Ritter.
\newblock Quotientopes via minkowski sums of shard polytopes.
\newblock In preparation, 2020.

\bibitem[PRW08]{PostnikovReinerWilliams}
Alexander Postnikov, Victor Reiner, and Lauren~K. Williams.
\newblock Faces of generalized permutohedra.
\newblock {\em Doc.~Math.}, 13:207--273, 2008.

\bibitem[PS19]{PilaudSantos-quotientopes}
Vincent Pilaud and Francisco Santos.
\newblock Quotientopes.
\newblock {\em Bull. Lond. Math. Soc.}, 51(3):406--420, 2019.

\bibitem[Rea03]{Reading-posetRegions}
Nathan Reading.
\newblock Lattice and order properties of the poset of regions in a hyperplane
  arrangement.
\newblock {\em Algebra Universalis}, 50(2):179--205, 2003.

\bibitem[Rea04]{Reading-latticeCongruences}
Nathan Reading.
\newblock Lattice congruences of the weak order.
\newblock {\em Order}, 21(4):315--344, 2004.

\bibitem[Rea05]{Reading-HopfAlgebras}
Nathan Reading.
\newblock Lattice congruences, fans and {H}opf algebras.
\newblock {\em J. Combin. Theory Ser. A}, 110(2):237--273, 2005.

\bibitem[Rea06]{Reading-CambrianLattices}
Nathan Reading.
\newblock Cambrian lattices.
\newblock {\em Adv.~Math.}, 205(2):313--353, 2006.

\bibitem[Rea11]{Reading-shardIntersectionOrder}
Nathan Reading.
\newblock Noncrossing partitions and the shard intersection order.
\newblock {\em J. Algebraic Combin.}, 33(4):483--530, 2011.

\bibitem[Rea15]{Reading-arcDiagrams}
Nathan Reading.
\newblock Noncrossing arc diagrams and canonical join representations.
\newblock {\em SIAM J. Discrete Math.}, 29(2):736--750, 2015.

\bibitem[Rea16a]{Reading-FiniteCoxeterGroupsChapter}
Nathan Reading.
\newblock Finite {C}oxeter groups and the weak order.
\newblock In {\em Lattice theory: special topics and applications. {V}ol. 2},
  pages 489--561. Birkh\"auser/Springer, Cham, 2016.

\bibitem[Rea16b]{Reading-PosetRegionsChapter}
Nathan Reading.
\newblock Lattice theory of the poset of regions.
\newblock In {\em Lattice theory: special topics and applications. {V}ol. 2},
  pages 399--487. Birkh\"auser/Springer, Cham, 2016.

\bibitem[RS09]{ReadingSpeyer}
Nathan Reading and David~E. Speyer.
\newblock Cambrian fans.
\newblock {\em J.~Eur.~Math.~Soc.}, 11(2):407--447, 2009.

\bibitem[SS93]{ShniderSternberg}
Steve Shnider and Shlomo Sternberg.
\newblock {\em Quantum groups: From coalgebras to {D}rinfeld algebras}.
\newblock Series in Mathematical Physics. International Press, Cambridge, MA,
  1993.

\bibitem[Tam51]{Tamari}
Dov Tamari.
\newblock {\em Monoides pr\'eordonn\'es et cha\^ines de Malcev}.
\newblock PhD thesis, Universit\'e Paris Sorbonne, 1951.

\bibitem[Zie98]{Ziegler-polytopes}
G{\"u}nter~M. Ziegler.
\newblock {\em Lectures on Polytopes}, volume 152 of {\em Graduate texts in
  Mathematics}.
\newblock Springer-Verlag, New York, 1998.

\end{thebibliography}
\label{sec:biblio}

\end{document}